\documentclass{siamart1116}
\usepackage{amsfonts,amsmath,amssymb}
\usepackage{mathrsfs,mathtools}
\usepackage{enumerate}
\usepackage{hyperref}
\usepackage{esint}
\usepackage{graphicx}
\usepackage{bm}
\usepackage{commath}
\usepackage{esint}
\usepackage{algorithm}
\usepackage{algpseudocode}
\DeclareMathAlphabet{\mathpzc}{OT1}{pzc}{m}{it}

\def\dif{\:\mathrm{d}}

\def\C{\mathcal{C}}

\def\0{\emptyset}

\def\tr{\mathrm{tr}_\Omega\:}

\def\dif{\:\mathrm{d}}

\def\dif{\:\mathrm{d}}

\theoremstyle{definition}
\newtheorem{thm}{Theorem}
\newtheorem{rem}[theorem]{Remark}
\newtheorem{exa}[thm]{Example}

\numberwithin{equation}{section}

\usepackage[textsize=small]{todonotes}
\title{Sobolev spaces with non-Muckenhoupt weights, fractional elliptic operators, and applications
\thanks{HA has been partially supported by  NSF grant DMS-1521590. 
 CNR has been supported via the framework of {\sc Matheon} by the Einstein Foundation Berlin within the ECMath project SE15/SE19 and acknowledges the
support of the DFG through the DFG-SPP 1962: Priority Programme ``Non-smooth and Complementarity-based Distributed Parameter Systems: Simulation and Hierarchical Optimization''  within Project 11}
}

\author{Harbir Antil\thanks{Department of Mathematical Sciences, George Mason University, Fairfax, VA 22030, USA. \texttt{hantil@gmu.edu}},
\and
Carlos N. Rautenberg\thanks{Weierstrass Institute
for Applied Analysis and Stochastics, Mohrenstr. 39, 10117 Berlin, Germany. \texttt{rautenberg@wias-berlin.de}\and Department of Mathematics, Humboldt-University of Berlin, Unter den Linden 6, 10099 Berlin, Germany. \texttt{carlos.rautenberg@math.hu-berlin.de}}
}

\begin{document}

\maketitle

\begin{abstract}
We propose a new variational model 
in weighted Sobolev spaces with non-standard weights and  applications to image processing. We show that these weights 
are, in general, not of Muckenhoupt type and therefore the classical analysis tools may not apply. For special cases of the weights,
the resulting variational problem is known to be equivalent to the fractional Poisson 
problem. The trace space for the weighted Sobolev space is identified to be embedded in a weighted $L^2$ 
space.  We propose a finite element scheme to solve the Euler-Lagrange 
equations, and for the image denoising application we propose an algorithm to identify the 
unknown weights. The approach is illustrated on several test problems
and it yields better results when compared to the existing total variation techniques. 
\end{abstract}



\section{Introduction}

In this paper we consider weighted Sobolev spaces where the weight $w$ is \emph{not} necessarely of Muckenhoupt type, and an associated variational model with concrete applications to image processing that shows advantageous features of such weights. The particular weights $w$ that we consider in this paper are closely related to fractional Sobolev spaces of differentiability order $s\in (0,1)$ and the fractional spectral Laplacian $(-\Delta)^s$, and are further motivated by considering a spatially dependent order $x\mapsto s(x)$. 

Weighted Sobolev spaces have been a topic of intensive study for around 60 years with a variety of focuses; we refer the readers to the monographs \cite{MR926688,MR802206,BOTuresson_2000a} and references within, and further \cite{MR3185296,MR2491902,MR1669639,MR775568,ANekvinda_1993a} for an introduction to the subject. In particular, a source of interest for such spaces is that they represent the correct solution space for several degenerate elliptic partial differential equations (\cite{MR926688}) and problems in potential theory (\cite{BOTuresson_2000a}). Recently, the topic has received a new impulse mainly associated with the extension result of Caffarelli-Silvestre; \cite{LCaffarelli_LSilvestre_2007a} (Stinga-Torrea; \cite{PRStinga_JLTorrea_2010a}) for fractional elliptic partial differential operators. In this setting, for example, the solution to $(-\Delta)^su=f$ in the fractional Sobolev space $H^s(\Omega)$  endowed with zero boundary conditions, can be equivalently obtained as (the restriction to $\Omega$ of) the solution of a PDE with a non-fractional elliptic operator in a weighted Sobolev space with weight $w(x,y)=y^{1-2s}$ in the extended domain $\mathcal{C}=\{(x,y)\in \Omega\times (0,+\infty)\}$.

The type of weights $w$ that have been more thoroughly studied correspond to two main classes: 1) weights belonging to some Muckenhoupt class; see \cite{MR2491902,MR1669639} and 2) composition of functions (mainly power functions) with the distance function to a particular set; see \cite{MR802206, ANekvinda_1993a}. In these cases, several important questions like density of smooth functions and characterization of traces have been, at least partially, answered; see \cite{MR3185296,MR2491902,MR802206,ANekvinda_1993a}. In this paper we consider weights of the type $w(x,y)=y^{1-2s(x)}$ on $\mathcal{C}=\{(x,y)\in \Omega\times (0,+\infty)\}$, which are, in general, neither of 1) or 2) type.

The variational problem of interest in this paper is closely related to fractional elliptic operators. The spike of research interest on problems with this type of operators is mainly due to the results from Caffarelli and collaborators; see \cite{PRStinga_JLTorrea_2010a,LCaffarelli_LSilvestre_2007a,LACaffarelli_PRStinga_2016a}. Following the renowned  results in \cite{LCaffarelli_LSilvestre_2007a}, a large number of contributions on modeling, numerical methods, applications, regularity results, different boundary conditions, and control problems, among others have been considered. 

Recently, for image denoising a model was 
proposed in \cite{HAntil_SBartels_2017a}. The variational problem can be formulated in our setting as 
\begin{equation}\label{eq:s_fix_intro}
 \min_{u} \frac{1}{2}\|(-\Delta)^{\frac{s}{2}} v\|_{L^2(\Omega)}^2 
     + \frac{\zeta}{2} \|v-f\|^2_{L^2(\Omega)}  ,
\end{equation}
where $(-\Delta)^s$ denotes the fractional power of the Laplacian with zero Neumann boundary
conditions and  $s \in (0,1)$; see \cite{LACaffarelli_PRStinga_2016a,antil2017fractional}. Additionally, $\zeta>0$ is a given constant, and $f=u_{\mathrm{true}}+\mathrm{noise}$ where $u_{\mathrm{true}}$ is the desired target of the optimization procedure. The choice of $s$ has a direct consequence on the global regularity of the solution to problem \eqref{eq:s_fix_intro}. However, it is desirable, from the reconstruction point of view, that the regularity of the solution to \eqref{eq:s_fix_intro} is low in places in $\Omega$ where edges or discontinuities are present in $u_{\mathrm{true}}$, and that is high in places where  $u_{\mathrm{true}}$ is smooth or contains homogeneous features. Hence, it is of interest to consider \eqref{eq:s_fix_intro} where $s:\Omega\to [0,1]$ is not a constant.  The first main roadblock in letting $s$ to be spatially dependent is the fact that 
there is no obvious way to define $(-\Delta)^{s(x)}$. In fact we will not attempt to 
do so in this paper, we leave this as an open question. 
Instead of \eqref{eq:s_fix_intro} we consider the following variational problem:
\begin{equation}\label{eq:s_intro}
\min_{u} \frac{1}{2}\int_{\C}y^{1-2s(x)} \left(\theta |u|^2 + |\nabla u|^2 \right) \dif x\dif y+ \frac{\mu}{2}\int_{\Omega} s(x)^2 |u(x,0) - f| ^2  \dif x,
\end{equation}
where $\mathcal{C}=\{(x,y)\in \Omega\times (0,+\infty)\}$, $0<\theta \ll 1$,   $\mu>0$ is a given constant, and $u(x,0)$ stands for the trace of $u$ at $\Omega\times\{0\}$. { For the solution $u$ of this optimization problem, we expect that $u(\cdot,0)\simeq u_{\mathrm{true}}$}. This problem is the focus of the present paper and the full motivational link between the problems \eqref{eq:s_fix_intro} and \eqref{eq:s_intro}, is given in the following section.

We next summarize the main contributions of this paper and discuss how the paper is organized in what follows. 

In section \ref{motivation}, we provide a rigorous motivation for the study of problem \eqref{eq:s_intro} by showing that  \eqref{eq:s_intro} represents a generalization of \eqref{eq:s_fix_intro}, after using the Caffarelli-Silvestre extension, in the case of $s$ non-constant. For a constant $s \in (0,1)$ the definition of fractional Laplacian in terms of the
Caffarelli-Silvestre or the Stinga-Torrea extension is by now well-known. However such a result remains open 
when $s(x) \in [0,1]$ for $x\in \Omega$. Towards this end we emphasize that 
we have formulated our problem \eqref{eq:s_intro} on an unbounded domain $\C=\Omega\times (0,+\infty)$, and not directly over $\Omega$.

In order to handle varying 
$s$ in space, for almost all (f.a.a) $x \in \Omega$ we assume $s(x) \in [0,1]$, and consider weighted Sobolev with weights $w(x,y)=y^{1-2s(x)}$ on $\mathcal{C}$. This class of Sobolev spaces is studied in section \ref{s:rest} and  we establish also two fundamental results. The first one is that since we allow 
$s(x) =0$, for some $x\in\Omega$, the weight $y^{1-2s(x)}$ is not, in general, in the Muckenhoupt class $A_2$
(cf. Proposition~\ref{eq:NoA2}). Due to the lack of this $A_2$ property we cannot use some of the 
existing machinery on weighted Sobolev spaces. In particular, the density of smooth functions is not always given in weighted spaces in the absence of the properties inherited by the Muckenhoupt class.  The second important result of section \ref{s:rest} characterizes the trace space for this class of weighted Sobolev spaces, this is given in Theorem \ref{thm:trace}, and it is established that the trace space embeds continuously on the $s$-weighted $L^2$ space. In addition, we  give examples of functions with discontinuous traces in Example \ref{exampledisc}. 

In section \ref{s:prob}, we establish properties for the variational problem of interest \eqref{eq:s_intro}, and a procedure for the selection of the function $s$. We finalize the paper with section \ref{s:num_method}, where we develop a discretization scheme, and provide numerical tests that positively compared with other methods  for image reconstruction. In particular, in section~\ref{s:disc} we introduce a finite element method for \eqref{eq:ext_var_intro} on a bounded domain $\C_\tau := \C \times (0,\tau)$. 
Since our targeted application is image denosing so we first state an 
Algorithm to determine $s$ in Algorithm~\ref{Algorithm}. 
We apply this approach to several prototypical examples in section~\ref{s:num_ex}.
In all cases we obtain better results when compared to the Total Variation (TV) approach.

\subsection{Motivation and some notation}\label{motivation}

Our point of departure for motivation to study the proposed variational model (and associated function space properties) is Figure~\ref{f:ex_2_diffview_intro} 
\begin{figure}[h!]
\centering
\hspace{-1.8cm}
\includegraphics[width=0.4\textwidth]{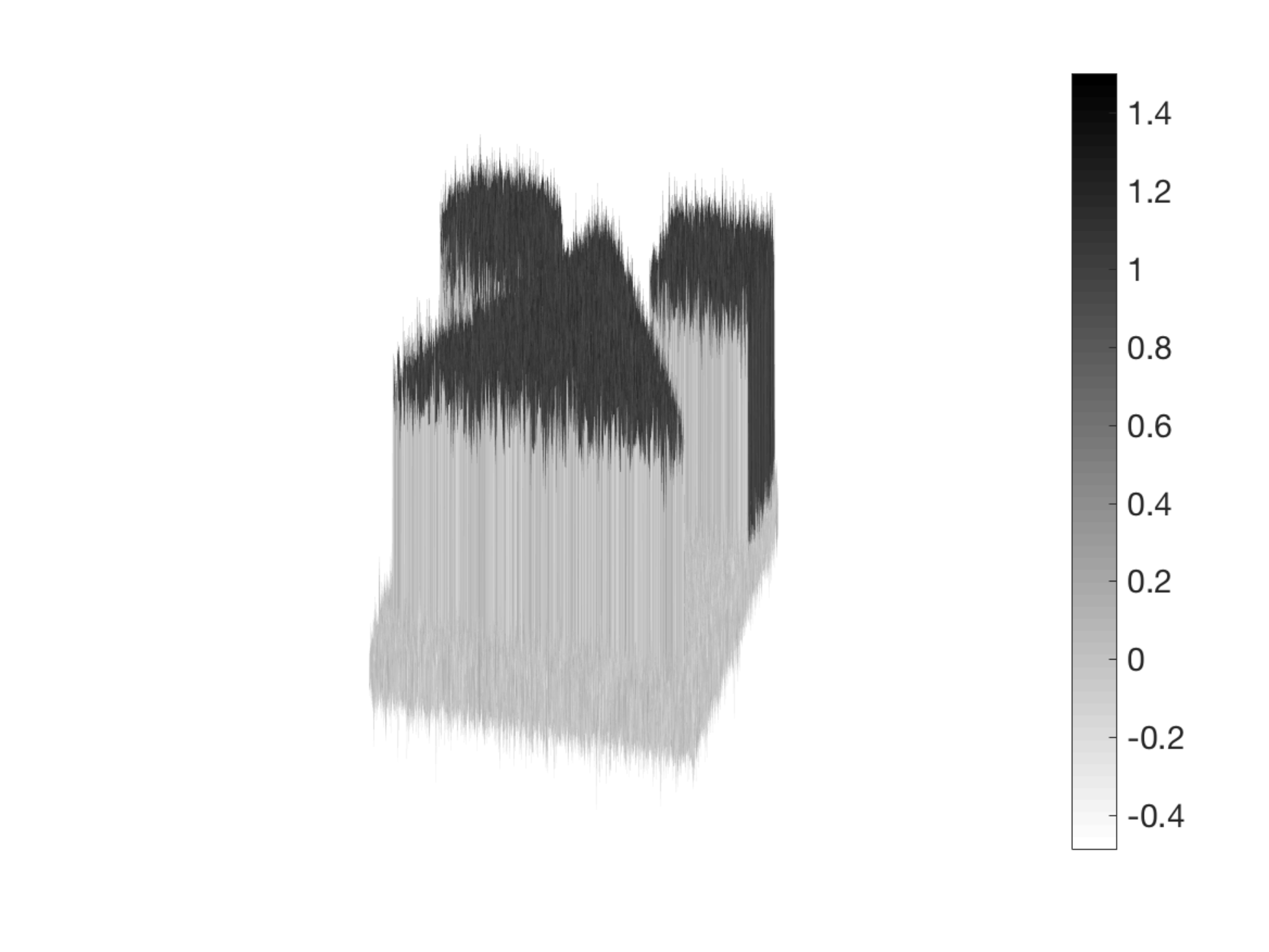}
\hspace{-0.59cm}
\includegraphics[width=0.4\textwidth]{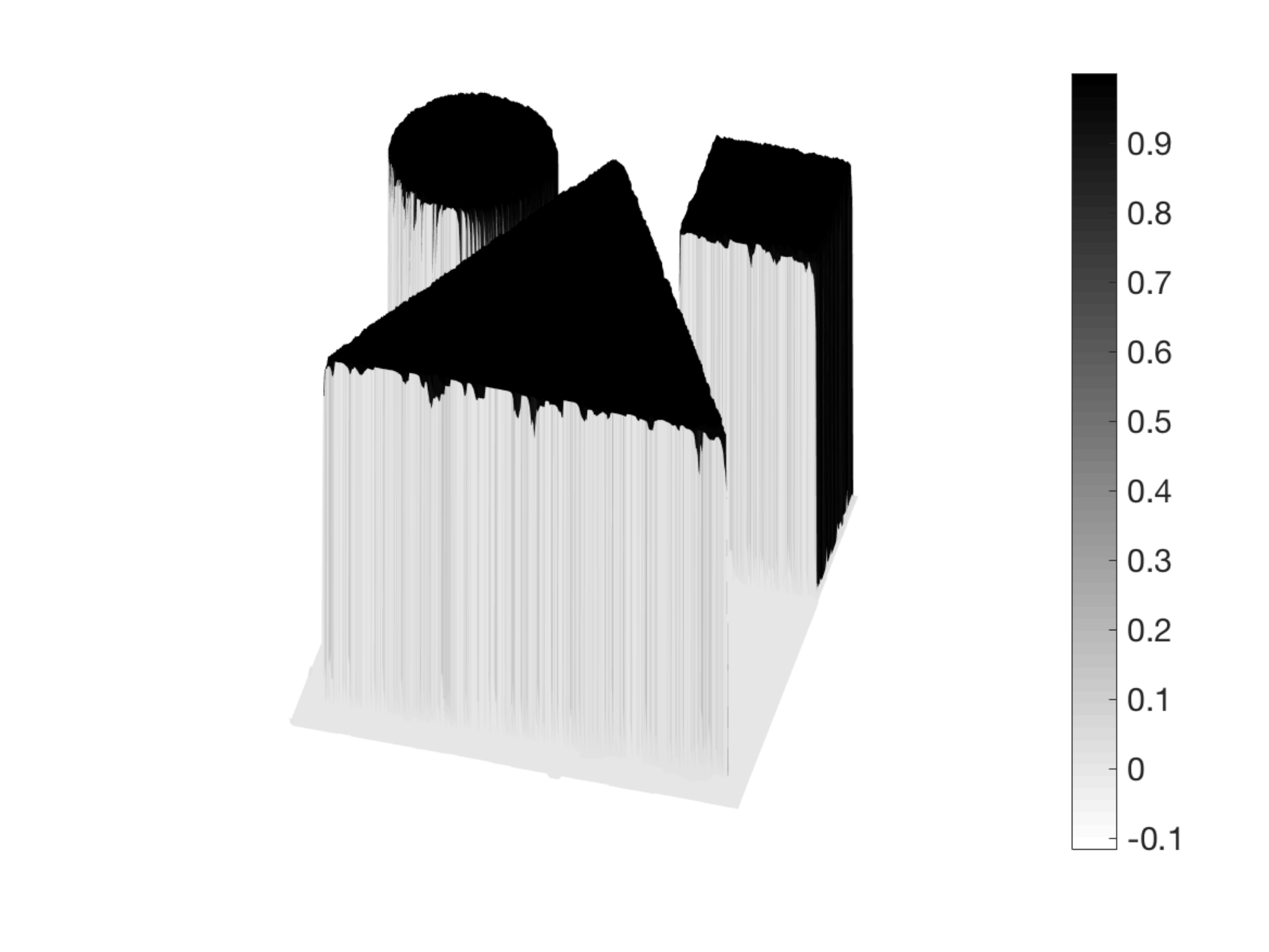}
\hspace{-0.62cm}
\includegraphics[width=0.4\textwidth]{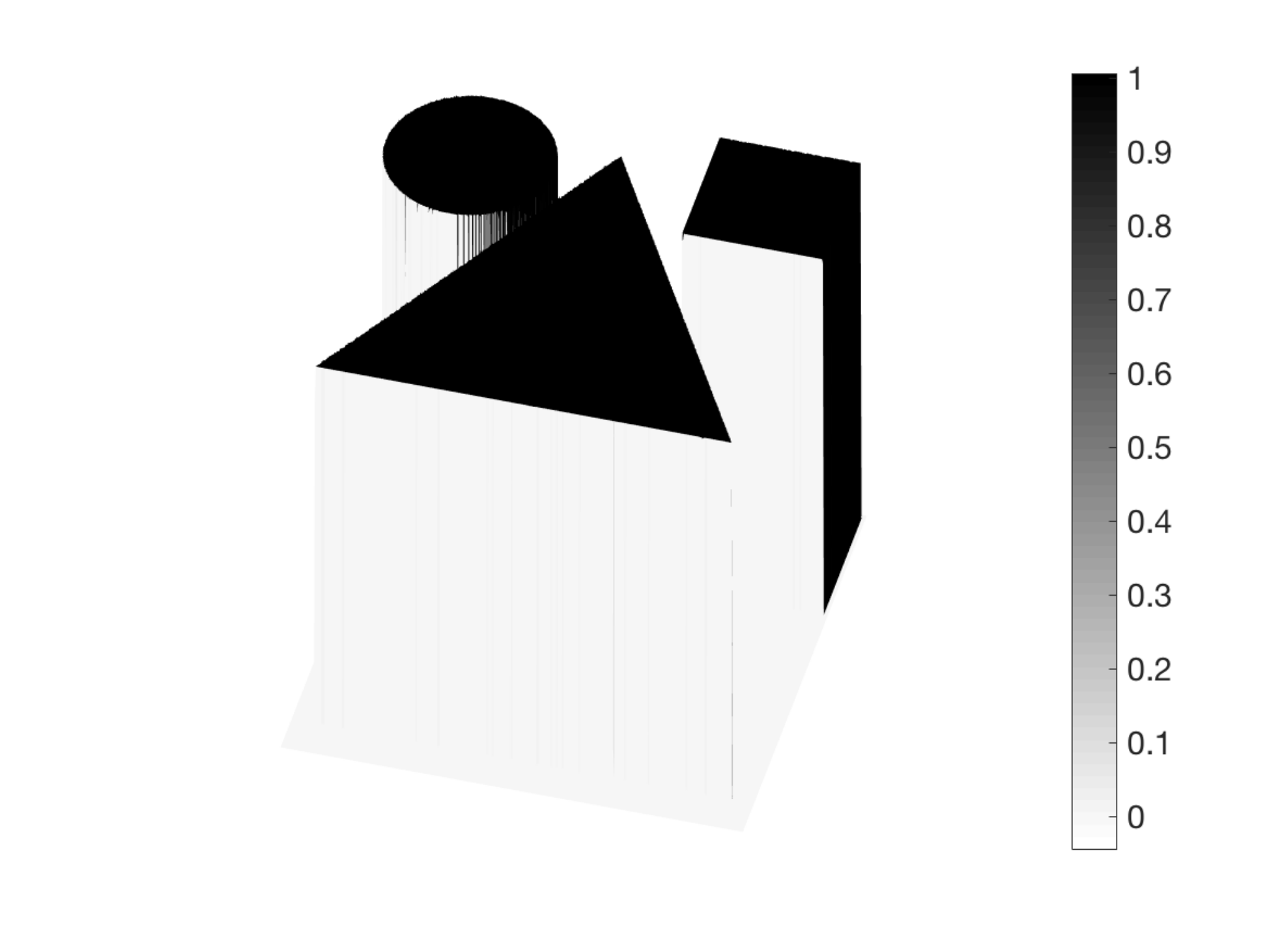}
\hspace{-0.8cm}
\caption{\label{f:ex_2_diffview_intro} Example 1. Left panel: noisy image. Middle panel: reconstruction
using Total Variation (TV) approach \cite{AChambolle_2004a}. Right panel: reconstruction using 
the approach introduced in this paper. 
}
\end{figure}
where the left panel 
shows a noisy image, the middle panel shows the reconstruction using the Total Variation (TV) model, explained in what follows,
(we use piecewise linear finite element discretization for the TV model and we stop
the algorithm when the relative error between the consecutive iterates is smaller than $1e$-8.), and the rightmost panel shows the reconstruction using 
the new framework introduced in this paper and briefly explained in this section. The Rudin-Osher-Fatemi (ROF) model, also called the  
TV model, was introduced  in \cite{LIRudin_SOsher_EFatemi_1992a} and is given by
\begin{equation}\label{eq:tv_intro}
 \min_{u \in {\rm BV}(\Omega)}  \int_{\Omega}|D u| + \frac{\zeta}{2} \|u-f\|^2_{L^2(\Omega)} ,
\end{equation}
where $\Omega \subset \mathbb{R}^N$ with $N \ge 1$ 
is a open bounded Lipschitz domain, and $f \in L^2(\Omega)$ is given by $f=u_{\mathrm{true}}+\xi$, where $\xi$ is the ``noise''. Moreover, 
$\zeta > 0$ is the regularization parameter, and  
 $\int_{\Omega}|D u|$ denotes the total variation seminorm of $u$ on $\Omega$ 
; see \cite{HAttouch_GButtazzo_GMichaille_2014a}.  The parameter $\zeta > 0$ in Figure~\ref{f:ex_2_diffview_intro} is chosen so that it maximizes a weighted sum of the Peak Signal to 
Noise Ratio (PSNR) and the Structural Similarity index (SSIM); as in \cite{MHintermueller_CNRautenberg_TWu_ALanger_2017a}. These two indexes measure with different criteria how close the solution $u_{\rm tv}$ to \eqref{eq:tv_intro} is to $u_{\mathrm{true}}$; see \cite{wang2004image}.
%
The fundamental question in such applications is:  \emph{Is it possible to capture the sharp 
transitions (edges) across the interfaces while removing the undesirable noise from remainder
of the domain.} 
%
Such a question is not limited to image denoising but is fundamental to many applications
in science and engineering. For instance multiphase flows (diffuse interface models),
fracture mechanics, image segmentation etc. 
Even though the ROF model has been extremely successful in practice, two main drawbacks are observed in reconstructions: 1) loss of contrast is always present 2) some corners tend to be rounded  (cf.~Figure~\ref{f:ex_2_diffview_intro}, 
middle panel).

The problem \eqref{eq:s_fix_intro} proposed in \cite{HAntil_SBartels_2017a} is obtained if  $\int_{\Omega}|D \cdot |$ is replaced by $\frac{1}{2}\|(-\Delta)^{\frac{s}{2}} \cdot\|_{L^2(\Omega)}^2$  
where $(-\Delta)^s$ denotes the fractional powers of the Laplacian with zero Neumann boundary
conditions, for $s\in (0,1)$. Solutions to \eqref{eq:s_fix_intro} are given in $H^s(\Omega)$, the fractional Sobolev space given by
\begin{equation*}
H^s(\Omega)=\left\{ v\in L^2(\Omega): \int_\Omega\int_\Omega \frac{|v(x)-v(y)|^2}{|x-y|^{n+2s}} \dif x\dif y<+\infty\right\}.
\end{equation*}
The minimization problem 
\eqref{eq:s_fix_intro} is appealing, in contrast to \eqref{eq:tv_intro},  as the Euler-Lagrange equations are linear in $u$, i.e., if $u$ solves \eqref{eq:s_fix_intro}, then equivalently it solves
\begin{equation}\label{eq:EL_s_fix}
  \langle (-\Delta)^s u, v\rangle+ \zeta (u-f, v)  = 0,
\end{equation}
for all $v\in H^s(\Omega)$. Both \eqref{eq:s_fix_intro} and hence \eqref{eq:EL_s_fix} have unique solutions 
in $H^s(\Omega)$. This can be easily deduced from the definition of Neumann fractional Laplacian 
and the equivalence between the spectral and $H^s(\Omega)$-norms (cf. \cite[Remark~2.5]{antil2017fractional} and \cite{LACaffarelli_PRStinga_2016a}). 
The $H^s(\Omega)$ space further provides flexibility in terms of the regularity by choosing $s$ appropriately. 

The Caffarelli-Silvestre (or Stinga-Torrea) extension technique establishes that if $u\in H^s(\Omega)$ solves \eqref{eq:EL_s_fix}, then there exists $\mathcal{U}\in H(\C;y^{1-2s})$ where $H(\C;y^{1-2s})$ is the weighted $H^1$ Sobolev space on $\C := \Omega \times (0,\infty)$ with weight $w(x,y)=y^{1-2s}$, $(x,y)\in \mathcal{C}$ (see the following section for more details), such that $\tr \mathcal{U}=u$, where $\tr$ operator is the restriction of the trace map to $\Omega$, and $\mathcal{U}$ solves
\begin{align}\label{eq:ext_var_intro0}
\begin{aligned}
 \int_\C y^{1-2s}\nabla \mathcal{U} \cdot \nabla \mathcal{W} & \dif x \dif y
  +  d_s \zeta \int_\Omega  (\tr \mathcal{U}-f)   \:\tr \mathcal{W} \:\dif x=0, 
\end{aligned}                       
\end{align}
for all $\mathcal{W}\in H(\C;y^{1-2s}) $, where  $d_s=2^{1-2s} \frac{\Gamma(1-s)}{\Gamma(s)}$. Therefore, for $s$ a constant such that  $s(x)=s\in (0,1)$ for all $x\in \Omega$, $\theta=0$ and $\zeta := \mu s^2 d_s^{-1}$, if $\mathcal{U}$ solves \eqref{eq:s_intro}, then it equivalently solves \eqref{eq:ext_var_intro0}. This establishes the connection of  \eqref{eq:s_intro} with \eqref{eq:EL_s_fix}, and hence with \eqref{eq:s_fix_intro}.

Further, if $s(\cdot)$ is not a constant and $s:\Omega\to[0,1]$ (with some additional assumptions made explicit in the following section), and if $\mathcal{U}\in H(\C;y^{1-2s(x)})$ solves \eqref{eq:s_intro}, then 
\begin{align}\label{eq:ext_var_intro}
\begin{aligned}
 \int_\C y^{1-2s(x)} \left(\nabla \mathcal{U} \cdot \nabla \mathcal{W} + \theta \mathcal{U} \mathcal{W}\right) & \dif x \dif y
  +  \mu \int_\Omega s(x)^2 (\tr \mathcal{U}-f)   \:\tr \mathcal{W} \:\dif x=0 , 
\end{aligned}                       
\end{align}
for all $\mathcal{W}\in H(\C;y^{1-2s(x)})$; this is rigorously done in section \ref{s:prob} where the key ingredient is the characterization of $\tr$ of $H(\C;y^{1-2s(x)})$. It is clear that for $\theta\ll1$, the above problem (whenever well-posed) represents a generalization of \eqref{eq:ext_var_intro0} and hence of \eqref{eq:s_fix_intro}.

Now we are in shape to explain Figure~\ref{f:ex_2_diffview_intro}. First, note that for $s\simeq 0$, we expect low regularity of solutions, and for $s\simeq 1$ we expect higher regularity of solutions: In fact, for constant $0<s_1< s_2<1$, we observe $H^{s_2}(\Omega)\subset H^{s_1}(\Omega)$.  Then, after a rough identification of edges, we choose $s$ small on edges, and large on other regions; the exact procedure is provided in section \ref{Sselect}. The right image in Figure~\ref{f:ex_2_diffview_intro}, is obtained via this procedure.

\section{Weighted Sobolev space with  variable $s$ and their traces}\label{s:rest}

Let $\Omega \subset\mathbb{R}^N$ with $N \ge 1$ be a non-empty, open, bounded domain 
with Lipschitz boundary $\partial\Omega$. We denote by $\C := \Omega \times (0,\infty)$ a semi-infinite 
cylinder with boundary $\partial_L\C := \partial\Omega \times [0,\infty)$. With $s:\Omega\to\mathbb{R}$ measurable and $s(x) \in [0,1]$ for almost all $x \in \Omega$, we define 
\begin{equation*}
\delta(x) := 1-2s(x) \in [-1,1], \qquad \text{ and } \qquad  w(x,y):=y^{\delta(x)},
\end{equation*}
so that
\begin{equation}\label{eg:intcond}
w, w^{-1}\in L^1_{\mathrm{loc}}(\C).
\end{equation}
We denote $L^2(\C;y^{\delta(x)}):=L^2(\C;w)$, where
\begin{equation*}
L^2(\C;y^{\delta(x)}):=\left\{v:\C\to \mathbb{R}: v \text{ is measurable and } \int_\C y^{\delta(x)} |v|^2 \dif x\dif y<+\infty\right\},
\end{equation*}
which is a well-defined weighted Lebesgue space.

For $u$ smooth, define the (extended real-valued) functional $ \|\cdot\|_{H}$  as
\begin{equation}\label{eq:EnergyNorm}
 \|u\|_{H} 
  := \left( \int_\C y^{\delta(x)} \left(|u|^2 + |\nabla u|^2 \right) \dif x\dif y \right)^{\frac12}.
\end{equation}
Note that if $ \|u\|_{H} <+\infty$, the value of $s$ controls how singular the function $u$ is in the neighborhood of $\Omega\times \{0\}$. Define $\mathcal{A}_\alpha(s):=\{x\in \Omega: s(x)=\alpha, \:\: \text{a.e.}\}$, suppose that for $\alpha=1$ its Lebesgue measure satisfies $|\mathcal{A}_1(s)|>0$, and that  $u$ is a constant in $\mathcal{A}_1(s)\times (0,h)$, then
\begin{equation*}
\|u\|^2_{H} \geq |u|^2|\mathcal{A}_1(s)|\int_0^h y^{-1}\dif y,
\end{equation*}
so that $u$ is zero in $\Omega\times (0,h)$. On the other hand, $u$ is allowed to have a more singular behaviour on $\mathcal{A}_0(s)$.

Consider the set 
\begin{equation*}
\C_\Omega=(\Omega\times \{0\})\cup\C,
\end{equation*}
that is, $\C_\Omega$ is the open cylinder $\C$ together with the $\Omega$ cap,
and denote by $H$, $H_0$, and $W$, or equivalently by  $H(\C;y^{\delta(x)})$, $H_0(\C;y^{\delta(x)})$, and $W(\C;y^{\delta(x)})$, the spaces 
$H(\C;w)$, $H_0(\C;w)$ and $W(\C;w)$ which are defined as
\begin{description}
\item[]
\item[] $H(\C;w)$ is the closure of the set $K:=\{w\in C^\infty(\overline{\C}): \|w\|_H<+\infty\}$ with respect to the $\|\cdot\|_H$ norm.
\item[]
\item[] $H_0(\C;w)$  is the closure of the set  $K_0:=\{w\in C_c^\infty(\C_\Omega): \|w\|_H<+\infty\}$ with respect to the $\|\cdot\|_H$  norm. 
\item[]
\item[] $W(\C;w)$ is the space of maps $u\in L^2(\C;y^{\delta(x)})\cap L_{\mathrm{loc}}^1(\C)$ with distributional gradients $\nabla u$ that satisfy $|\nabla u| \in L^2(\C;y^{\delta(x)})\cap L_{\mathrm{loc}}^1(\C)$.

\item[]
\end{description}

A few words are in order concerning the definition of $H$ given the slight difference from the existing literature: Usually, the definition of $H$ is given over the completion (of finite energy functions) of  $C^\infty(\C)$, in contrast to $C^\infty(\overline{\C})$, with respect to the $H$ norm; see  \cite{MR1669639}. Provided that $\partial\C$ is regular enough, $\C$ is bounded, and $w,w^{-1}$ are bounded above and below ($\epsilon>0$ away from zero) on a neighborhood of $\partial\C$, then the closure in the definition of $H$ can be equivalently taken with $C^\infty(\C)$ or $C^\infty(\overline{\C})$, indistinctly. Since in our case the two last conditions are not satisfied, we consider $H$ the way it is defined here.

First note that since $K_0\subset K$, which follows from $C_c^\infty(\C_\Omega)\subset C^\infty(\overline{\C})$, we have that $H_0\subset H$. The notation of ``$H_0$'' comes from the following: If $v\in K_0$, then $v|_{\partial \Omega \times \{0\}} =0$, so that restrictions of elements of $H_0$ to $\Omega$ have ``zero boundary conditions''. In particular, in the case $s\in (0,1)$ is constant, the latter remark can be taken out of quotation marks and considered in the sense of the trace; see \cite{ACapella_JDavila_LDupaigne_YSire_2011a}. The  condition \eqref{eg:intcond} implies (see \cite{MR1669639,MR775568}) that  the space $W$
is a proper weighted Sobolev space endowed with the $H$-norm.  While it holds that $H(\tilde{\C};\tilde{w})\subset W(\tilde{\C};\tilde{w})$, for general weights $\tilde{w}$ and domains $\tilde{\C}$, it does not necessarily hold that $H(\tilde{\C};\tilde{w})=W(\tilde{\C};\tilde{w})$. For $\tilde{w}=1$, the equality holds provided conditions on $\tilde{\C}$ are given, e.g., if $\tilde{\C}$ is bounded and has a Lipschitz boundary. A sufficient condition for $H(\tilde{\C};\tilde{w})=W(\tilde{\C};\tilde{w})$, when $\tilde{w}$ and $\tilde{\C}$ are benign enough, is that $\tilde{w}$ belongs to the $A_2(\tilde{\C})$  Muckenhoupt class, that is
\begin{equation}\label{eq:M}
\left(\frac{1}{Q}\int_{Q}\tilde{w}  \dif x\dif y\right)\left(\frac{1}{Q}\int_{Q}\tilde{w}^{-1} \dif x\dif y\right)\leq M,
\end{equation}
for some $M>0$ and all open cubes $Q\subset \tilde{\C}$, see \cite{BOTuresson_2000a,JDuoandikoetxea_2001a}.  Further, note that  \eqref{eg:intcond}, does not imply \eqref{eq:M}. Additionally, if $0<\epsilon \leq s\leq 1-\epsilon$ a.e., then $(x,y)\mapsto y^{1-2s}$ is a 
Muckenhoupt $A_2(\C)$ weight.   However, we are {particularly} interested in cases where $s$ satisfies 
\begin{equation}\label{eq:essinfs}
 \mathrm{ess\:inf}_{x\in \Omega} s(x) =0 ,  
\end{equation}
as this will allow almost perfect approximation of functions with jump discontinuities.
Remarkably enough condition \eqref{eq:essinfs} prevents $w$ to be a Muckenhoupt weight in 
general as we show next.

\begin{proposition}\label{eq:NoA2}
Suppose that for $x_0\in \Omega$, $s(x)\simeq |x-x_0|^q$ (locally) on $x_0$ for some $q>0$, that is, there exists positive constants $R_0, M_0, m_0$ such that
\begin{equation}\label{eq:sandwich}
m_0 |x-x_0|^q\leq s(x)\leq  M_0|x-x_0|^q,
\end{equation}
for all $x$ such that $|x-x_0|\leq R_0$. Then, $w\notin A_2(\C)$ where $w(x,y):=y^{1-2s(x)}$.
\end{proposition}

\begin{proof} \textit{First step.
Suppose first that $x_0=0\in \Omega$, $s(x)=|x|^q$ for $x\in B_{R_0}(0)$, and let $Q_R=B_{R}(0)\times (0,y_0)$ for $0<R\leq R_0<1$ and $0<y_0<1$.} 

Using cylindrical coordinates we observe
\begin{align*}
\int_{Q_R}w  \dif x\dif y&\geq  C_1 \int_0^R\int_0^{y_0}r^{N-1}y^{1-2 r^q} \dif y \dif r\geq C_2  \int_0^R r^{N-1}y_0^{2-2 r^q} \dif r \\
&\geq C_3 R^N y_0^2,
\end{align*}
where $C_1, C_2$, and $C_3$ are positive and independent of $R$ and $y_0$.  Similarly, 
\begin{align*}
\int_{Q_R}w^{-1}  \dif x\dif y&\geq C_1 \int_0^R\int_0^{y_0}r^{N-1}y^{-(1-2 r^q)} \dif y \dif r\geq C_4  \int_0^R r^{N-1-q}y_0^{2 r^q}  \dif r \\
&\geq C_5 R^{N-q} y_0^{2R^q},
\end{align*}
where also  $C_4$, and $C_5$ are positive and independent of $R$ and $y_0$. Since $|Q_R|\simeq R^N y_0$, we have
\begin{equation*}
\left(\frac{1}{|Q_R|}\int_{Q_R}w  \dif x\dif y\right)\left(\frac{1}{|Q_R|}\int_{Q_R}w^{-1} \dif x\dif y\right)\geq C_6  \frac{y_0^{2R^q}}{R^q},
\end{equation*}
for some $C_6>0$ independent of $R$. Taking $R\downarrow 0$, we have that the L.H.S. expression is unbounded.

\textit{Second step: Consider $x_0\neq 0$, $s(x)\neq |x|^q$, and show that the above inequality holds similarly for a  family of cubes $R\mapsto \tilde{Q}_R$}.

If $x_0\neq 0$ and $s(x)=|x|^q$, then a linear change of coordinates can be performed and the proof above still holds, so the choice of $x_0=0$ is without loss of generality. 

Additionally, let $\tilde{Q}_R\supset Q_R$ be defined as $\tilde{Q}_R=S_{R}(0)\times (0,y_0)$ where $S_R(0)$ is the smallest square that contains $B_{R}(0)$. Hence,  the quotient $|S_R(0)|/|B_R(0)|$ is independent of $R$ and for $c:=|\tilde{Q}_R|/|Q_R|$ we observe
\begin{align*}
&\left(\frac{1}{|\tilde{Q}_R|}\int_{\tilde{Q}_R}w  \dif x\dif y\right)\left(\frac{1}{|\tilde{Q}_R|}\int_{\tilde{Q}_R}w^{-1} \dif x\dif y\right)\\&\qquad\qquad\geq \frac{1}{c^2}\left(\frac{1}{|Q_R|}\int_{Q_R}w  \dif x\dif y\right)\left(\frac{1}{|Q_R|}\int_{Q_R}w^{-1} \dif x\dif y\right).
\end{align*}
Hence, by taking $R\downarrow 0$, we have that  $(x,y)\mapsto y^{1-2|x|^q}$ is not in $A_2(\C)$. 

Finally, for $s(x)\neq |x|^q$ but when \eqref{eq:sandwich} holds true, we similarly have 
\begin{align*}
&\left(\frac{1}{|Q|}\int_{Q}w  \dif x\dif y\right)\left(\frac{1}{|Q|}\int_{Q}w^{-1} \dif x\dif y\right)\\&\qquad\qquad\geq \tilde{c}\left(\frac{1}{|Q|}\int_{Q}\tilde{w}  \dif x\dif y\right)\left(\frac{1}{|Q|}\int_{Q}\tilde{w}^{-1} \dif x\dif y\right),
\end{align*}
for some $\tilde{c}>0$, where $\tilde{w}(x,y):=y^{1-2|x|^q}$ and $w(x,y):=y^{1-2s(x)}$, i.e., $w\notin A_2(\C)$. 
\end{proof} 
\begin{rem}
{\rm Although the above result is quite general with respect of the rate of decrease of $s$, at this point, we do not know if there are continuous functions $s$ with zeros within the domain $\Omega$, for which $(x,y)\mapsto y^{1-2s(x)}$ is an element of $A_2(\C)$.}
\end{rem}

 The study of the trace space of $H$ is of utmost importance in what follows. In particular, we are interested in the restriction of the trace operator to $\Omega$. The paper \cite{ACapella_JDavila_LDupaigne_YSire_2011a} studies the trace of $H$ when $s$ 
is independent of $x$ and away from 0. However, that approach which is based on 
\cite{lions1959theoremes} cannot be applied here. We next prove a trace characterization result, the proof is inspired by \cite{ANekvinda_1993a} where the authors 
consider bounded domains and weights of distance to the boundary type.

\begin{theorem}\label{thm:trace}
Let $s$ be such that its zero set $\mathcal{A}_0(s):=\{x\in \Omega: s(x)=0, \:\: \text{a.e.}\}$  has zero measure. Then, there exist a unique bounded linear trace operator 
\[
 \tr : H, H_0 \rightarrow L^2(\Omega; (1-\delta(x))^2)=L^2(\Omega; s(x)^2)
\]
such that $\tr(u) = u|_{\Omega\times\{0\}}$, for all $u\in H\cap C^\infty(\overline{\C})$, and hence also for all $u \in H_0\cap C_c^\infty(\C_\Omega)$.
\end{theorem}

\begin{proof}
Consider $u\in  H$ such that $u\in C^\infty(\overline{\C})$. Let $(x,y)\in \overline{\C}$ and $x\notin \mathcal{A}_0(s)$, then it  follows that \begin{equation*}
u(x,y)=u(x,0)+\int_0^1  y D_{N+1} u(x,ty) \dif t,
\end{equation*}
where $D_{N+1}$ corresponds to the partial derivative with respect to the $N+1$ coordinate. Then, multiplication by $ y^{\delta(x)/2}$ and integrating from $0$ to $\sigma>0$ with respect to $y$ leads to
\begin{equation*}
|u(x,0)|\int _0^\sigma y^{\delta(x)/2} \dif y\leq\int _0^\sigma  |u(x,y)|y^{\delta(x)/2} \dif y+ \int _0^\sigma\int_0^1   y |D_{N+1} u(x,ty)| y^{\delta(x)/2} \dif y\dif t.
\end{equation*}
Note that $\delta(x)/2\leq 1/2$ and suppose without loss of generality $0<\sigma\leq 1$, then $y^{1/2}\leq y^{\delta(x)/2}$ for all $y\in (0,\sigma)$. Therefore,
\begin{equation*}
\frac{2}{3}\sigma^{3/2}\leq \int _0^\sigma y^{\delta(x)/2} \dif y,
\end{equation*}
and hence,
\begin{equation}\label{eq:IneqI1I2}
|u(x,0)|\leq I_1+I_2, 
\end{equation}
for 
$$
I_1:=\frac{3}{2}\sigma^{-3/2}\int _0^\sigma  |u(x,y)|y^{\delta(x)/2} \dif y , \;\; 
I_2=\frac{3}{2}\sigma^{-3/2}\int_0^1\int _0^\sigma   y |D_{N+1} u(x,ty)| y^{\delta(x)/2} \dif y\dif t ,
$$
and where we have used Tonelli's theorem to switch the order of integration in the expression that orginates $I_2$.

H\"{o}lder's inequality implies that, for $I_1$, we have
\begin{align*}
I_1&\leq \frac{3}{2}\sigma^{-\frac32} \left (\int _0^\sigma  \dif y\right)^{\frac12} \left(\int _0^\sigma  |u(x,y)|^2y^{\delta(x)}  \dif y\right)^{\frac12}\leq \frac{3}{2}\sigma^{-1} \left(\int _0^R  |u(x,y)|^2y^{\delta(x)}  \dif y\right)^{\frac12},
\end{align*}
for any $R\geq \sigma$. 

Further, for $I_2$ consider the change of variable $z=ty$ and again by H\"{o}lder's inequality we observe:
\begin{align*}
&I_2\leq\frac{3}{2}\sigma^{-\frac12}\int_0^1\int _0^{\sigma t}  |D_{N+1} u(x,z)| z^{\frac{\delta(x)}{2}} t^{-1-\frac{\delta(x)}{2}} \dif z\dif t\\
&\hspace{-0.2cm}\leq \frac{3\sigma^{-\frac12} }{2}\hspace{-0.1cm}\left (\int_0^1 \left(t^{-\frac{1+\delta(x)}{4}}\right)^2\hspace{-0.2cm}\dif t\right )^{\frac12}\hspace{-0.1cm} \left (\int_0^1 \left( \int _0^{\sigma t}  |D_{N+1} u(x,z)| z^{\frac{\delta(x)}{2}} t^{\frac{1+\delta(x)}{4}-1-\frac{\delta(x)}{2}} \dif z\right)^2\hspace{-0.2cm} \dif t\right )^{\frac12}\\
&\leq \frac{3}{2}\sigma^{-\frac12} \left (\frac{2}{1-\delta(x)}\right )^{\frac12} \left (\int_0^1 t^{-\frac{3+\delta(x)}{2}} \left( \int _0^{\sigma t}  |D_{N+1} u(x,z)| z^{\frac{\delta(x)}{2}}  \dif z\right)^2 \dif t\right )^{\frac12}\\
&\leq \frac{3}{2} \left (\frac{2}{1-\delta(x)}\right )^{\frac12} \left (\int_0^1 t^{-\frac{1+\delta(x)}{2}} \left( \int _0^{\sigma t}  |D_{N+1} u(x,z)|^2 z^{\delta(x)}  \dif z\right) \dif t  \right )^{\frac12}.
\end{align*}
Then, for any $R\geq \sigma$, we have
\begin{align*}
I_2&\leq \frac{3}{2} \left (\frac{2}{1-\delta(x)}\right )^{1/2} \left (\int_0^1 t^{-\frac{1+\delta(x)}{2}}\dif t  \right )^{1/2} \left( \int _0^{R}  |D_{N+1} u(x,z)|^2 z^{\delta(x)}  \dif z   \right )^{1/2}\\
&\leq  \frac{3}{1-\delta(x)} \left( \int _0^{R}  |D_{N+1} u(x,z)|^2 z^{\delta(x)}  \dif z\right )^{1/2}.
\end{align*}

Therefore, from \eqref{eq:IneqI1I2}, we obtain that for any $R\geq \sigma$
\begin{align*}
&|u(x,0)|(1-\delta(x))\\
&\quad\leq  \frac{3}{2}\sigma^{-1} (1-\delta(x)) \left(\int _0^R  |u(x,y)|^2y^{\delta(x)}  \dif y\right)^{\frac12}+3 \left( \int _0^{R}  |D_{N+1} u(x,z)|^2 z^{\delta(x)}  \dif z\right )^{\frac12}\\
&\quad\leq 3\sigma^{-1}\left(\int _0^R  |u(x,y)|^2y^{\delta(x)}  \dif y\right)^{\frac12}+3 \left( \int _0^{R}  |D_{N+1} u(x,z)|^2 z^{\delta(x)}  \dif z\right )^{\frac12}\\
&\quad\leq 6\sigma^{-1} \left(\int _0^R |u(x,y)|^2y^{\delta(x)}  \dif y + \int _0^{R}  |D_{N+1} u(x,z)|^2 z^{\delta(x)}  \dif z \right)^{\frac12}.
\end{align*}
Integration over $\Omega$ with respect to $x$ of the above squared inequality leads to
\begin{align*}
&\int_\Omega |u(x,0)|^2(1-\delta(x))^2\dif x\\
&\qquad\qquad\leq M_\sigma^2 \left(\int_\Omega\int _0^R |u(x,y)|^2y^{\delta(x)}  \dif y + \int_\Omega\int _0^{R}  |D_{N+1} u(x,z)|^2 z^{\delta(x)}  \dif z \right)\\
&\qquad\qquad\leq M_\sigma^2 \left(\int_\Omega\int _0^R |u(x,y)|^2y^{\delta(x)}  \dif y + \int_\Omega\int _0^{R}  |\nabla u(x,z)|^2 z^{\delta(x)}  \dif z \right),
\end{align*}
where $M_\sigma= 6\sigma^{-1}$, $R\geq \sigma$, with $\sigma\in (0,1]$ and for an arbitrary $u\in  H$ such that $u\in C^\infty(\overline{\C})$.

Since $R\geq \sigma$ is arbitrary, we have that for an arbitrary $u\in  H\cap C^\infty(\overline{\C})$, we have
\begin{equation*}
\|u(\cdot,0)\|_{L^2(\Omega; (1-\delta(x))^2)}\leq M_\sigma \|u\|_{H}.
\end{equation*}
The operator $\tr$ is the above extension to $H$: Since $H$ is the closure of  
$C^\infty(\overline{\C})$
with respect to the energy norm, we can take $\sigma=1$ with  $M_\sigma=6$, and obtain
\begin{equation*}
\|\tr u\|_{L^2(\Omega; (1-\delta(x))^2)}\leq 6 \|u\|_{H},
\end{equation*}
which completes the proof, since   $(1-\delta(x))^2=4s(x)^2$. The result for $H_0$ is analogous. 
\end{proof}

The behavior of $s$ controls the local regularity on space $\tr H$, in particular, $\tr H$ contains piecewise constants in any dimension for appropriate choices of $s$ as we see in the following example. 

\begin{exa}\label{exampledisc}There exists non-constant $s:\Omega\to \mathbb{R}$ maps  such that $\tr H$ contains piecewise smooth functions. For the sake of simplicity, let $\Omega=B_{1/2}(0)$, define $s(x)=\hat{s}(x_1)$, where $x=(x_1,x_2,\dots,x_N)$, and such that $0<\delta\leq \hat{s}(x_1)\leq \kappa<1/4$ for all $|x_1|<\epsilon$, and $0<\delta\leq \hat{s}(x_1)$ for all $x_1$.  Let
\begin{equation*}
w(x,y):=\left\{
  \begin{array}{ll}
    0, & \hbox{$x_1< -\sqrt{y}$ ;} \\
    \frac{1}{2\sqrt{y}}(x_1+\sqrt{y}), & \hbox{$\sqrt{y}\leq x_1\leq-\sqrt{y}$ ;} \\
    1, & \hbox{ $\sqrt{y}<x_1$.}
  \end{array}
\right.
\end{equation*}
so that 
\begin{equation*}
|\nabla w(x,y)|^2=\left\{
  \begin{array}{ll}
    0, & \hbox{$x_1< -\sqrt{y}$ ;} \\
    \frac{1}{4y}+\frac{x_1^2}{16 y^3} , & \hbox{$\sqrt{y}\leq x_1\leq-\sqrt{y}$ ;} \\
    0, & \hbox{ $\sqrt{y}<x_1$.}
  \end{array}
\right.
\end{equation*}
Define $u(x,y):=\eta(y)w(x,y)$ for $(x,y)\in\C$, with $\eta\in C^\infty_c(\mathbb{R})$, such that $\eta(y)=1$ if $y\in [0,R^*]$ for some $R^*$, and therefore,
\begin{equation*}
\int_{B_{\epsilon}(0)}\int_0^1y^{1-2s}|\nabla w(x,y)|^2\dif x \dif y \leq \int_{B_{\epsilon}(0)}\int_0^1y^{1-2\kappa}|\nabla w(x,y)|^2\dif x \dif y<+\infty.
\end{equation*}
Then, $u\in W$ and smoothing argument involving mollification shows that $u\in H$, and $\tr u = \chi_{\{x_1\geq 0\}}$. 
\end{exa}

\section{The Optimization Problem }\label{s:prob}
 Throughout this section we suppose that $f\in L^2(\Omega; s(x)^2)$, $s(x) \in [0,1]$ 
for almost all $x \in \Omega$, and that  $\mathcal{A}_0(s):=\{x\in \Omega: s(x)=0, \:\: \text{a.e.}\}$  has zero measure. 

Given a regularization parameter $\mu > 0$, and parameter 
$\theta > 0$ we consider the following variational problem 
\begin{equation}\label{1st}
\min_{u\in H(\mathcal{C}; y^{1-2s(x)})} J(u,s):=\int_{\C}y^{1-2s(x)} \left(\theta |u|^2 + |\nabla u|^2 \right) \dif x\dif y+ \frac{\mu}{2}\int_{\Omega} s(x)^2 |\tr u - f| ^2  \dif x.
\end{equation}

\begin{rem}[Parameter $\theta$]
\label{rem:par_theta}
{\rm 
In case the weights $w$ are of class $A_2$, we can set 
 $\theta = 0$, given that  the Poincar\'e type 
 inequalities are available in this case.  However, the weights $y^{1-2s(x)}$ in our case are not in 
 $A_2$ since we allow $s(x) =0$ for some $x\in \Omega$. Such Poincar\'e type inequalities on $H(\C;y^{1-2s(x)})$ are not known yet.
} 
\end{rem}
The model \eqref{1st} provides a completely new approach to approximate nonsmooth
functions $f$ with $\tr u$. Such a situation often occurs in inverse 
problems where given $f$ one is interested in a reconstruction $\tr u$ which is 
close to $f$. For instance in 
image denoising problems where $f$ represents a noisy image, given by $f=u_{\mathrm{true}}+\xi$ where $u_{\mathrm{true}}$ is the true target to recover, $\int_\Omega\xi=0$, and $\int_\Omega|\xi|^2=\sigma^2$ for a known  $\sigma>0$. The first term in \eqref{1st} is the regularization and the second term  is the so-called data fidelity, together they  ensure that the reconstruction $\tr u$ is close 
to $u_{\mathrm{true}}$ on $\Omega$.

For a fixed $s$ and given $f$ we next state existence and uniqueness of solution to \eqref{1st}. This follows immediately from Theorem~\ref{thm:trace}.
\begin{theorem}[existence and uniqueness]
\label{thm:exist_var}
There exist a unique solution to \eqref{1st}. 
\end{theorem}
\begin{proof}
{Existence follows from application of direct methods of calculus of variations, the norm definition on $H$ and the Theorem~\ref{thm:trace}. Uniqueness follows directly from convexity arguments.}
\end{proof}
The first order necessary and sufficient optimality conditions for the minimization 
problem \eqref{1st} are given by: 
\begin{align}\label{eq:ext_var}
\begin{aligned}
 \int_\C y^{1-2s(x)} \left(\nabla u \cdot \nabla w + \theta u w\right) \dif x \dif y
  + & \mu \int_\Omega \tr u   \:\tr w \:s(x)^2\dif x  \\
  &= \mu\int_\Omega f \tr w \:s(x)^2\dif x, \quad \mbox{for all } w\in H . 
\end{aligned}                       
\end{align}
The first step for computational implementation of the problem above requires to solve the problem in a bounded domain. Notice that all results above are valid if we replace the unbounded domain $\C$ with boundary $\partial_L \C$ 
by a bounded domain $\C_\tau = \Omega \times (0,\tau)$ with $0< \tau < \infty$ and boundary $\partial_L\C_\tau = (\partial\Omega \times [0,\tau]) \cup \Omega \times \{\tau\}$. In that case the problem \eqref{eq:ext_var} becomes: find $v \in H^\tau := H(\C_\tau;y^{\delta(x)})$ 
\begin{align}\label{eq:ext_var_trunc}
\begin{aligned}
 \int_{\C_\tau} y^{1-2s(x)} \left(\nabla v \cdot \nabla w + \theta v w\right) \dif x \dif y
  + & \mu \int_\Omega \tr v   \:\tr w \:s(x)^2\dif x  \\
  &= \mu\int_\Omega f \tr w \:s(x)^2\dif x, \quad \mbox{for all } w\in H^\tau . 
\end{aligned}                       
\end{align}
It is known that for a constant $s \in (0,1)$ the solution $v$ to \eqref{eq:ext_var_trunc} converges to $u$ solving \eqref{eq:ext_var} exponentially with respect to $\tau > 1$, see 
\cite{RHNochetto_EOtarola_AJSalgado_2014a}. Such exponential approximation is also expected in the case of a non-constant $s$. Indeed, as we discussed earlier, a constant $s$ implies that 
$\tr u$ solves the fractional equation \eqref{eq:EL_s_fix}. A relation between \eqref{eq:ext_var}
and fractional PDE of type \eqref{eq:EL_s_fix} with $s(x)$ instead of $s$ is an open 
question.  We expect that studying \eqref{eq:ext_var} or equivalently \eqref{1st} is a first step in establishing the definition of fractional 
$(-\Delta)^{s(x)}$ for $x \in \Omega$.

For given $s$ and $f$, Theorem~\ref{thm:exist_var} shows existence of a unique solution to  \eqref{1st} and equivalently \eqref{eq:ext_var}. Nevertheless for our
model \eqref{1st} to be useful in practice we need to determine the function $s$. For this matter, consider the following important points that will lead to the selection procedure 
for $s$.

\begin{enumerate}[$(i)$]
 \item \label{disc}
  {\bf Precise edge recovery.} In example \ref{exampledisc}, a family of $s$ functions    is identified so that $u\in H(\mathcal{C}; y^{1-2s(x)})$, and $\tr u = \chi_{\{x_1\geq 0\}}$. This suggests that for the optimization problem \eqref{1st} to recover edges, or more precisely for $\tr u$ to have discontinuities in the same places where $u_{\mathrm{true}}$ has them  (for $f=u_{\mathrm{true}}+\mathrm{noise}$) the function $s$ needs to be close to or equal to zero in these regions. This suggest to utilize a rough edge detection at first and then force $s$ to be close zero on neighborhoods of these regions.

 \item \label{homogeneous}
  {\bf Homogeneous/flat regions recovery.} We proceed rather formally first. Suppose first  that $\theta=0$, and  that  problem \eqref{1st} is posed not on $H(\mathcal{C}; y^{1-2s(x)})$, but on the Banach space of equivalence classes generated by the seminorm $\int_{\C}y^{1-2s(x)}  |\nabla u|^2 \dif x\dif y$. Consider $s(x)=1$ on a certain region $\Omega_0$. Suppose that $|\nabla u|\geq e>0$ on $\Omega_0\times (0,y_0)$, then
\begin{equation*}
\int_{\C}y^{1-2s(x)}  |\nabla u|^2 \dif x\dif y\geq \epsilon \int_{\Omega_0\times (0,y_0)}y^{-1}  \dif x\dif y =+\infty.
\end{equation*}
Hence, the solution to  \eqref{1st}, needs to satisfy $|\nabla u|\simeq 0$ near $\Omega\times\{0\}$. This suggest that flat regions would be recovered if $s=1$ on that same region. However, if $\theta>0$, then analogously, solutions to  problem \eqref{1st}, would be forced to satisfy $|u|\simeq 0$ near $\Omega\times\{0\}$. Hence, we consider $0<\theta\ll 1$ and $s(x)\simeq 1$, but $s(x)<1$, on regions where homogeneous features are present so that $u$ is not forced to be identically zero but $|\nabla u|\simeq 0$ is enforced there.  
\end{enumerate}

We provide an algorithm based on the two points above in what follows. A more detailed 
bilevel optimization framework (where both $u$ and $s$ are optimization variables) 
will be considered in a forthcoming publication.

\section{Numerical Method}\label{s:num_method}

We now focus on the discretization of the truncated problem \eqref{eq:ext_var_trunc} and the selection of an appropriate $s$ function.

\subsection{Discretization of \eqref{eq:ext_var_trunc}}\label{s:disc}

From hereon we will assume that $\Omega$ is polygonal/polyhedral Lipschitz. 
We recall that the results of previous sections remains valid if we replace the 
unbounded domain $\C$ with boundary $\partial_L\C$ by $\C_\tau = \Omega \times (0,\tau)$ 
with $\tau > 0$ and boundary $\partial_L\C = (\partial\Omega \times [0,\tau]) \cup \Omega \times \{\tau\}$. The Euler-Lagrange equations for the resulting problem on $\C_\tau$ are given
by \eqref{eq:ext_var_trunc}.

We begin by introducing a discretization for \eqref{eq:ext_var_trunc}. We will follow the notation from \cite{HAntil_CNRautenberg_2016a}. For a constant $s$ such a discretization for $v$ was  first 
considered in \cite{RHNochetto_EOtarola_AJSalgado_2014a}. Let $\mathscr{T}_\Omega=\{E\}$ be a conforming and 
quasi-uniform triangulation of $\Omega$, where $E\in \mathbb{R}^N$ is an element that is 
isoparametrically equivalent to either to the unit cube or to the unit simplex in $\mathbb{R}^N$. 
We assume $\# \mathscr{T}_\Omega \propto M^N$. Thus, the element size $h_{\mathscr{T}_\Omega}$ 
fulfills $h_{\mathscr{T}_\Omega}\propto M^{-1}$. 

Furthermore, let $\mathcal{I}_\tau=\{I_k\}_{k=0}^{K-1}$, where $I_k = [y_k,y_{k+1}]$, 
is anisotropic mesh in $[0,\tau]$ in the sense that $[0,\tau]=\bigcup_{k=0}^{K-1} I_k$. For a constant $s \in (0,1)$ we define the anisotropic mesh in $y$-direction as 
\begin{equation}\label{eq:grad}
 y_k = \left(\frac{k}{K}\right)^\gamma , 
 \quad k = 0, \dots,K, \quad \gamma > \frac{1}{s} .
\end{equation}
This choice is motivated by the singular behavior of the solution towards the boundary $\Omega$ for a constant $s$. In that case anisotropically refined meshes are preferable as these can be used to compensate the singular effects \cite{meidner2017hp, RHNochetto_EOtarola_AJSalgado_2014a}. 
In all our implementations we will choose a fixed constant $s$ in \eqref{eq:grad}.
 
We construct the triangulations $\mathscr{T}_\tau$ of the cylinder $\mathcal{C}_\tau$ as tensor product triangulations by using $\mathscr{T}_\Omega$ and $\mathcal{I}_\tau$. Let $\mathbb{T}$ denotes the collection of such anisotropic meshes $\mathscr{T}_\tau$. For each $\mathscr{T}_\tau \in \mathbb{T}$ we define the finite element space $\mathbb{V}(\mathscr{T}_\tau)$ as 
\[
\mathbb{V}(\mathscr{T}_\tau):=\{V\in C^0(\overline{ \mathcal{C}_\tau}):V|_{T}\in\mathcal{P}_1(E)\oplus\mathbb{P}_1(I)\ \forall T=E\times I\in \mathscr{T}_\tau \}.
\]
In case $\partial\Omega$ has Dirichlet boundary conditions we define our finite element space as $\mathbb{V}_0(\mathscr{T}_\tau) = \mathbb{V}(\mathscr{T}_\tau) \cap \{V : V|_{\partial\C_\tau} = 0\}$, i.e., functions with zero boundary conditions. 
In case $E$ is a simplex then $\mathcal{P}_1(E)=\mathbb{P}_1(E)$, the set of polynomials of degree at most $1$. If $E$ is a cube then $\mathcal{P}_1(E)$ equals $\mathbb{Q}_1(E)$, the set of polynomials of degree at most 1 in each variable. In our numerical illustrations we shall work with simplices. 

We define the finite element space for $s$ as   
\[
 \mathbb{S}(\mathscr{T}_\Omega) := \{S \in L^\infty(\Omega) : S|_E \in \mathbb{P}_0(K)
  \quad \mbox{for all } E \in \mathscr{T}_\Omega \}
\]
which is a space of piecewise constant functions defined on $\mathscr{T}_\Omega$. The discrete 
version of \eqref{eq:ext_var_trunc} is then given by: 
Find $V \in \mathbb{V}(\mathscr{T}_\tau)$ 
\begin{align}\label{eq:ext_var_trunc_disc}
\begin{aligned}
 \int_{\C_\tau} y^{1-2S(x)}  & \left(\nabla V \cdot \nabla W + \theta V W\right) \dif x \dif y
  + \mu \int_\Omega \tr V   \:\tr W \:S(x)^2\dif x  \\
  &= \mu\int_\Omega f \tr W \:S(x)^2\dif x, \quad \mbox{for all } W\in \mathbb{V}(\mathscr{T}_\tau) , 
\end{aligned}                       
\end{align}
for $S \in \mathbb{S}(\mathscr{T}_\Omega)$. We compute the stiffness and mass matrices 
in \eqref{eq:ext_var_trunc_disc} exactly. The corresponding forcing boundary term is computed by a quadrature formula which is exact for polynomials up to degree 5. For a given $S$ the resulting discrete linear system \eqref{eq:ext_var_trunc_disc} is solved 
using the Preconditioned Conjugate Gradient (PCG) method with a block diagonal preconditioner.

\subsection{Parameter selection}\label{Sselect}

In view of what was stated about $\theta$ in section \ref{s:prob}, we have set $\theta = 10e$-10 in all our examples. 
We let 
$
 \tau = 1+\frac{1}{3}(\#\mathscr{T}_\Omega) ,
$
this choice is motivated by the fact that for a constant $s$ such a $\tau$ 
balances the finite element approximation on $\C_\tau$ and the truncation error 
from $\C$ to $\C_\tau$ \cite[Remark 5.5]{RHNochetto_EOtarola_AJSalgado_2014a}. 
The number of points in the $y$-direction is taken to be $K=20$. We use
a moderately anisotropic mesh in the $y$-direction by setting $s = 0.32$ in \eqref{eq:grad}.
Our experiments suggest that the results are stable under reducing $\theta$ or $s$
or increasing $K$ further. 

It then remains to specify the constant $\mu$ and the function $S$ in order to realize 
\eqref{eq:ext_var_trunc_disc}. We have observed that $\mu$ only affects the 
``contrast" or magnitude of $\tr V$. Nevertheless one can determine $\mu$ using 
the well established techniques such as L-curve method \cite{PCHansen_DPOleary_1993a}. In our case 
we choose a fixed $\mu$ for each example (no optimization was carried to select $\mu$). 
On the other hand selecting $S$ which is a function is much more delicate. One option is to use 
a bilevel strategy as in \cite{MHintermueller_CNRautenberg_2017a,MHintermueller_CNRautenberg_TWu_ALanger_2017a} to 
determine both $S$ and $V$ in an optimization framework. This is a part of our future work. In 
this paper we propose a different approach in Algorithm~\ref{Algorithm}.

Even though our targeted application is image denoising the approach we present is general 
enough to be applicable to a wider range of applications. We first notice that a typical 
image is given on a rectangular grid (pixels). Since for \eqref{eq:ext_var_trunc_disc} we are working on simplices we first need to interpolate the given image from the rectangular grid to a 
simplicial mesh. This is a delicate question especially given the fact that typically we only 
have access to a noisy image. Before we interpolate the noisy image onto a simplicial mesh
we perform an intermediate step. 

We solve ROF model \eqref{eq:tv_intro} using \cite{AChambolle_2004a} and $\zeta>0$ chosen to a fixed value for all examples. We stop the algorithm 
when the relative difference between two consecutive iterates is smaller 
than the given tolerance ${\rm tol_{tv}}$. We select a mild tolerance ${\rm tol_{tv}}$ 
in this step as we want to preserve the sharp features but still remove a certain noise.
We call the resulting solution as $u_{\rm tv}$. We then generate a piecewise linear Lagrange interpolant $I_{\mathscr{T}_\Omega} u_{\rm tv}$ of $u_{\rm tv}$.

We next evaluate 
$I_{\mathscr{T}_\Omega} u_{\rm tu}$ on the simplicial mesh. In order to reduce the computational
cost we use Adaptive Finite Element Method (AFEM) to generate the appropriate mesh. In the
nutshell, we start with a coarse mesh $\mathscr{T}_\Omega = \{E\}$ where $E$ is an element in
$\mathscr{T}_\Omega$. For each $E\in \mathscr{T}_\Omega$ we then evaluate gradient of 
$I_{\mathscr{T}_\Omega} u_{\rm tv}$ on $E$ and we denote this gradient by 
$\nabla I_{\mathscr{T}_\Omega} u_{\rm tv}|_{E}$. We use this gradient to define an edge 
indicator function, we call this edge indicator function as the estimator on $E$. Based on 
a marking strategy we then mark a subset of elements in $\mathscr{T}_\Omega$. Subsequently
we perform the mesh refinements. We execute this loop $N_{\rm refine}$ times. 

Finally we set $S$ so that it is close to 0 at the sharp edges (large gradient) 
in the image and close to 1 away from sharp edges. Intuitively smaller the $S$ the lesser
the smoothness and otherwise.

\begin{algorithm}[h!]
	\caption{Selection of $S$}
		\label{Algorithm}
	\textbf{Data:} $f\in L^2(\Omega,s(x)^2)$, $\zeta$, ${\rm tol_{tv}}$, $N_{\rm refine}$, 
	$\lambda$, $\beta$, $\nu$
	\begin{algorithmic}[1]
	\State \label{step1} Solve total variation minimization problem \eqref{eq:tv_intro}
	 with regularization 
	       parameter $\zeta$ and tolerance ${\rm tol_{tv}}$ using for instance \cite{AChambolle_2004a}. Generate a piecewise linear Lagrange interpolant 
	       $I_{\mathscr{T}_\Omega} u_{\rm tv}$ of $u_{\rm tv}$.
	       
	\State \label{step2} Construct an Adaptive Finite Element Method (AFEM) based on the following 
	iterative loop with $N_{\rm refine}$ iterations:
	       \[
	         \textrm{\bf Solve} \rightarrow 
	         \textrm{\bf Estimate} \rightarrow \textrm{\bf Mark} \rightarrow 
	          \textrm{\bf Refine} 
	       \]
	       We describe each of these modules next:
	       \begin{enumerate}[a.]
	        \item {\bf Solve}: For a given triangulation $\mathscr{T}_\Omega  = \{E\}$     
	        evaluate the elementwise gradient of $I_{\mathscr{T}_\Omega} u_{\rm tv}$. 
	        We denote the gradient on each element $E$ by 
	        $\nabla I_{\mathscr{T}_\Omega} u_{\rm tv}|_E$. 
	        \item {\bf Estimate}:  Use the edge indicator function        
	       \[
              \mathcal{E}(E;\lambda) 
              = 1 - \left(1 + \lambda^{-2} 
               |\nabla I_{\mathscr{T}_\Omega}u_{\rm tv} |_{E}|^2\right)^{-1} \quad 
                  \forall E \in \mathscr{T}_\Omega 
           \]
           as an estimator. Here, $\lambda > 0$ is a given parameter. 
	        \item {\bf Mark}: Use the D\"orfler marking strategy \cite{WDoerfler_1996a} (bulk chasing criterion) with parameter $\beta \in (0,1]$. Select a set $\mathcal{M} \subset \mathscr{T}_\Omega$ fulfilling 
	        \[
	         \sum_{E \in \mathcal{M}} \mathcal{E}(E;\lambda)^2
	          \ge \beta \sum_{E \in \mathscr{T}_\Omega} \mathcal{E}(E;\lambda)^2 .
	        \]  
	        \item {\bf Refine}: Generate a new mesh $\mathscr{T}_\Omega' = \{E'\}$ by bisecting all the 
	        elements contained in $\mathcal{M}$ using the newest-vertex bisection algorithm 
	        \cite{RHNochetto_KGSiebert_AVeeser_2009a}. 
	       \end{enumerate}	     	        
    \State \label{step3} On each element $E' \in \mathscr{T}_\Omega'$, set $S(E') = 1-\mathcal{E}(E';\nu)$.       
	\end{algorithmic}
\end{algorithm}
Once we have $S$ then we can immediately solve the linear equation \eqref{eq:ext_var_trunc_disc} 
for $V$.

\section{Numerical Examples}\label{s:num_ex}

In this section we illustrate the proposed scheme with the help of several examples. We consider $f$ given by $f=u_{\mathrm{true}}+\xi$ where $u_{\mathrm{true}}$ is the true image, and $\xi$ is noise with the following properties: $\int_\Omega\xi=0$, and $\int_\Omega|\xi|^2=\sigma^2$ for a known  $\sigma>0$.
In all cases we find $S$ by using Algorithm~\ref{Algorithm} and then solve the 
resulting linear equation \eqref{eq:ext_var_trunc_disc} for $V$. We call $\tr V$ as
the reconstruction. We also compare our reconstruction $\tr V$ with the 
reconstruction we obtain by solving \eqref{eq:tv_intro} with ${\rm tol}^*_{\rm tv} = 10e$-8. 
In this case we choose the parameter $\zeta$ so that a normalized weighted sum of Peak Signal to 
Noise Ratio (PSNR) and the Structural Similarity (SSIM) index is maximized, as in  
section~4.2 of \cite{MHintermueller_CNRautenberg_TWu_ALanger_2017a}. That is, we compare our results to the ones of the TV model with an optimized parameter $\zeta>0$. Note that this parameter is not the one used in step \ref{step1} from algorithm \ref{Algorithm}.

\subsection{Example 1: circle, triangle and square}

In our first example we consider an image with a circle, triangle and a square as shown in
Figure~\ref{f:ex_1_sigma10} (top row left). Notice that in this case the right pointing edge of
the triangle does not align with the grid. 
We consider two different noise levels. In both cases the noise is normally distributed with mean 0 and standard deviations 
0.10 and 0.15 respectively. The results are shown in Figures~\ref{f:ex_1_sigma10} and 
\ref{f:ex_1_sigma15} for standard deviations $0.10$ and $0.15$, respectively. In both cases
we first compute $S$ by using Algorithm~\ref{Algorithm} where we have set:
$\zeta = 0.2$, 
${\rm tol}_{\rm tv} = 10e$-4, 
$N_{\rm refine} = 8$,
$\lambda = 300$, 
$\beta = 0.99$,
$\nu = 200$
in Algorithm~\ref{Algorithm}. We further set $\mu = 8050$ in \eqref{eq:ext_var_trunc_disc}.
We then solve for $V$ using PCG. We call $\tr V$ as our reconstruction.  

The top row of Figure~\ref{f:ex_1_sigma10} shows the original and noisy images (left to 
right). The middle row shows $u_{\rm tv}$ and $S$ obtained using Algorithm~\ref{Algorithm}. 
In the bottom row we compare the results using the TV approach (left) and $\tr V$ (right) 
computed using
our approach by solving \eqref{eq:ext_var_trunc_disc}. Notice that Figure~\ref{f:ex_2_diffview_intro} is simply obtained by viewing Figure~\ref{f:ex_1_sigma10}, 
in particular, the noisy, reconstruction using TV and $\tr V$ panels, from a different angle.
Similar description holds for Figure~\ref{f:ex_1_sigma15}. Figure~\ref{f:ex_2_diffview_15} 
is again viewing Figure~\ref{f:ex_1_sigma15} from a different angle.

As we can notice TV tends to round up the corners (cf. Figure~\ref{f:ex_2_diffview_intro} (middle)). On the other hand as our theory predicted we can truly capture the edges in 
$\tr V$ (cf. Figure~\ref{f:ex_2_diffview_intro} (right)). This is further corroborated by Table~\ref{t:ex_1} where we have shown a comparison between PSNR and SSIM for these two 
approaches for two different standard deviations.

\begin{table}[h!]
\centering
\begin{tabular}{|l|l|l|l|l|} \hline 
$\sigma$  &  PSNR (TV)  & PSNR (New) & SSIM (TV)   & SSIM (New)  \\ \hline
 0.1      &  3.7299e+01 & 4.8147e+01 & 9.4016e-01  & 9.5710e-01  \\ \hline 
 0.15     &  3.5712e+01 & 4.0451e+01 & 9.3439e-01  &  9.4890e-01 \\ \hline
\end{tabular}
\caption{\label{t:ex_1}Example 1: PSNR and SSIM using two 
different standard deviations ($\sigma = 0.1$ and $\sigma = 0.15$) using
TV and proposed scheme (New).}
\end{table}

\begin{figure}[h!]
\centering
\includegraphics[width=0.49\textwidth]{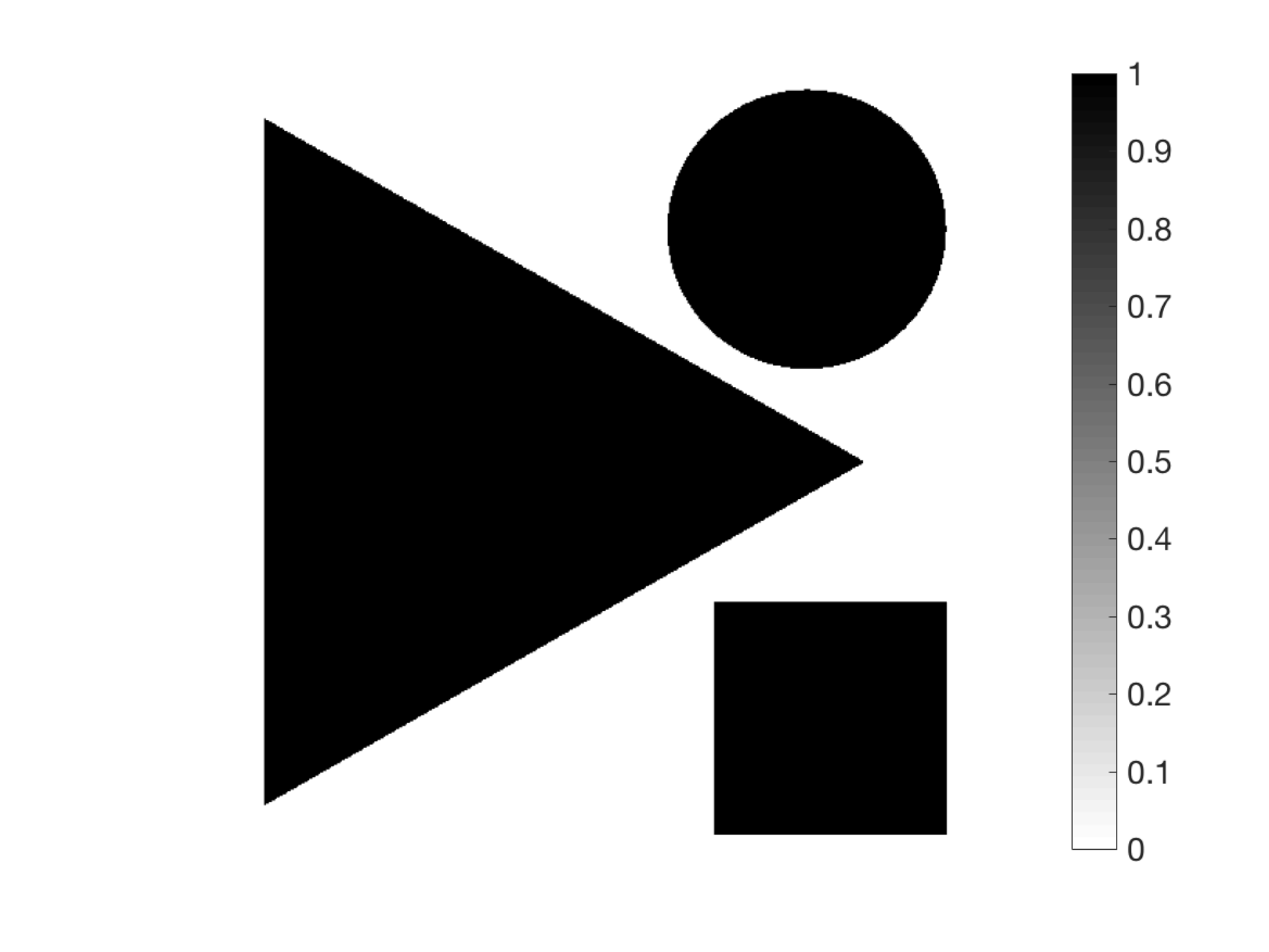}
\includegraphics[width=0.49\textwidth]{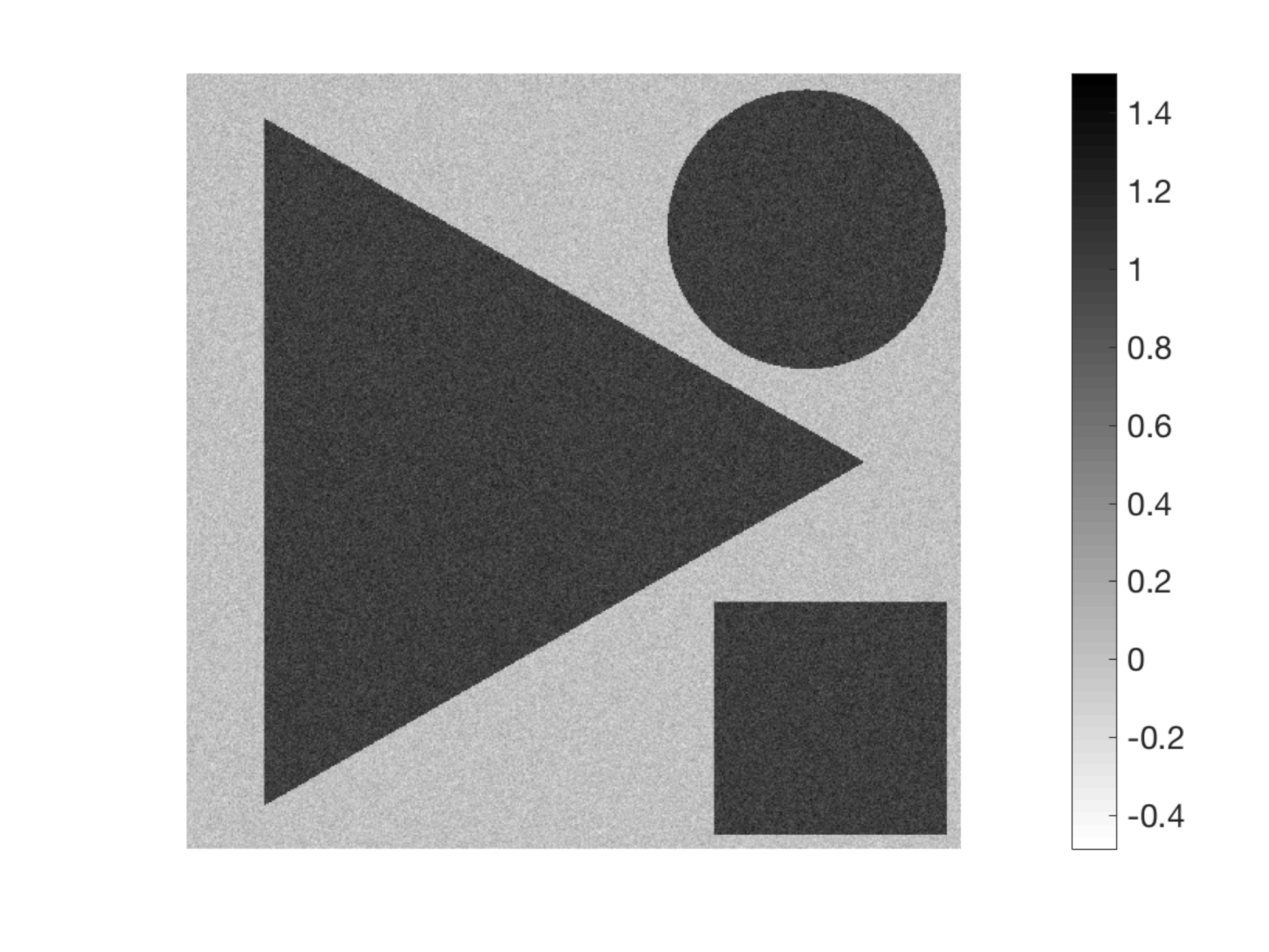}
\includegraphics[width=0.49\textwidth]{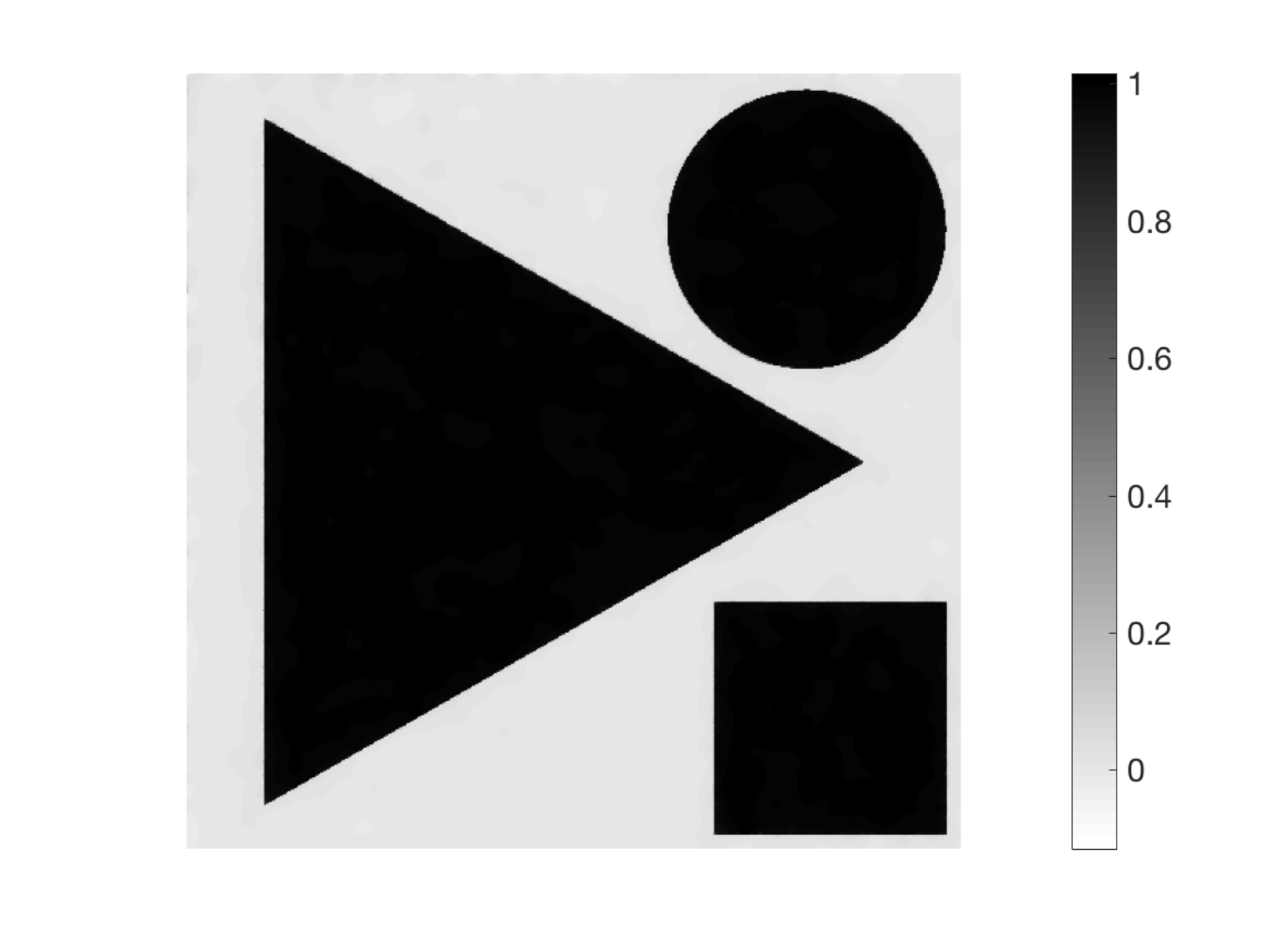}
\includegraphics[width=0.49\textwidth]{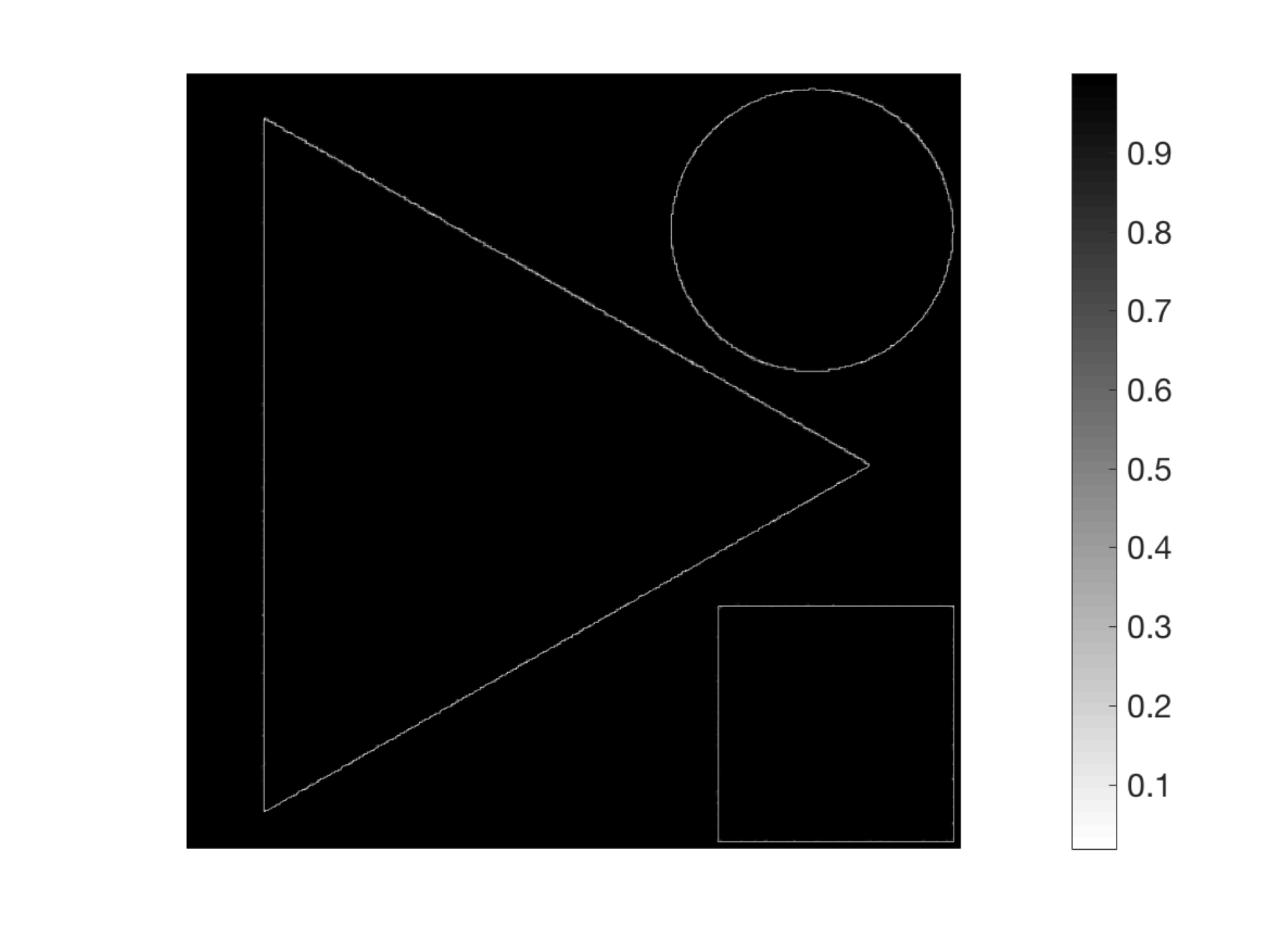}
\includegraphics[width=0.49\textwidth]{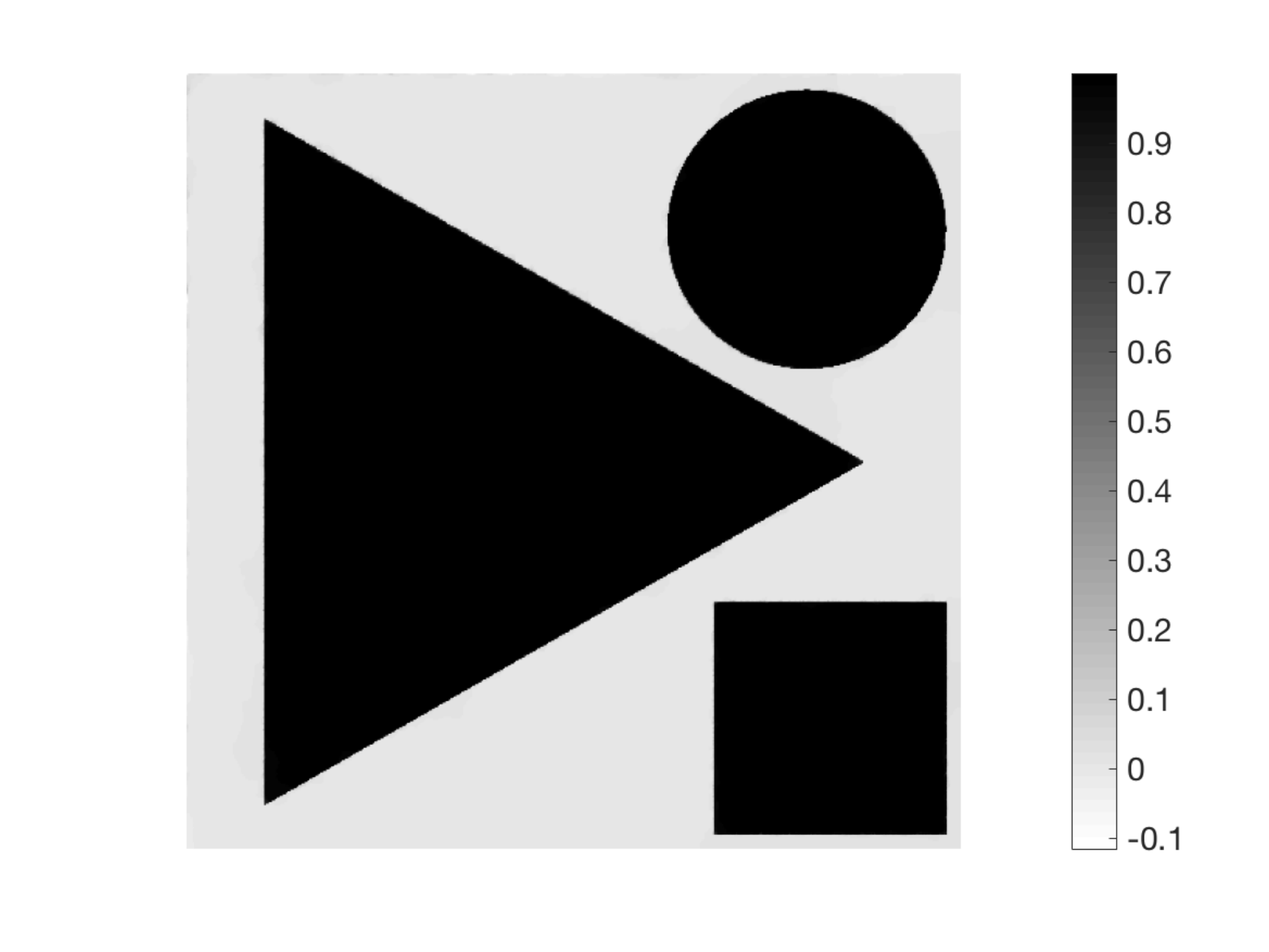}
\includegraphics[width=0.49\textwidth]{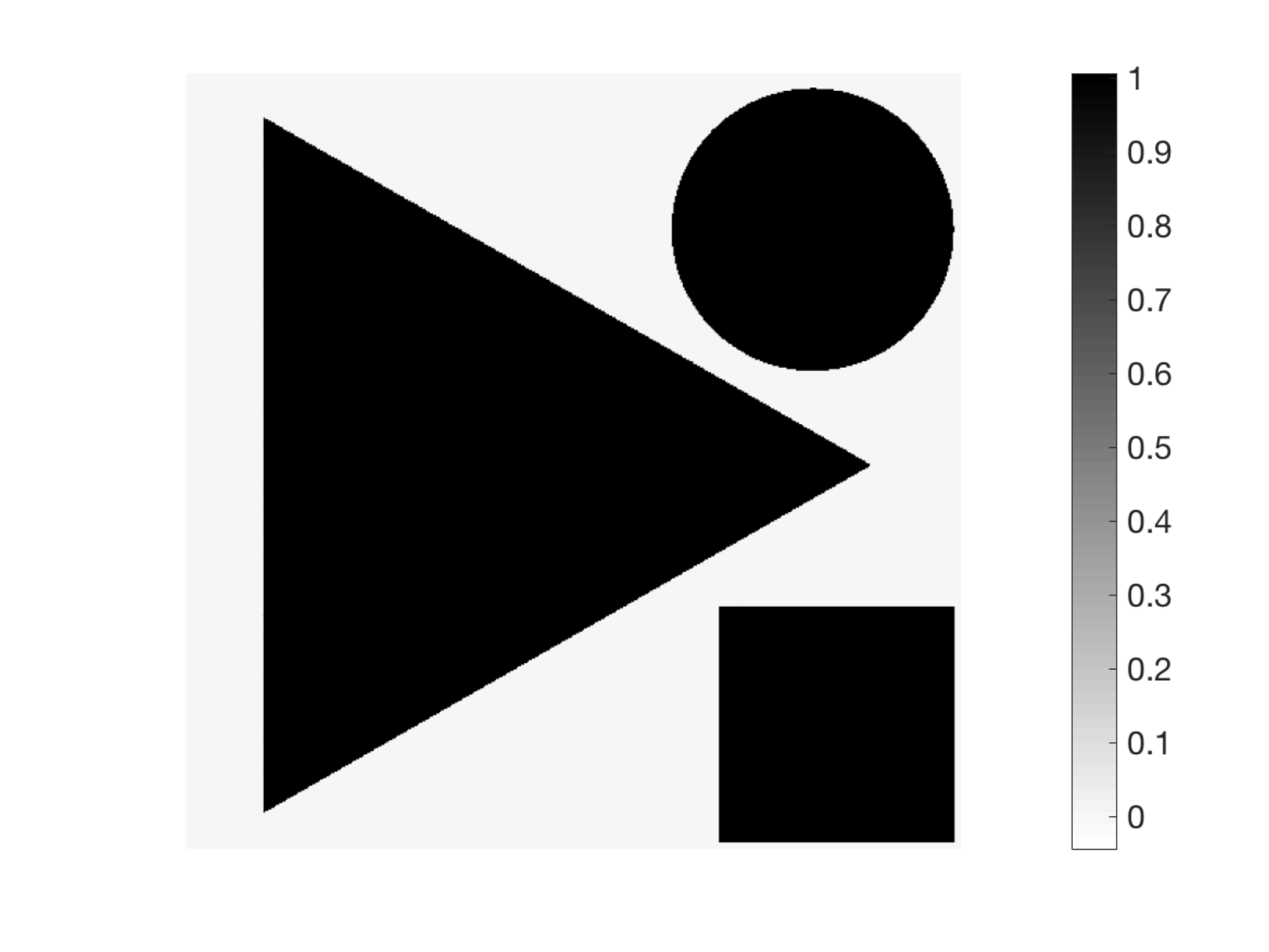}
\caption{\label{f:ex_1_sigma10}
Example 1 ($\sigma=0.1$). Top row (from left to right): original and noisy images, respectively. Middle row: $u_{\rm tv}$ (left) from Step~\ref{step1} in Algorithm~\ref{Algorithm} and the corresponding $s$  (right) from Step~\ref{step3}.
Bottom row: reconstruction using  total variation with optimized $\zeta>0$ (left) and our approach (right), respectively.}
\end{figure}

\begin{figure}[h!]
\centering
\includegraphics[width=0.49\textwidth]{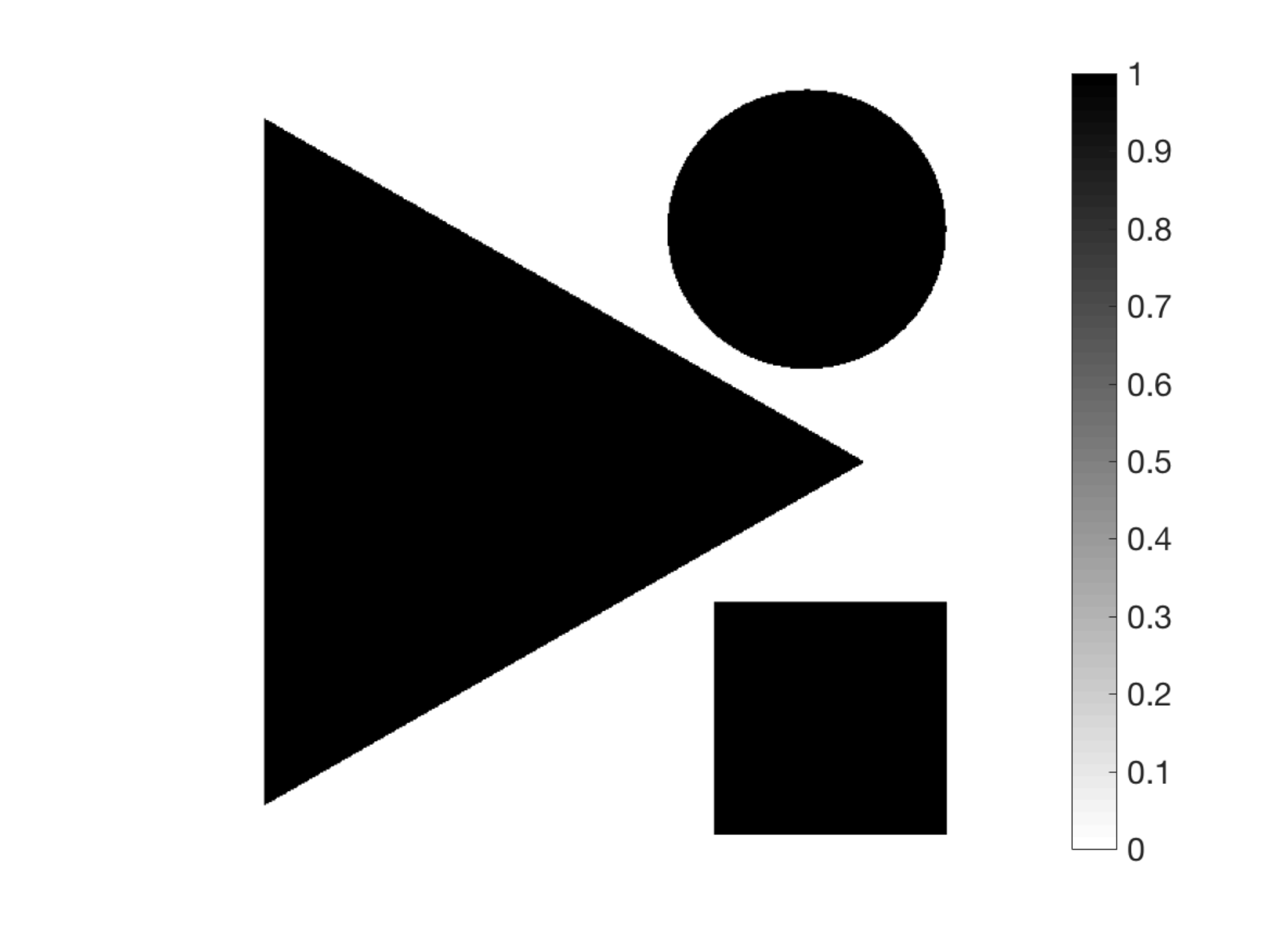}
\includegraphics[width=0.49\textwidth]{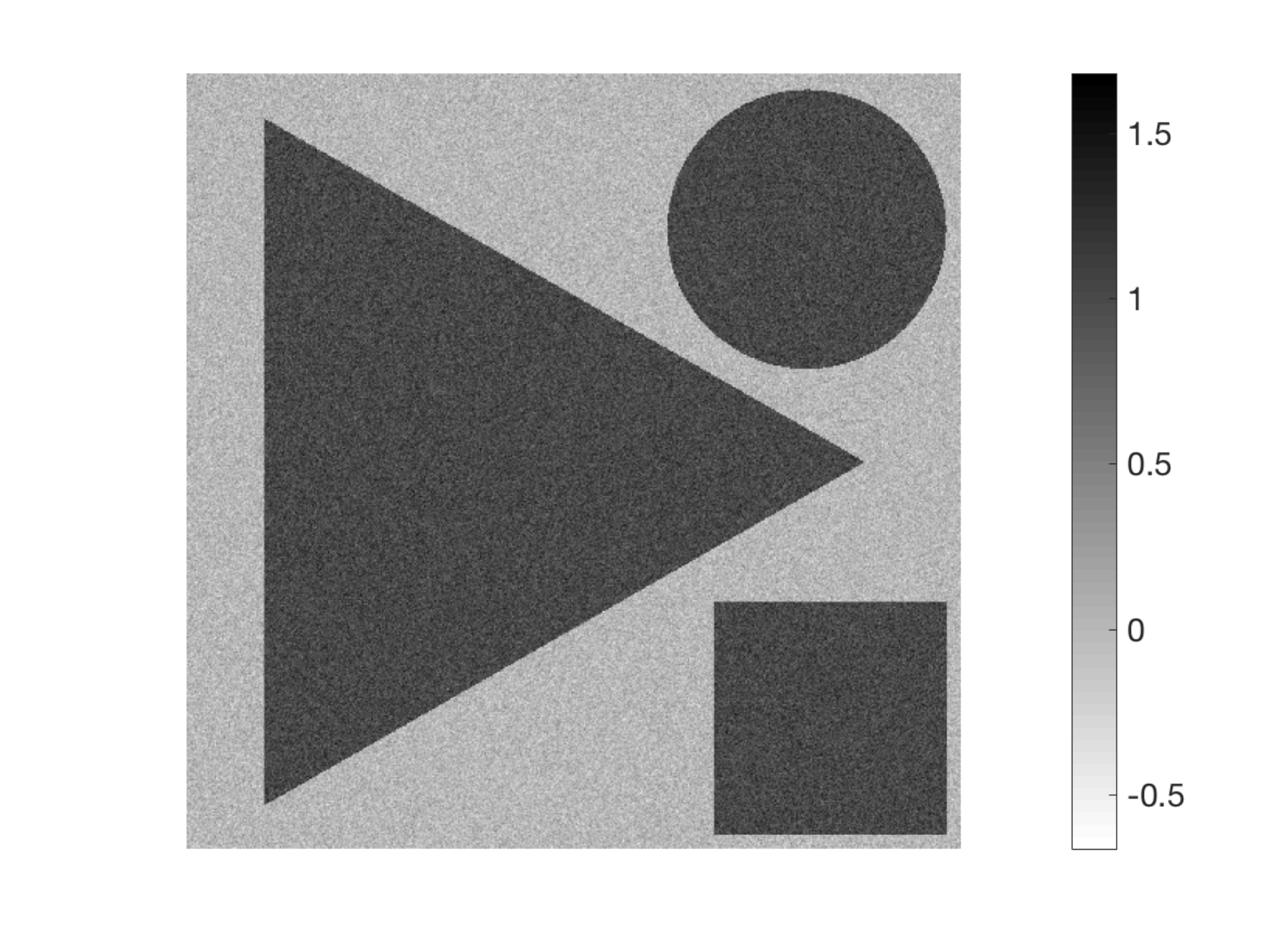}
\includegraphics[width=0.49\textwidth]{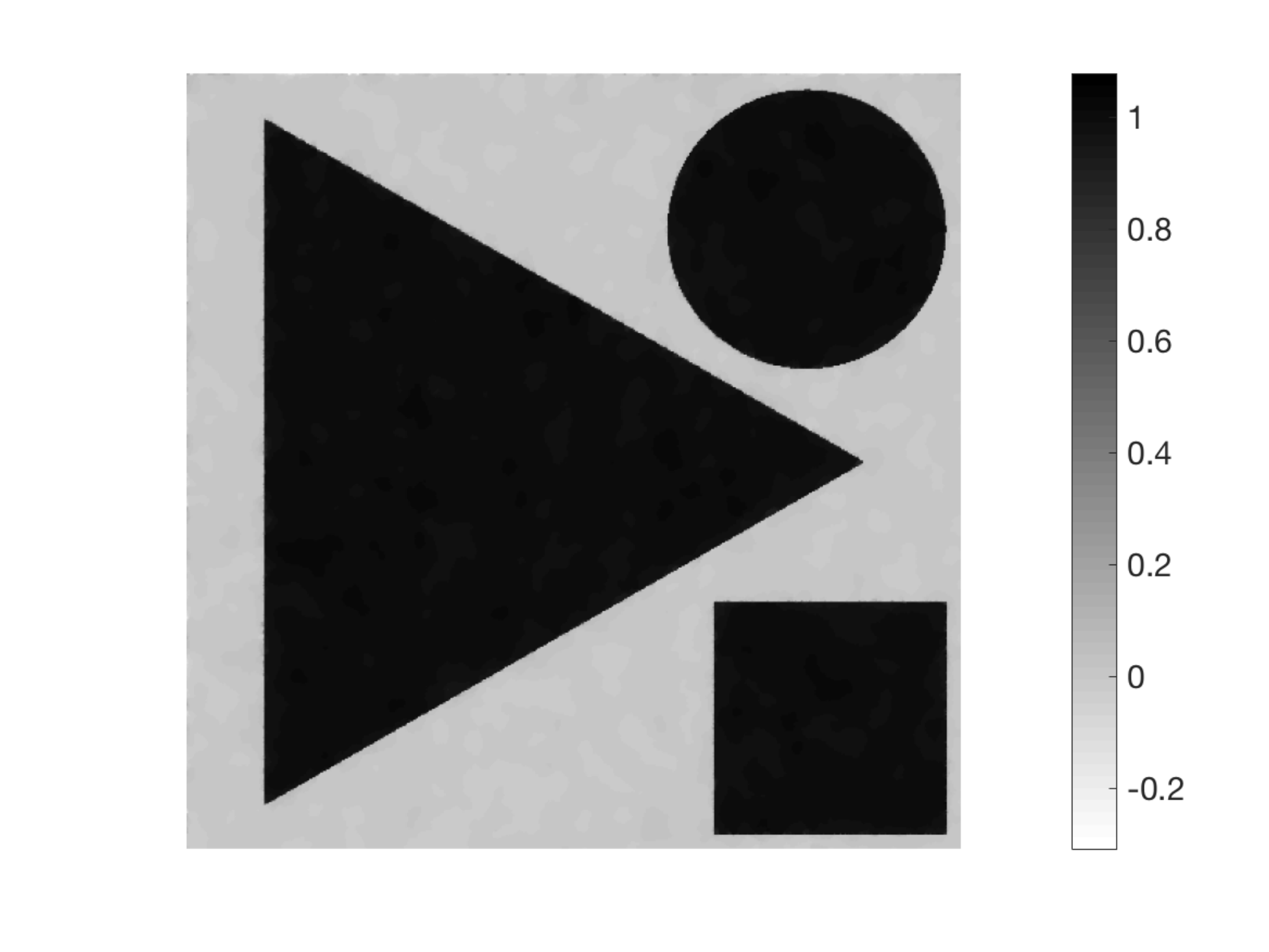}
\includegraphics[width=0.49\textwidth]{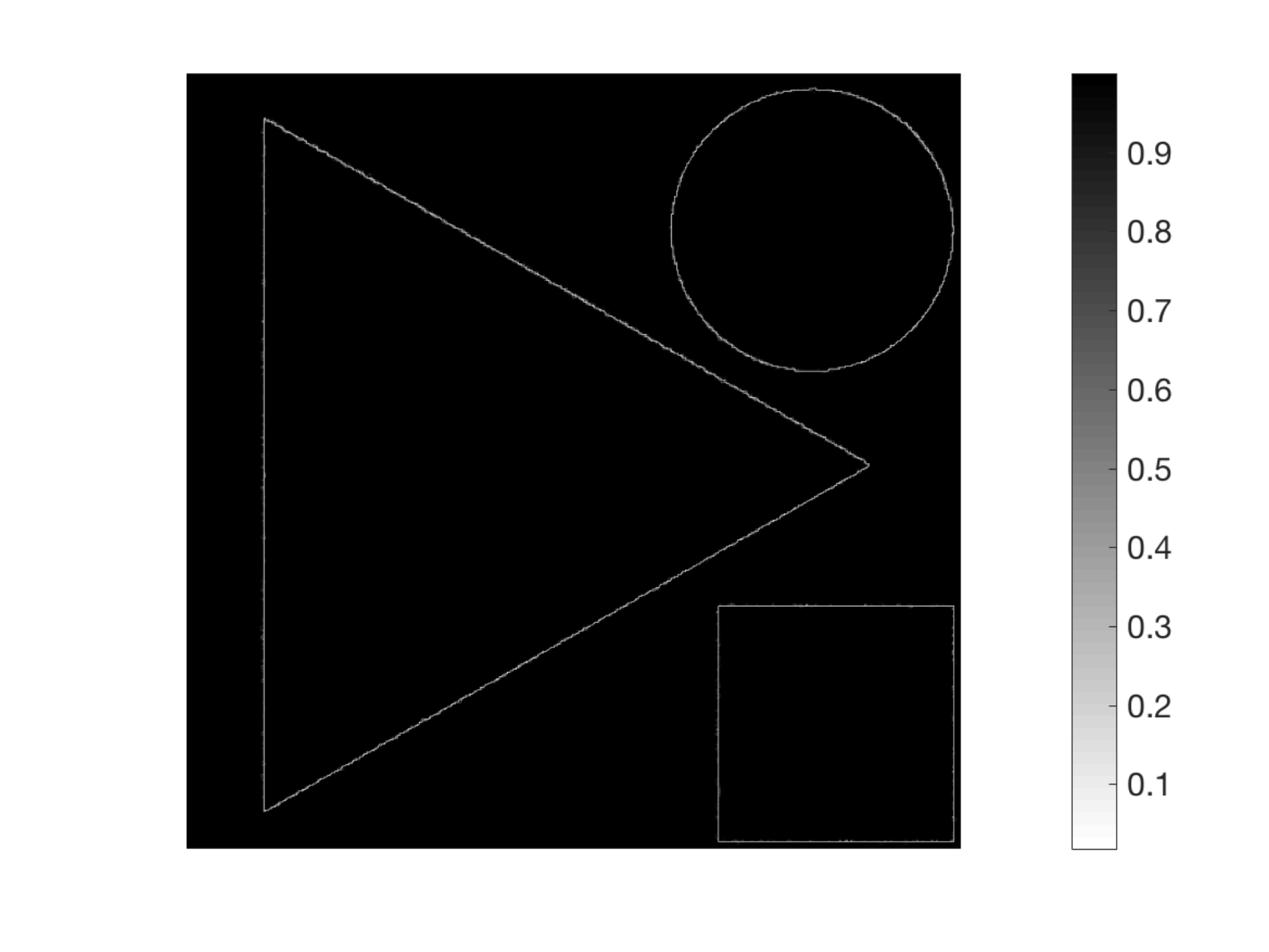}
\includegraphics[width=0.49\textwidth]{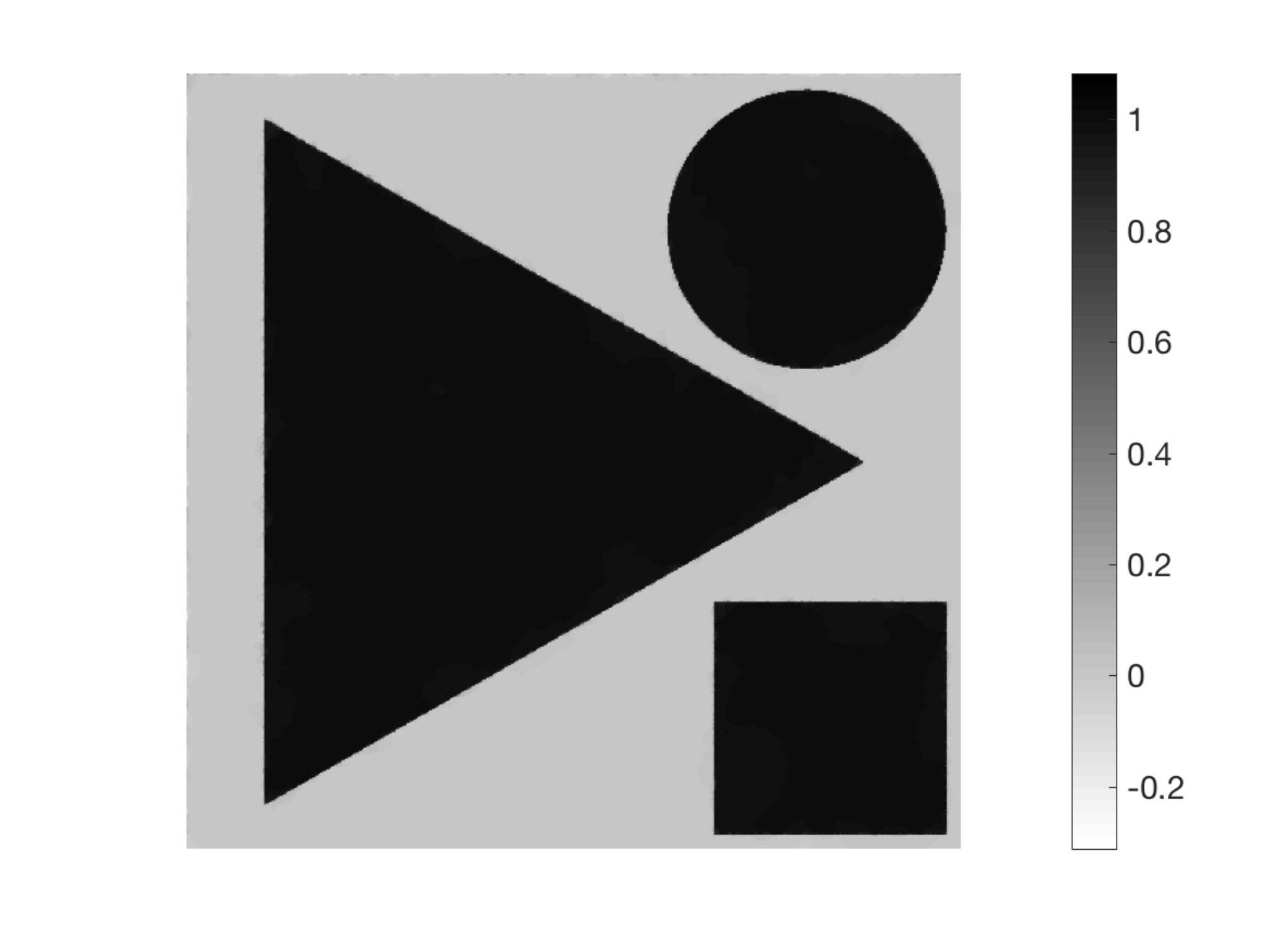}
\includegraphics[width=0.49\textwidth]{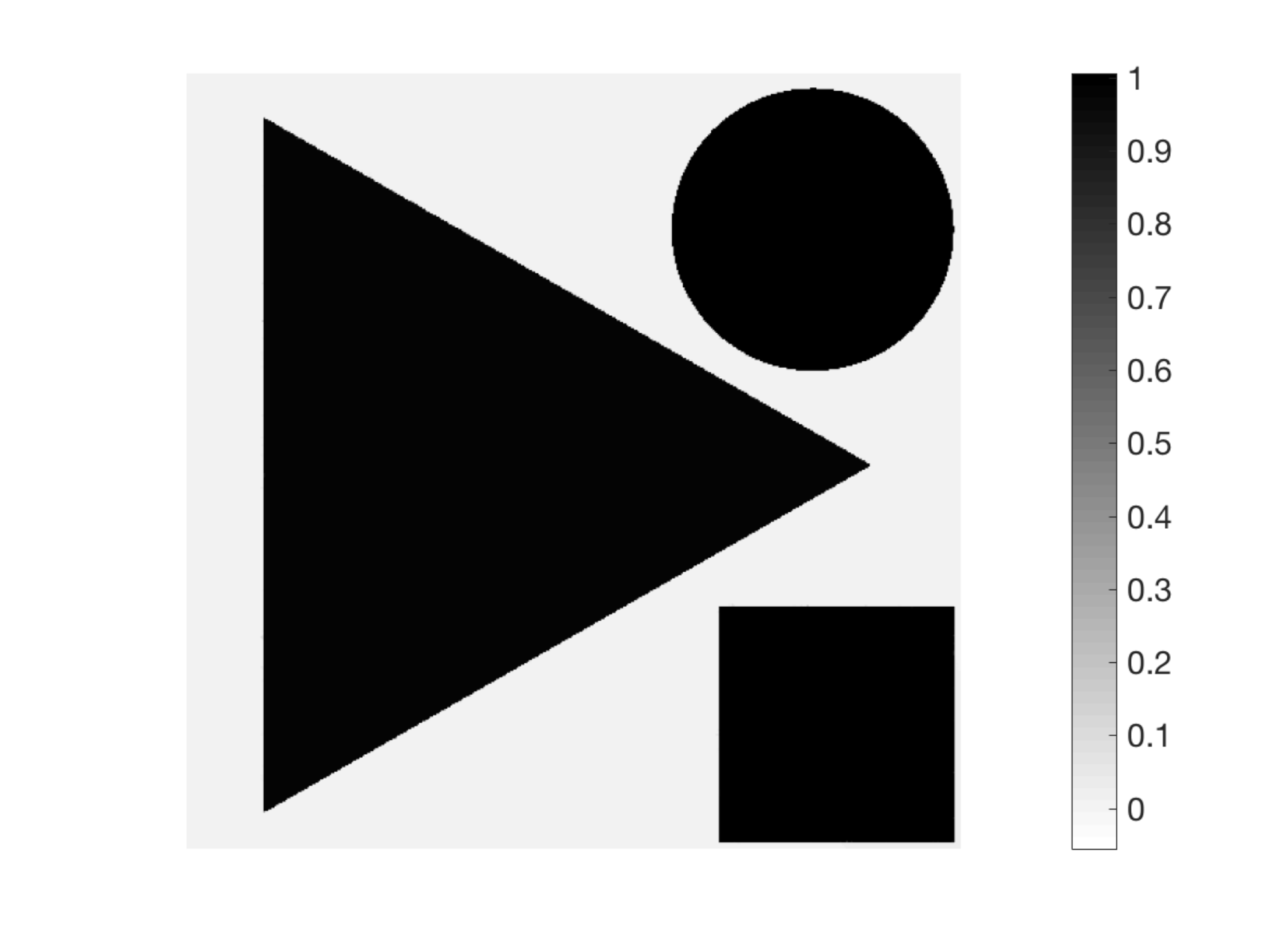}
\caption{\label{f:ex_1_sigma15}
Example 1 ($\sigma=0.15$). Top row (from left to right): original and noisy images, respectively. Middle row: $u_{\rm tv}$ (left) from Step~\ref{step1} in Algorithm~\ref{Algorithm} and the corresponding $s$  (right) from Step~\ref{step3}.
Bottom row: reconstruction using  total variation with optimized $\zeta>0$ (left) and our approach (right), respectively.}
\end{figure}

\begin{figure}[h!]
\centering
\hspace{-1.8cm}
\includegraphics[width=0.4\textwidth]{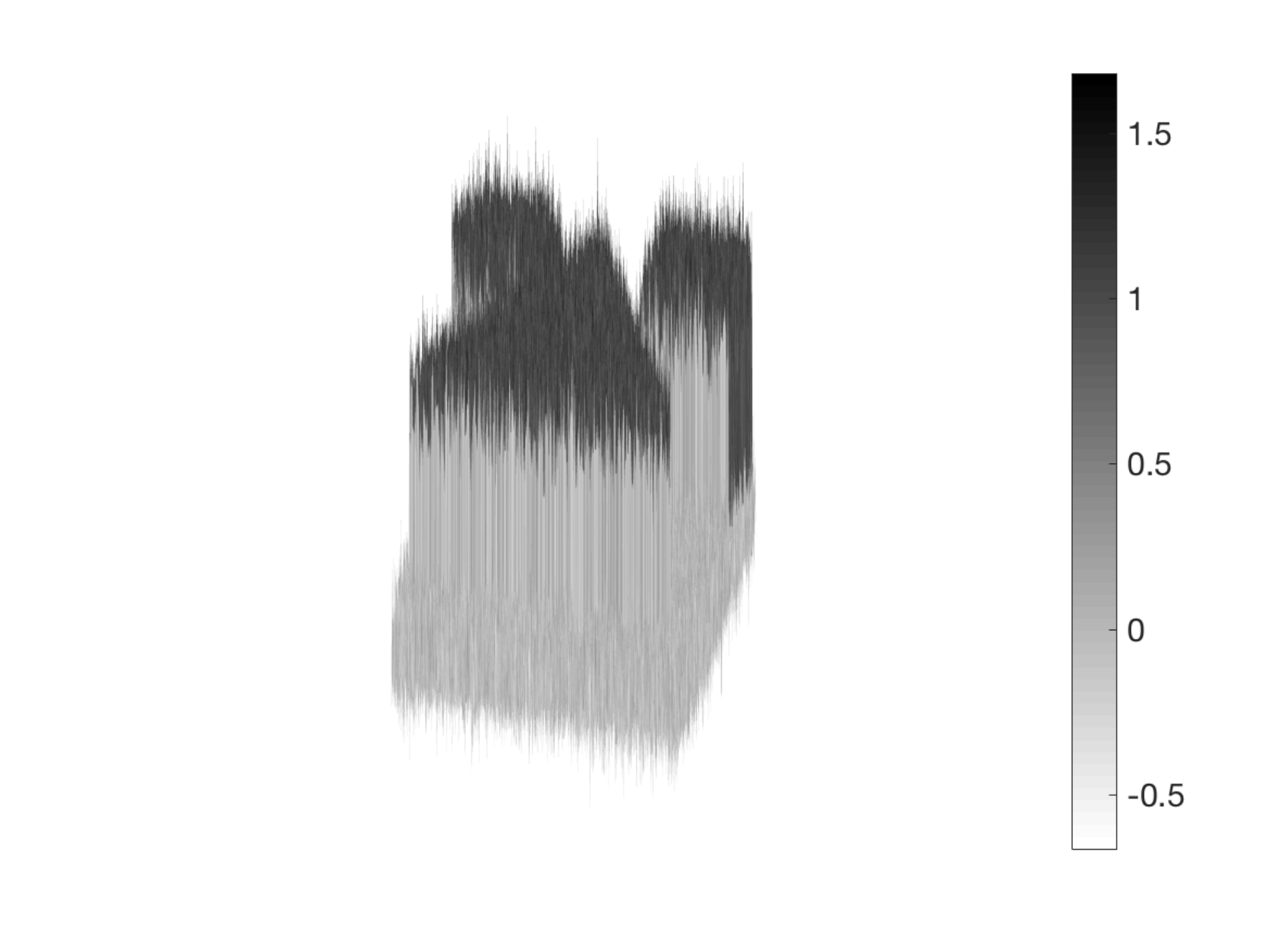}
\hspace{-0.59cm}
\includegraphics[width=0.4\textwidth]{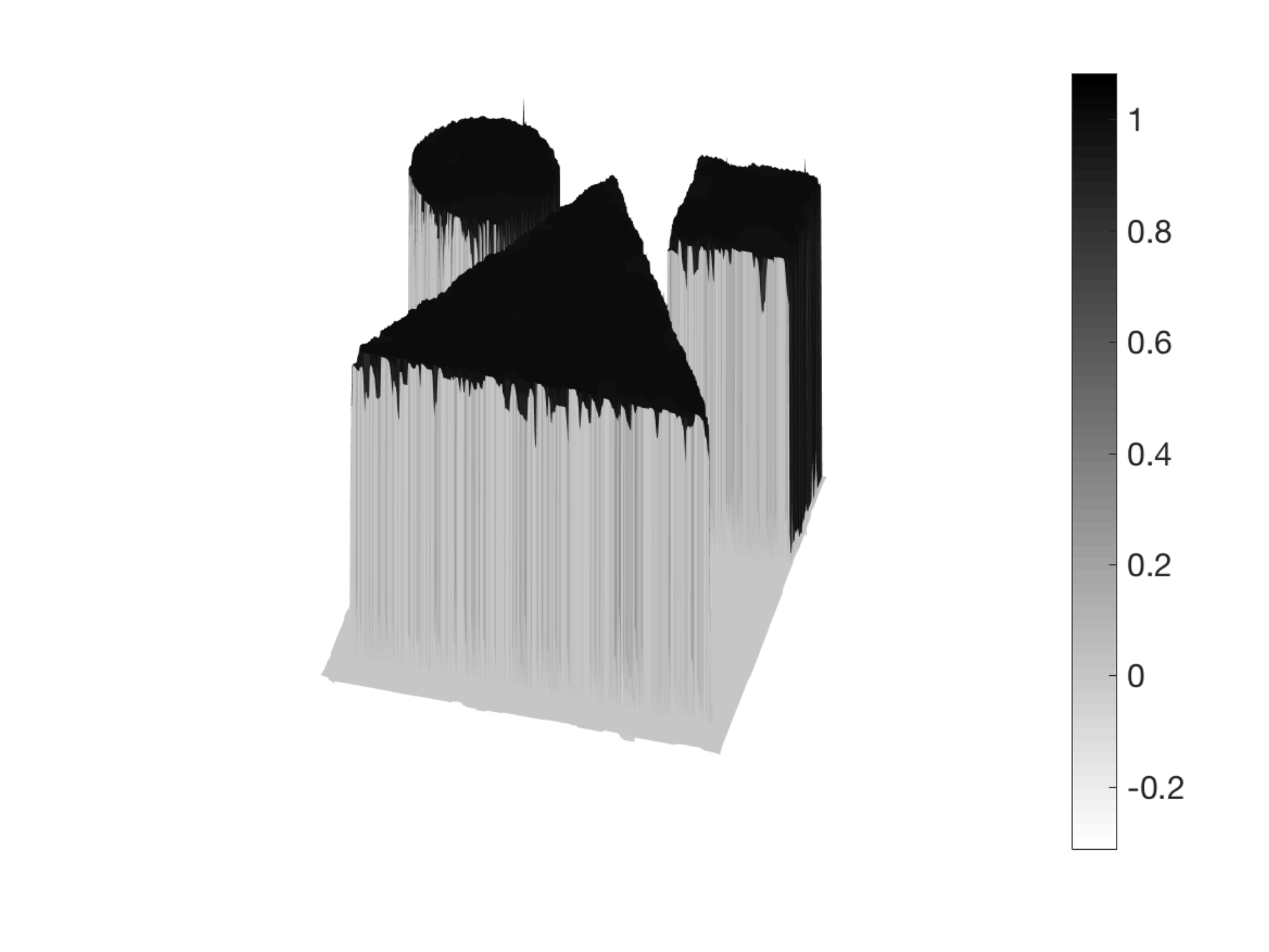}
\hspace{-0.62cm}
\includegraphics[width=0.4\textwidth]{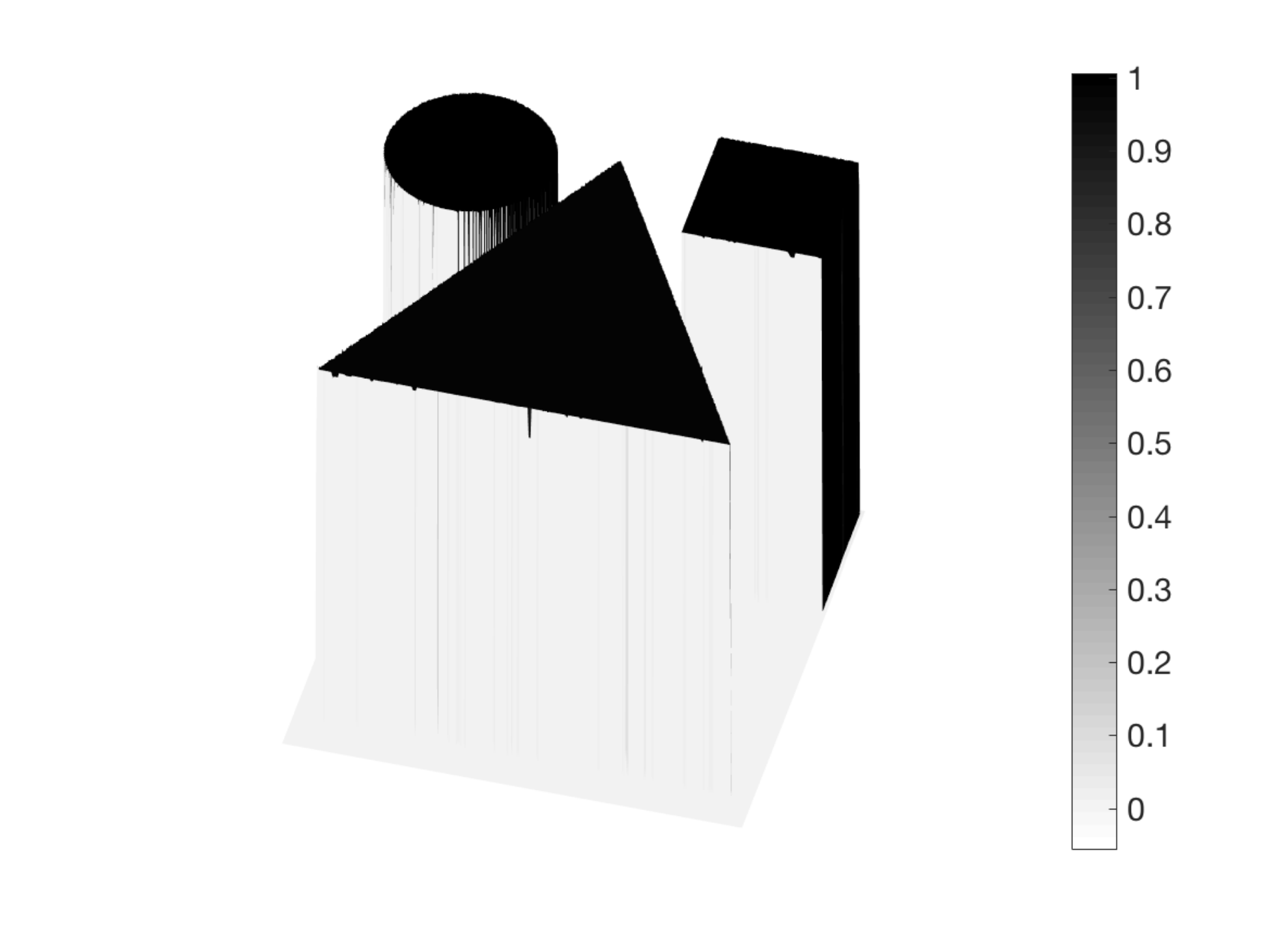}
\hspace{-0.8cm}
\caption{\label{f:ex_2_diffview_15} Example 1 ($\sigma = 0.15$). Left panel: noisy image. Middle panel: reconstruction
using Total Variation (TV) approach with optimal $\zeta>0$. Right panel: reconstruction using 
the our approach. TV tends to smooth out the edges and corners 
that are not aligned with the grid. On the other hand it is possible to have perfect 
recovery using our approach.}
\end{figure}

\subsection{Example 2: stripes}

In our second example we consider 6 stripes with intensities equal to 0.8, 0.7, 0.6, 0.5, 0.4 
and 0.35 (from left to right), respectively (cf. Figure~\ref{f:ex_2_sigma10} (top row left)). As in the previous example
we again consider two additive noise levels with mean 0 standard deviations $0.1$ and $0.15$ respectively. Our results are shown in Figures~\ref{f:ex_2_sigma10} and \ref{f:ex_2_sigma15} respectively. Here again we first find $S$ by using Algorithm~\ref{Algorithm} using the parameters
$\zeta = 0.2$, 
${\rm tol}_{\rm tv} = 10e$-4, 
$N_{\rm refine} = 8$,
$\lambda = 15$, 
$\beta = 0.99$,
$\nu = 100$. 
We further set $\mu = 2900$ in \eqref{eq:ext_var_trunc_disc}.
We then solve for $V$. We call $\tr V$ as our reconstruction. 
A comparison between PSNR and SSIM is shown in Table~\ref{t:ex_2}.
As we noticed in the previous example we again obtain visually almost
perfect reconstruction using our approach. 

\begin{table}[h!]
\centering
\begin{tabular}{|l|l|l|l|l|} \hline 
$\sigma$  &  PSNR (TV)  & PSNR (New) & SSIM (TV)   & SSIM (New)  \\ \hline
 0.1      &  3.6253e+01 & 2.7917e+01 & 9.0799e-01  &  9.4789e-01 \\ \hline 
 0.15     &  3.3616e+01 & 2.7419e+01 & 8.7028-01  &  9.3000e-01 \\ \hline
\end{tabular}
\caption{\label{t:ex_2}Example 2: PSNR and SSIM using two 
different standard deviations ($\sigma = 0.1$ and $\sigma = 0.15$) using
TV and proposed scheme (New).}
\end{table}

\begin{figure}[h!]
\centering
\includegraphics[width=0.49\textwidth]{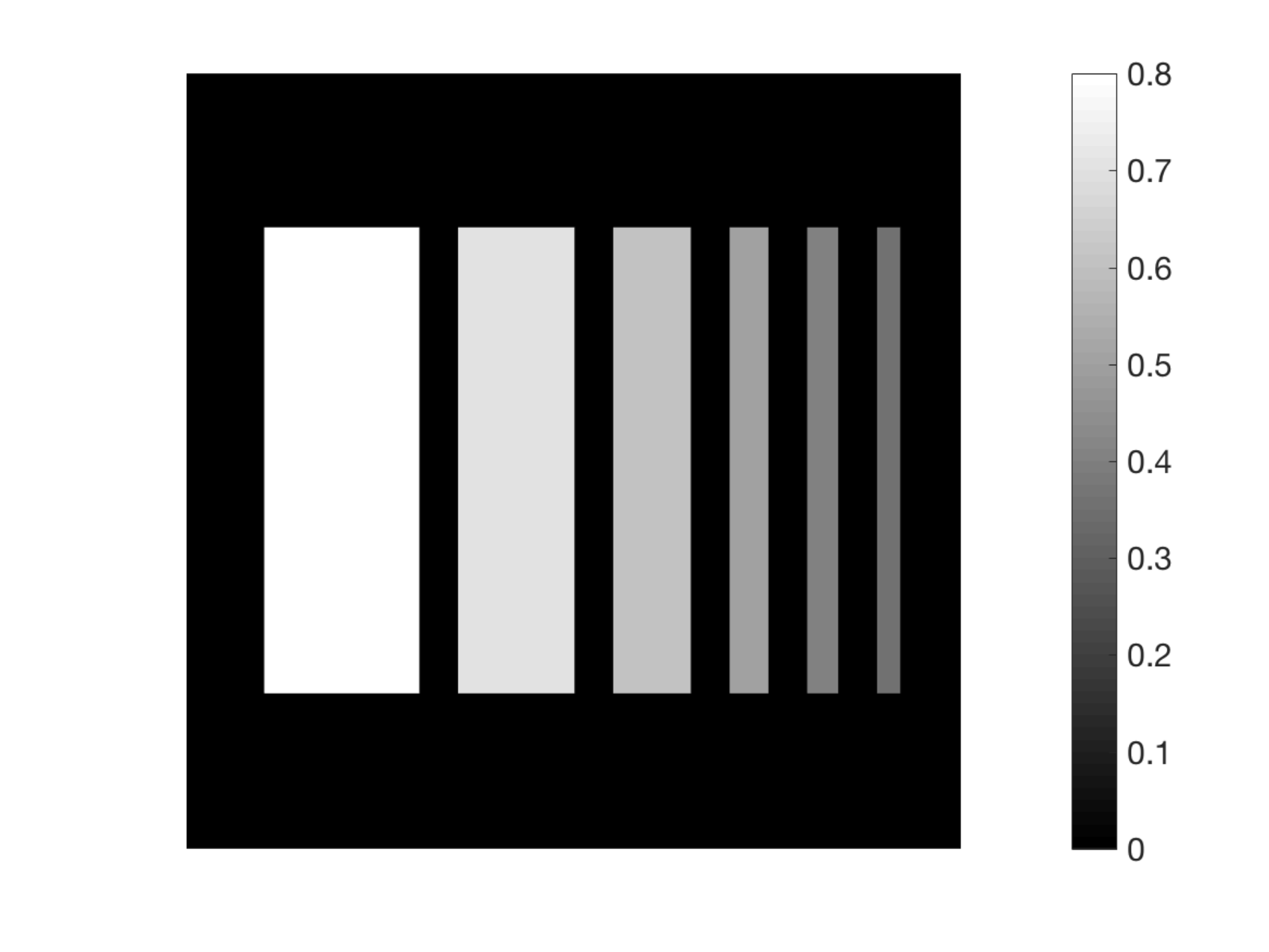}
\includegraphics[width=0.49\textwidth]{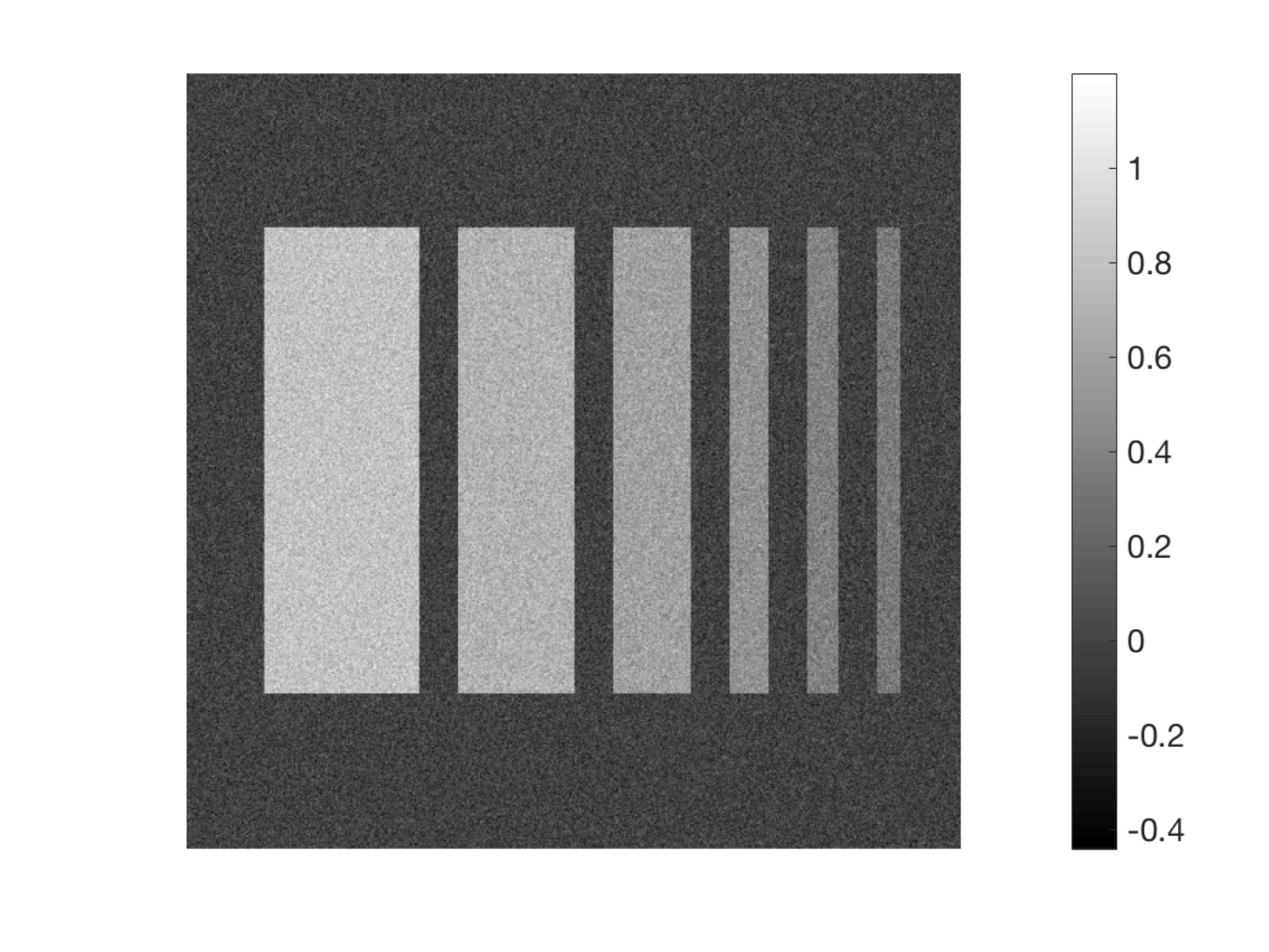}
\includegraphics[width=0.49\textwidth]{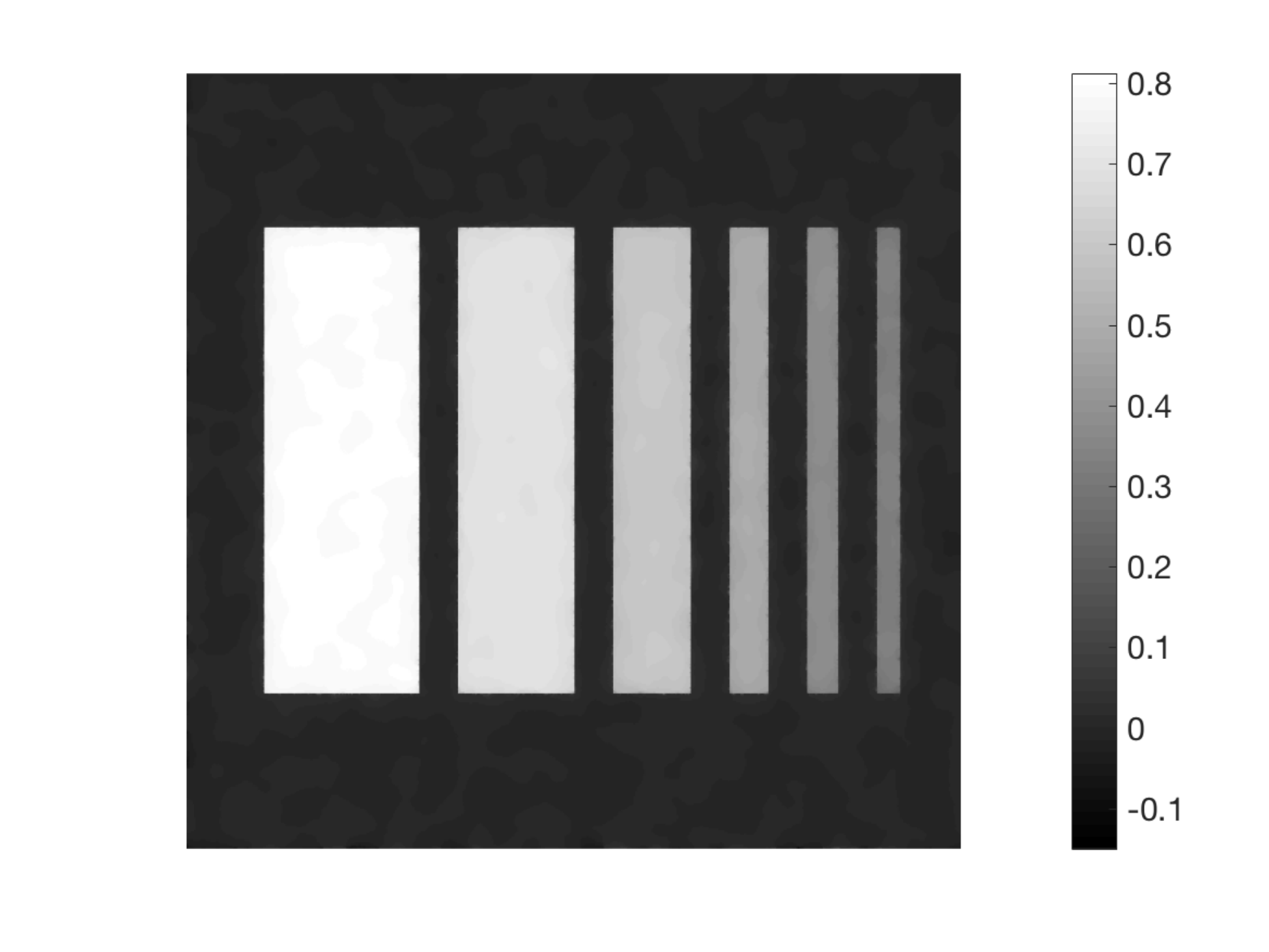}
\includegraphics[width=0.49\textwidth]{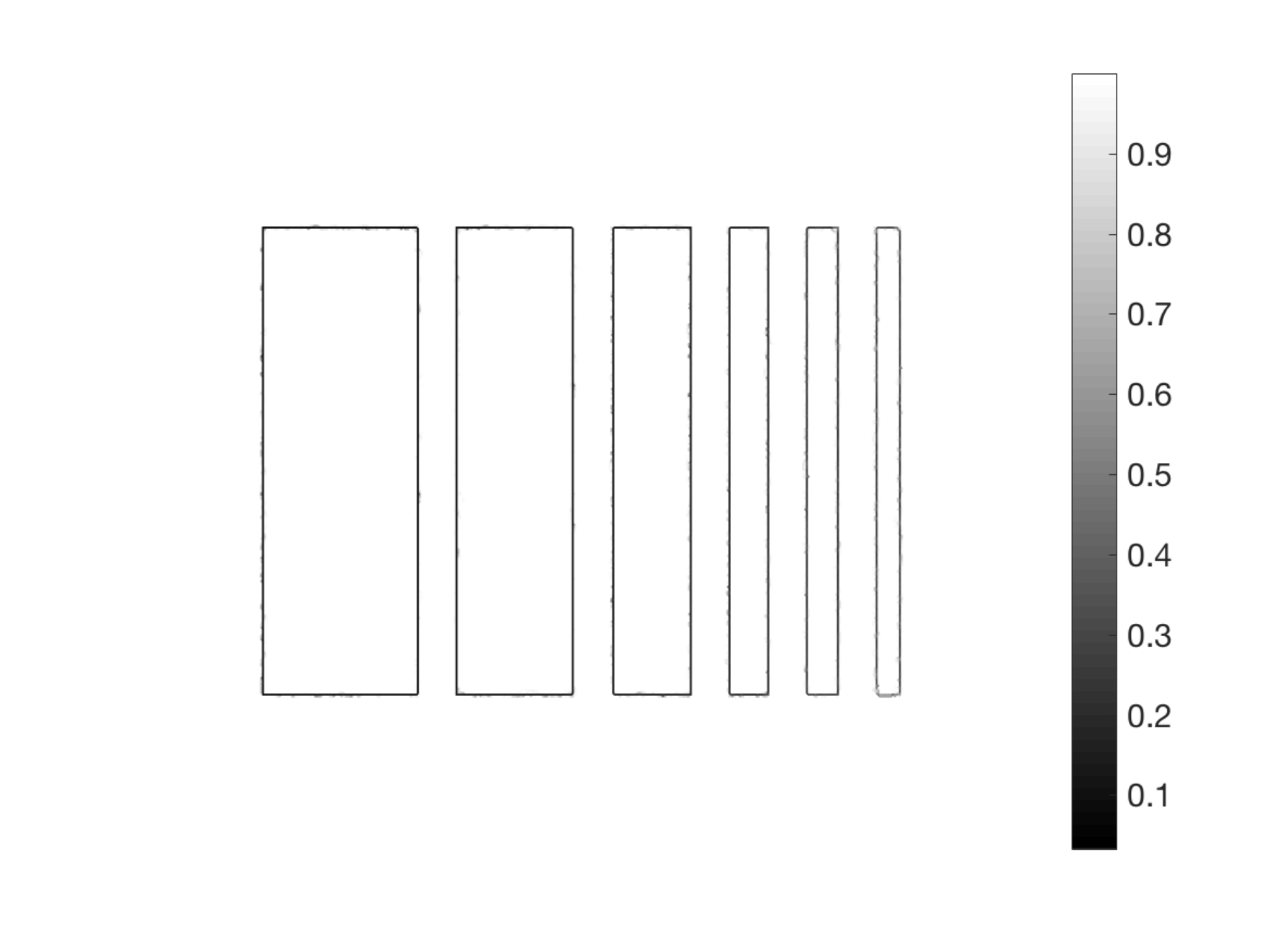}
\includegraphics[width=0.49\textwidth]{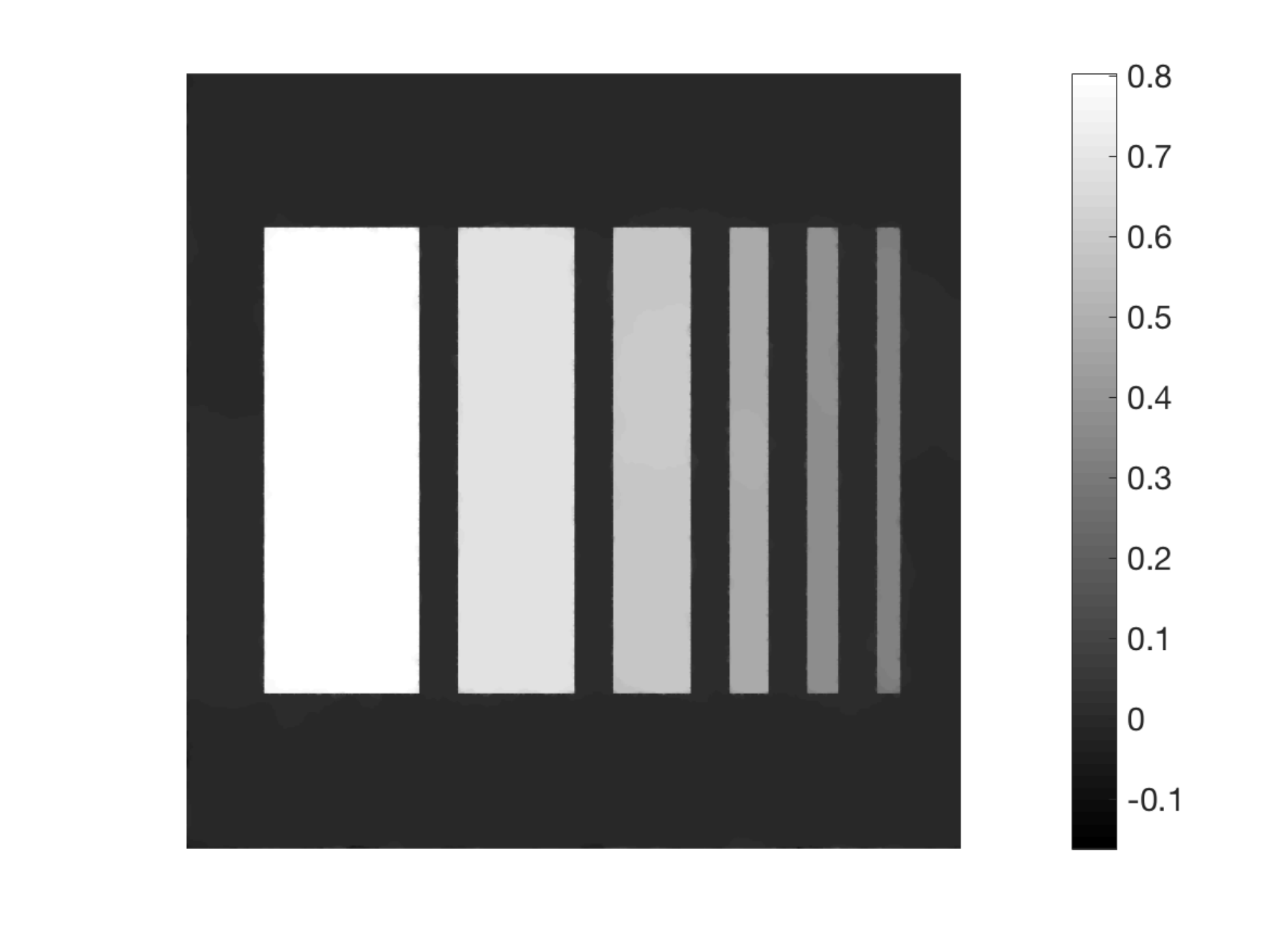}
\includegraphics[width=0.49\textwidth]{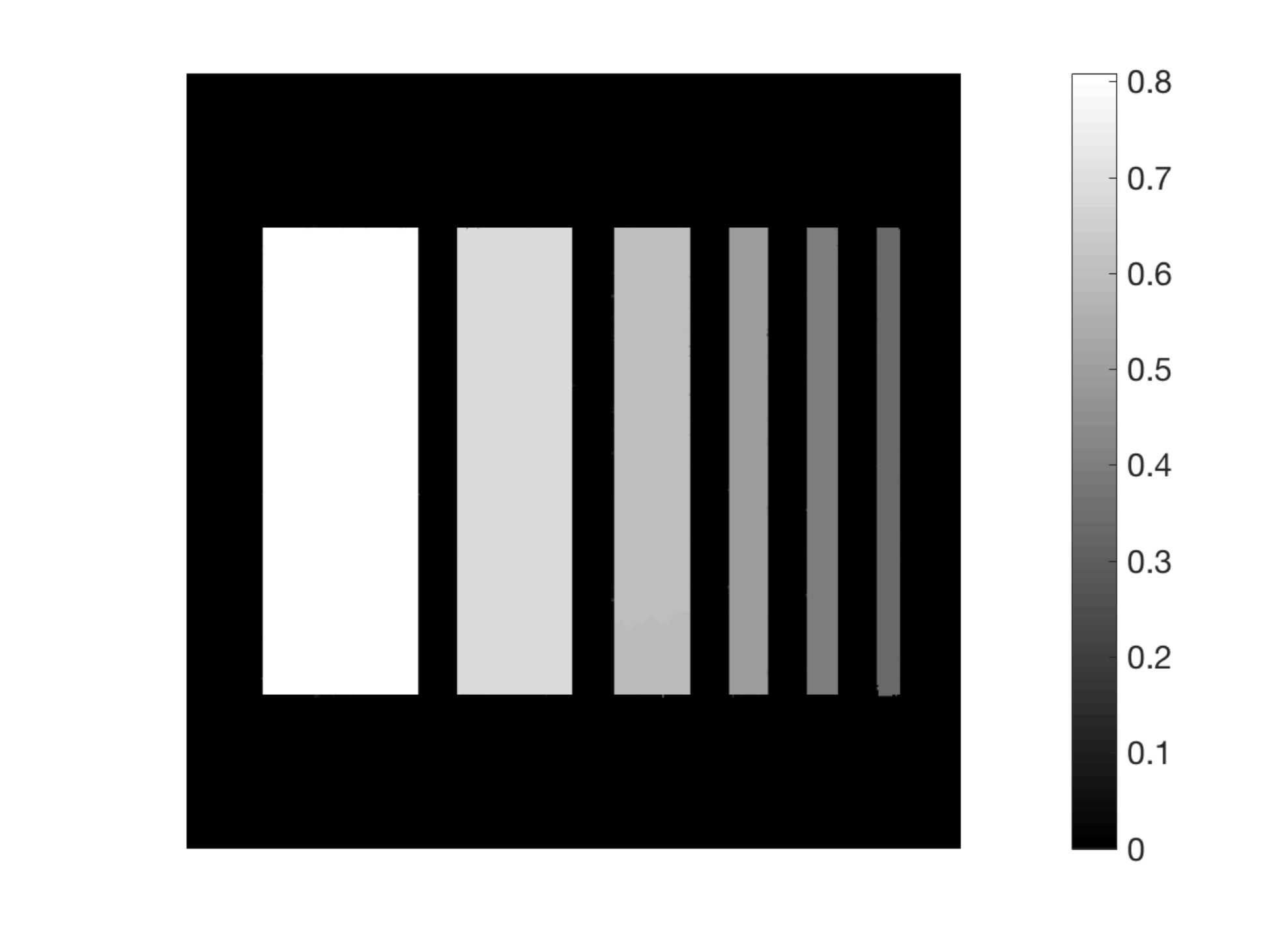}
\caption{\label{f:ex_2_sigma10}
Example 2 ($\sigma=0.1$). Top row (from left to right): original and noisy images, respectively. Middle row: $u_{\rm tv}$ (left) from Step~\ref{step1} in Algorithm~\ref{Algorithm} and the corresponding $s$  (right) from Step~\ref{step3}.
Bottom row: reconstruction using  total variation with optimized $\zeta>0$ (left) and our approach (right), respectively.}
\end{figure}

\begin{figure}[h!]
\centering
\includegraphics[width=0.49\textwidth]{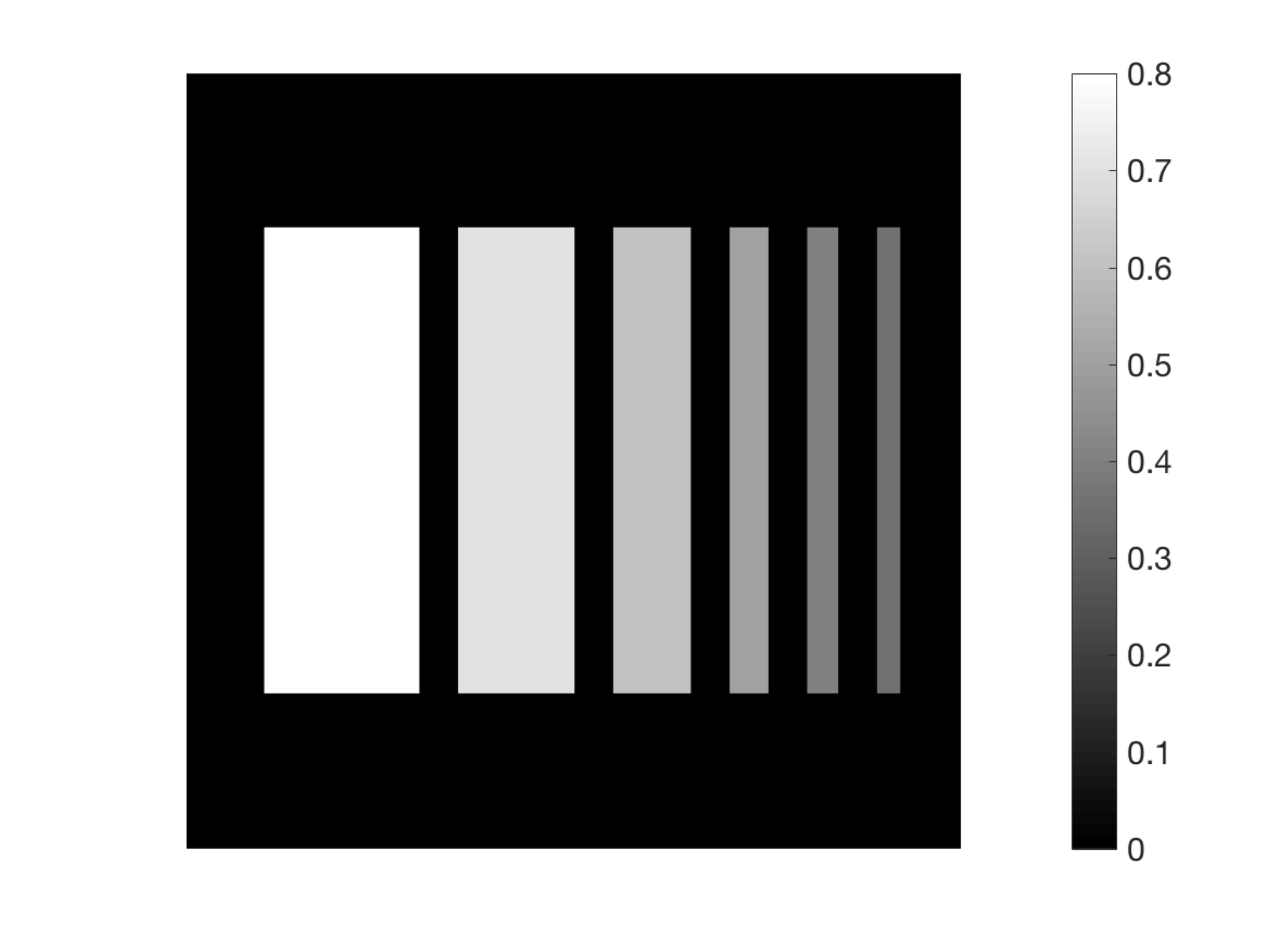}
\includegraphics[width=0.49\textwidth]{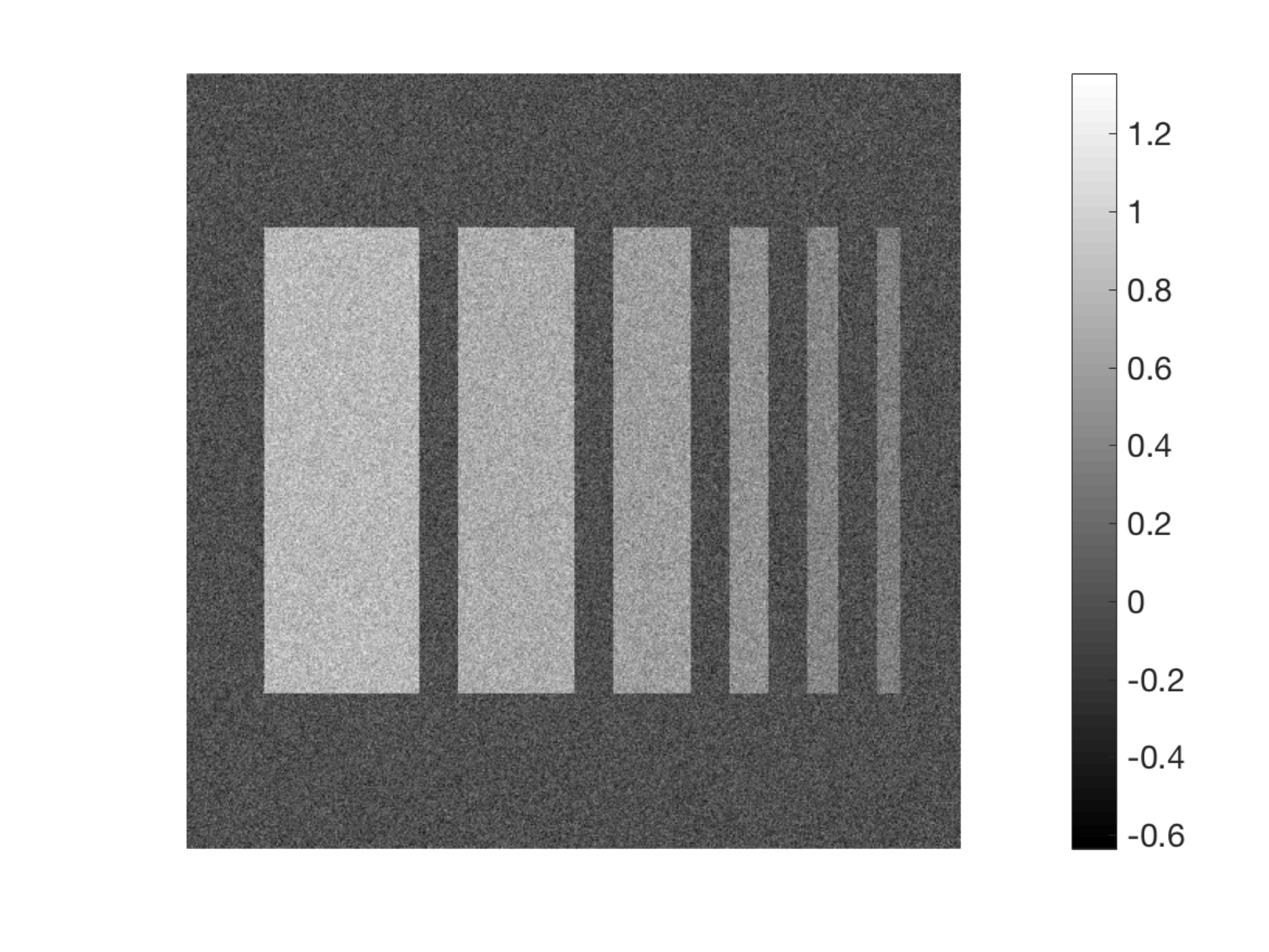}
\includegraphics[width=0.49\textwidth]{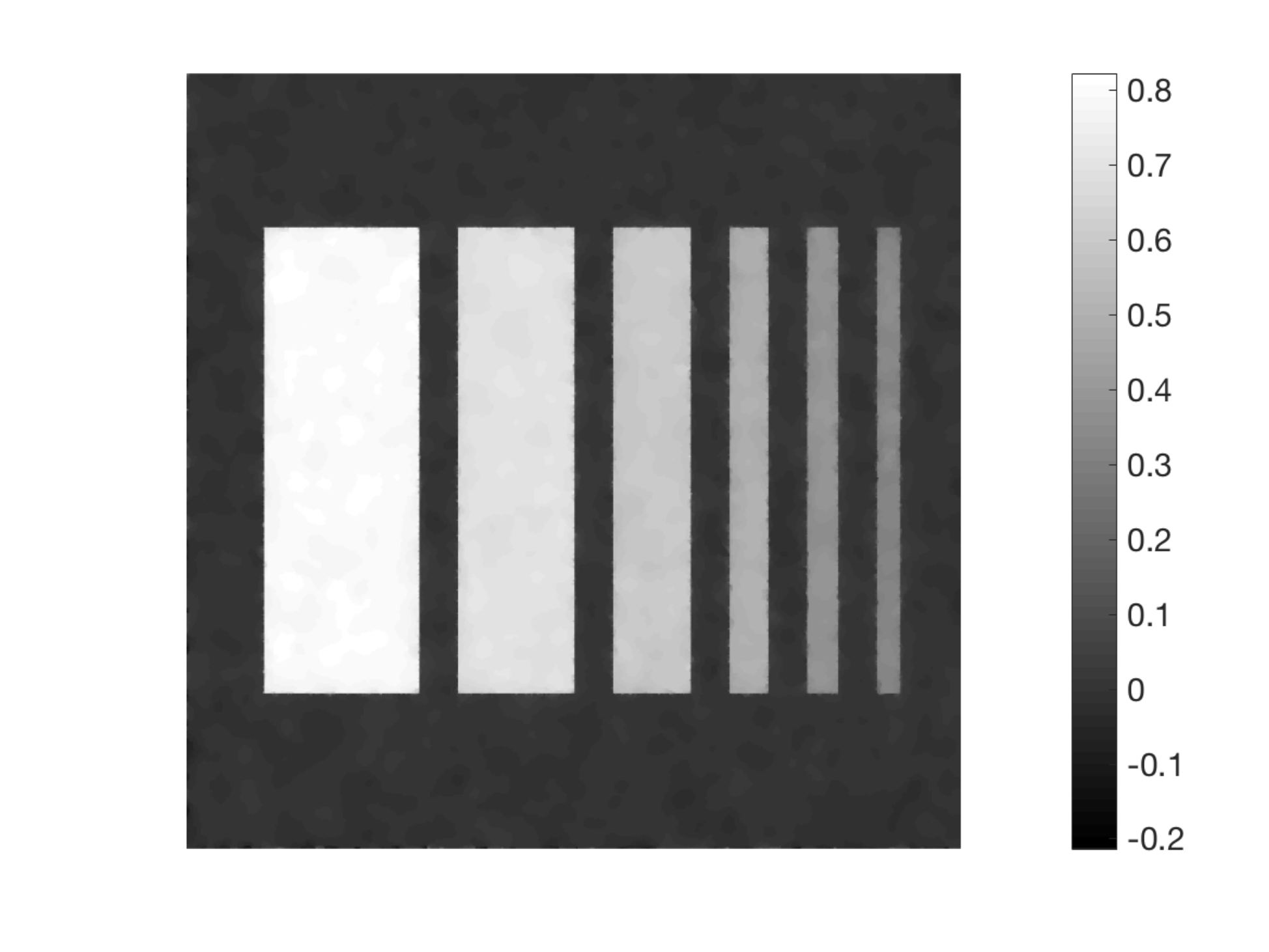}
\includegraphics[width=0.49\textwidth]{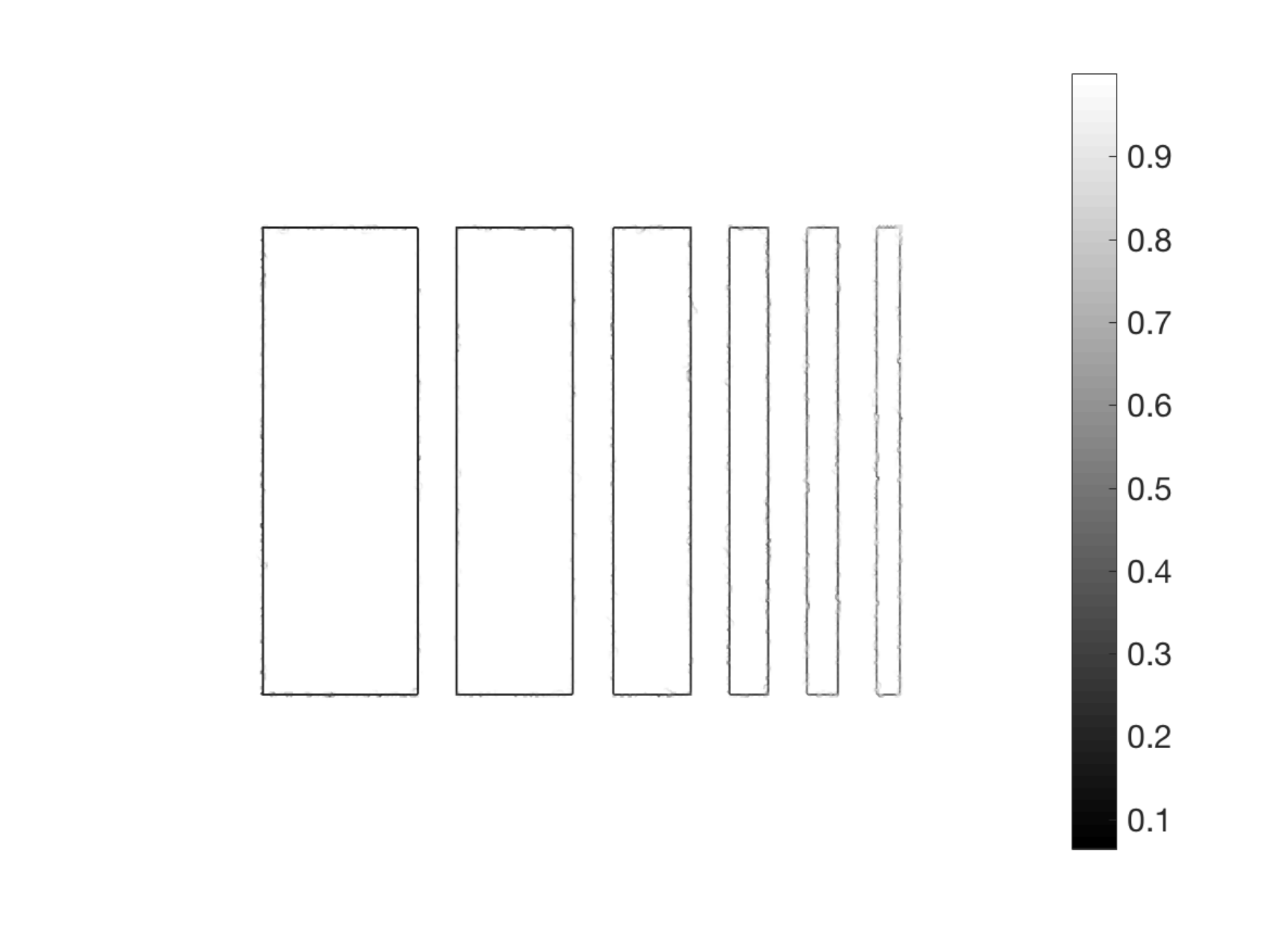}
\includegraphics[width=0.49\textwidth]{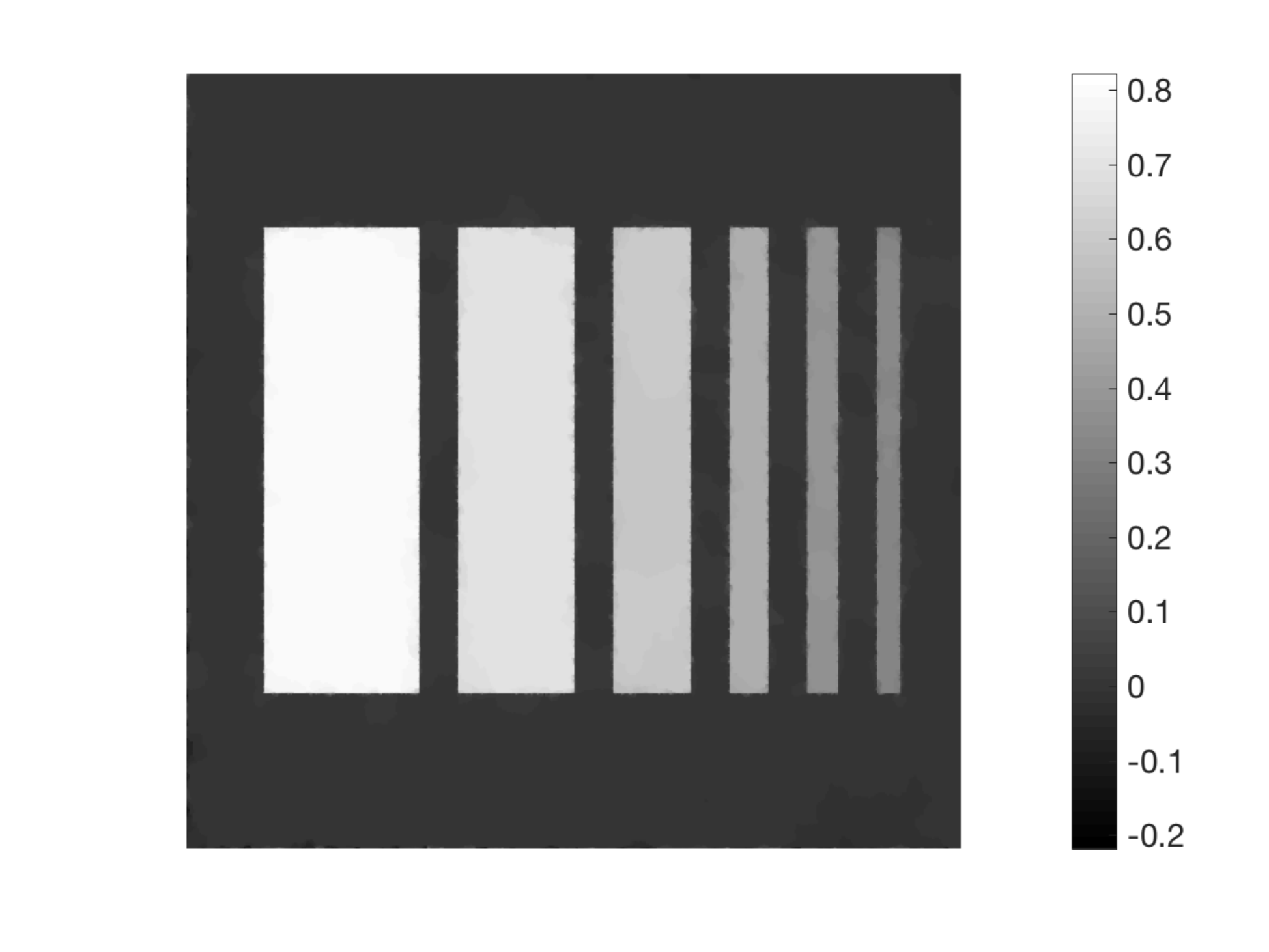}
\includegraphics[width=0.49\textwidth]{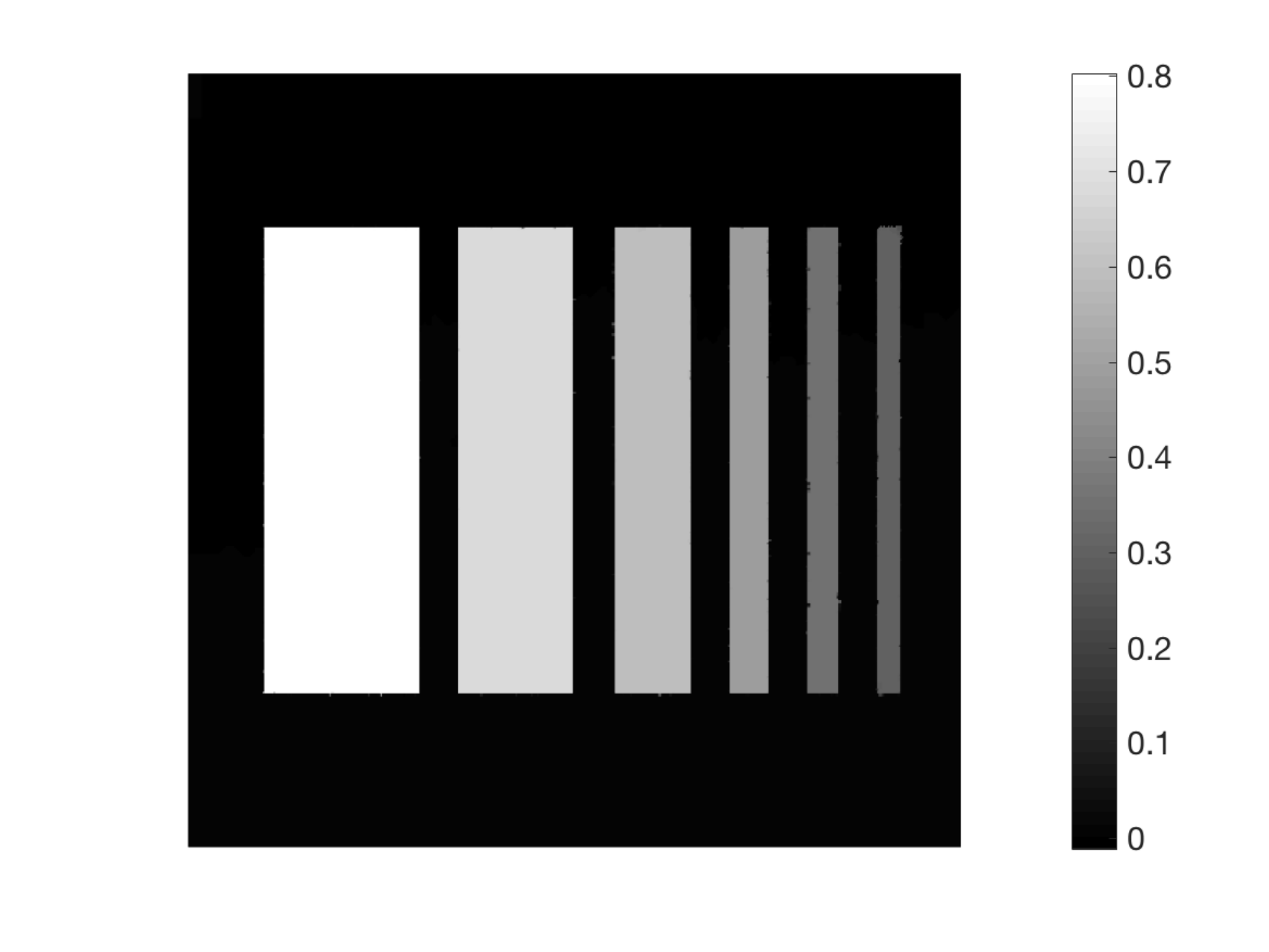}
\caption{\label{f:ex_2_sigma15}
Example 2 ($\sigma=0.15$). Top row (from left to right): original and noisy images, respectively. Middle row: $u_{\rm tv}$ (left) from Step~\ref{step1} in Algorithm~\ref{Algorithm} and the corresponding $s$  (right) from Step~\ref{step3}.
Bottom row: reconstruction using  total variation with optimized $\zeta>0$ (left) and our approach (right), respectively.}
\end{figure}

\subsection{Example 3: cameraman}

In the first two examples we considered synthetic images. In our
final example we consider a more realistic situation. We consider
a prototypical image of the cameraman (cf. Figure~\ref{f:ex_3_sigma10}). As in the previous examples 
we again consider two additive noise levels with standard deviations $0.1$ and $0.15$ respectively. Our results are shown in Figures~\ref{f:ex_3_sigma10} and \ref{f:ex_3_sigma15}, respectively

At first we apply the Algorithm~\ref{Algorithm} to find $S$. Here we have set the underlying parameters as 
$\zeta = 0.2$, 
${\rm tol}_{\rm tv} = 10e$-4, 
$N_{\rm refine} = 8$,
$\lambda = 0.7$, 
$\beta = 0.99$,
$\nu = 20$.
We further set $\mu = 10^4$ in \eqref{eq:ext_var_trunc_disc}.
We then solve for $V$ and we call $\tr V$ as our reconstruction. 
A comparison between PSNR and SSIM is shown in Table~\ref{t:ex_3}.
As we noticed in the previous examples we again obtain better reconstructions using our approach.

\section{Conclusion and further directions}
A new variational model associated to the fractional Laplacian was introduced. In particular, we have identified a weighted Sobolev space with respect to the weight $w=y^{1-2s(x)}$ appropriate for the treatment of the problem. We have shown that in general these weights are not of Muckenhoupt type, and have established that the trace space embeds in an $s-$weighted Lebesgue space. We have provided a discretization method for the full problem, and an algorithm for its resolution that also builds a selection procedure for $s$. The full scheme is advantageous when it comes to recovery of discontinuous features, details, homogeneous regions, and also for contrast preservation, in data perturbed by additive noise.

Future research directions are multiple, we enumerate some of them. 

1) The study of the optimization problem \eqref{1st}, seems to be the first step for the identification of a possible definition for $(-\Delta)^{s(x)}$. Such a task does not seem directly approachable via spectral or functional calculus points of view.

2) Full characterization of the trace space for $H$ and $H_0$. We have identified the embedding of $\tr H $ into $L^2(\Omega;s(x)^2)$, but it would be of interest to understand the Sobolev regularity of $\tr H$.

3) Differentiability and stability properties of the solution to \eqref{1st} with respect to $s$. For the optimal selection of $s$,  
it is required to identify a topology over an admissible set for $s$, so that solutions to \eqref{1st}, are stable with respect to perturbations. This seems like a complex task where the usual obstacles from homogenization and convergence of differential operators are present. In addition, difficulty is increased as the solutions to \eqref{1st} belong to a state space depending on $s$ as well. The differentiability issue is further more complex, but such a study will be the first step in establishing stationarity systems useful for implementation in the optimal selection of $s$.

4) The extension of the presented methods to a general class of inverse problems. Problem \eqref{1st}, admits the following generalization
\begin{equation*}
\min_{u\in H(\mathcal{C}; y^{1-2s(x)})} \int_{\C}y^{1-2s(x)} \left(\theta |u|^2 + |\nabla u|^2 \right) \dif x\dif y+ \frac{\mu}{2}\int_{\Omega} s(x)^2 |K(\tr u) - f| ^2  \dif x,
\end{equation*}
where $K$ is a bounded linear operator on $L^2(\Omega; s(x)^2)$. This would allow to deal with the problem of finding $y$ such that $Ky=f$. In this case, the choice of $s$ could be simply associated with regions of the domain $\Omega$ where it is more important to recover $y$ more accurately.

\begin{table}[h!]
\centering
\begin{tabular}{|l|l|l|l|l|} \hline 
$\sigma$  &  PSNR (TV)  & PSNR (New) & SSIM (TV)   & SSIM (New)  \\ \hline
 0.1      & 2.7054e+01  & 2.9640e+01 & 8.0475e-01  & 8.3653e-01  \\ \hline 
 0.15     & 2.4376e+01  & 2.6884e+01 & 7.4340e-01  & 7.9335e-01  \\ \hline
\end{tabular}
\caption{\label{t:ex_3}Example 3:PSNR and SSIM using two 
different standard deviations ($\sigma = 0.1$ and $\sigma = 0.15$) using
TV and proposed scheme (New).}
\end{table}

\begin{figure}[h!]
\centering
\includegraphics[width=0.49\textwidth]{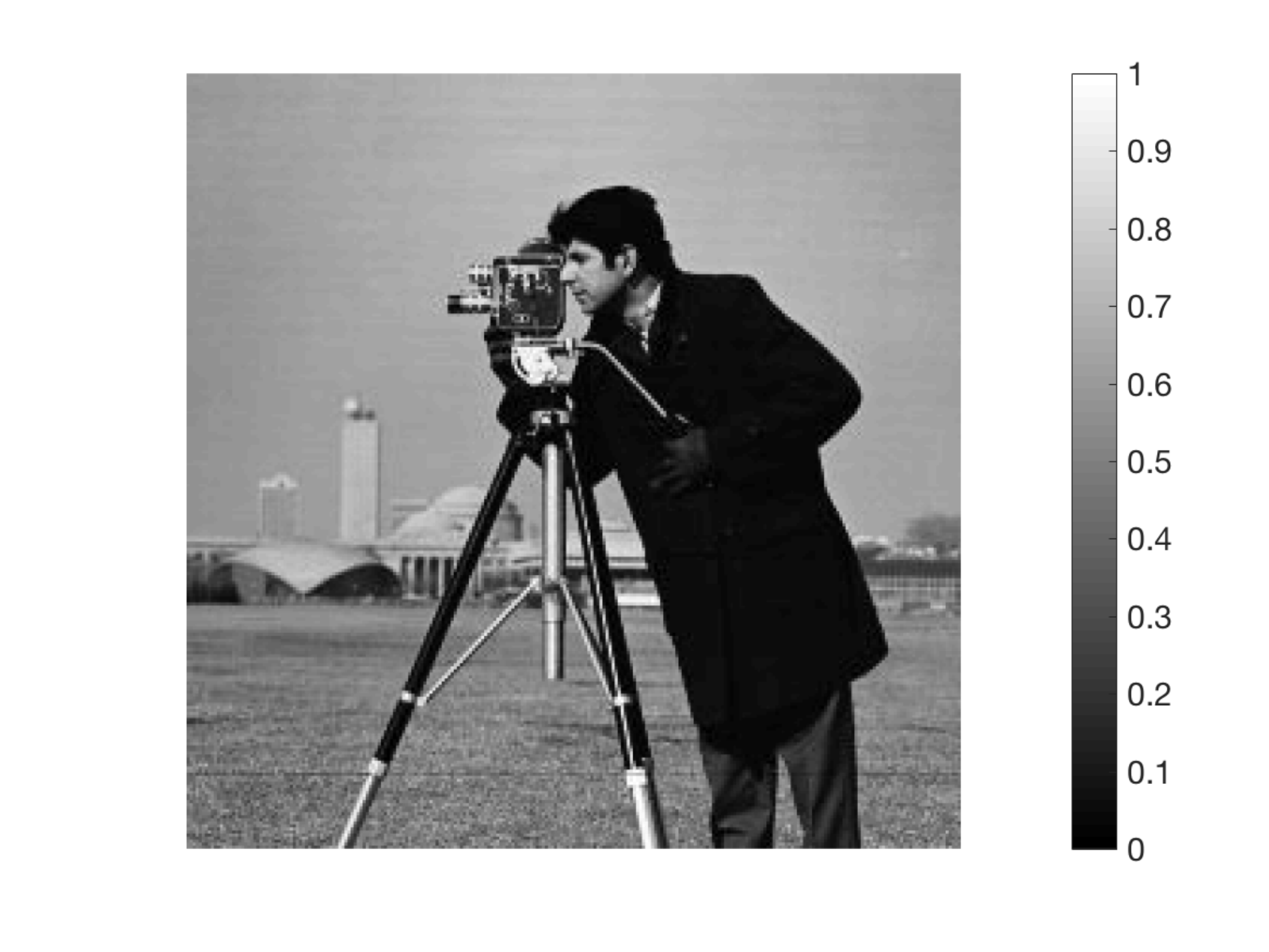}
\includegraphics[width=0.49\textwidth]{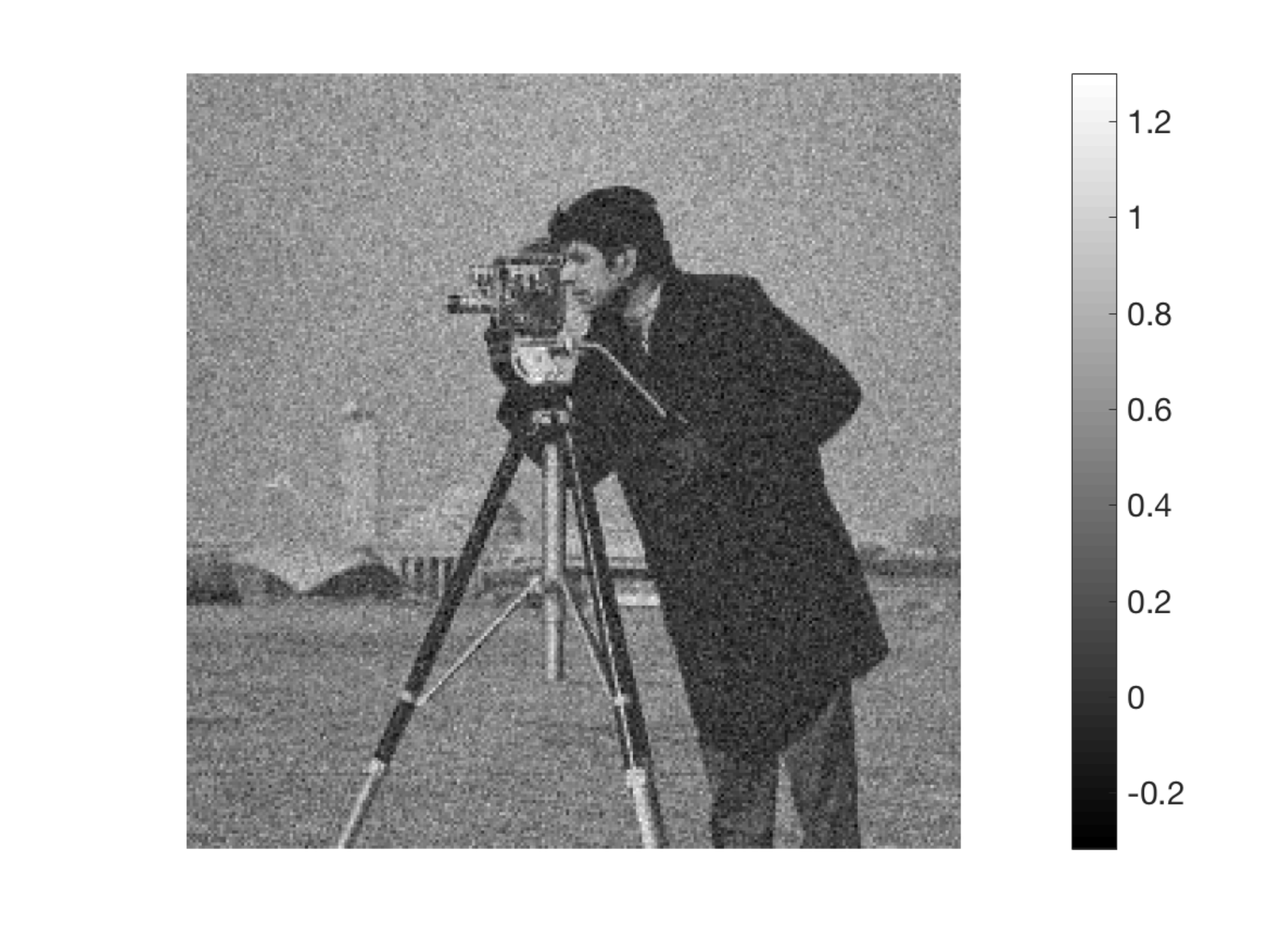}
\includegraphics[width=0.49\textwidth]{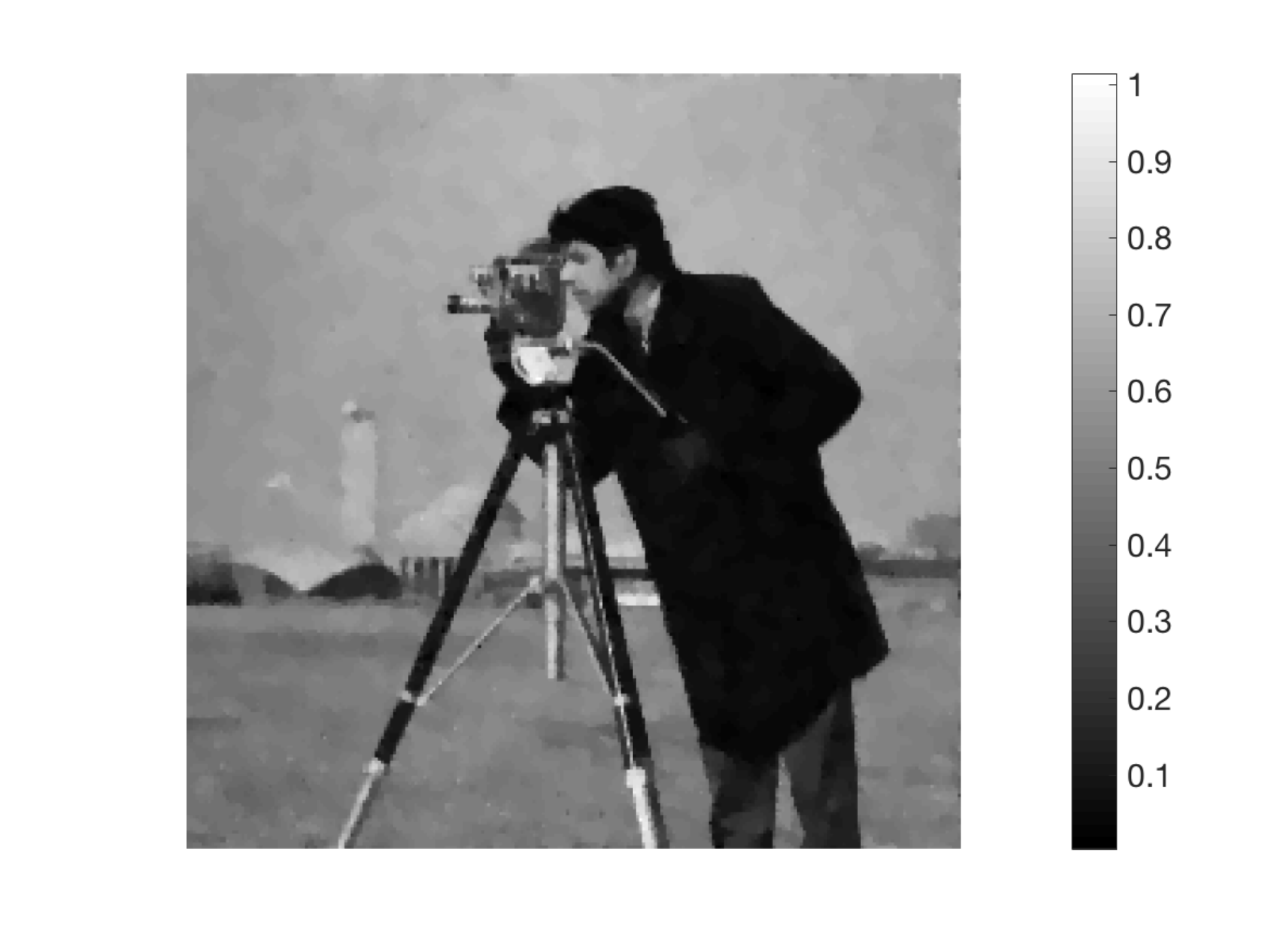}
\includegraphics[width=0.49\textwidth]{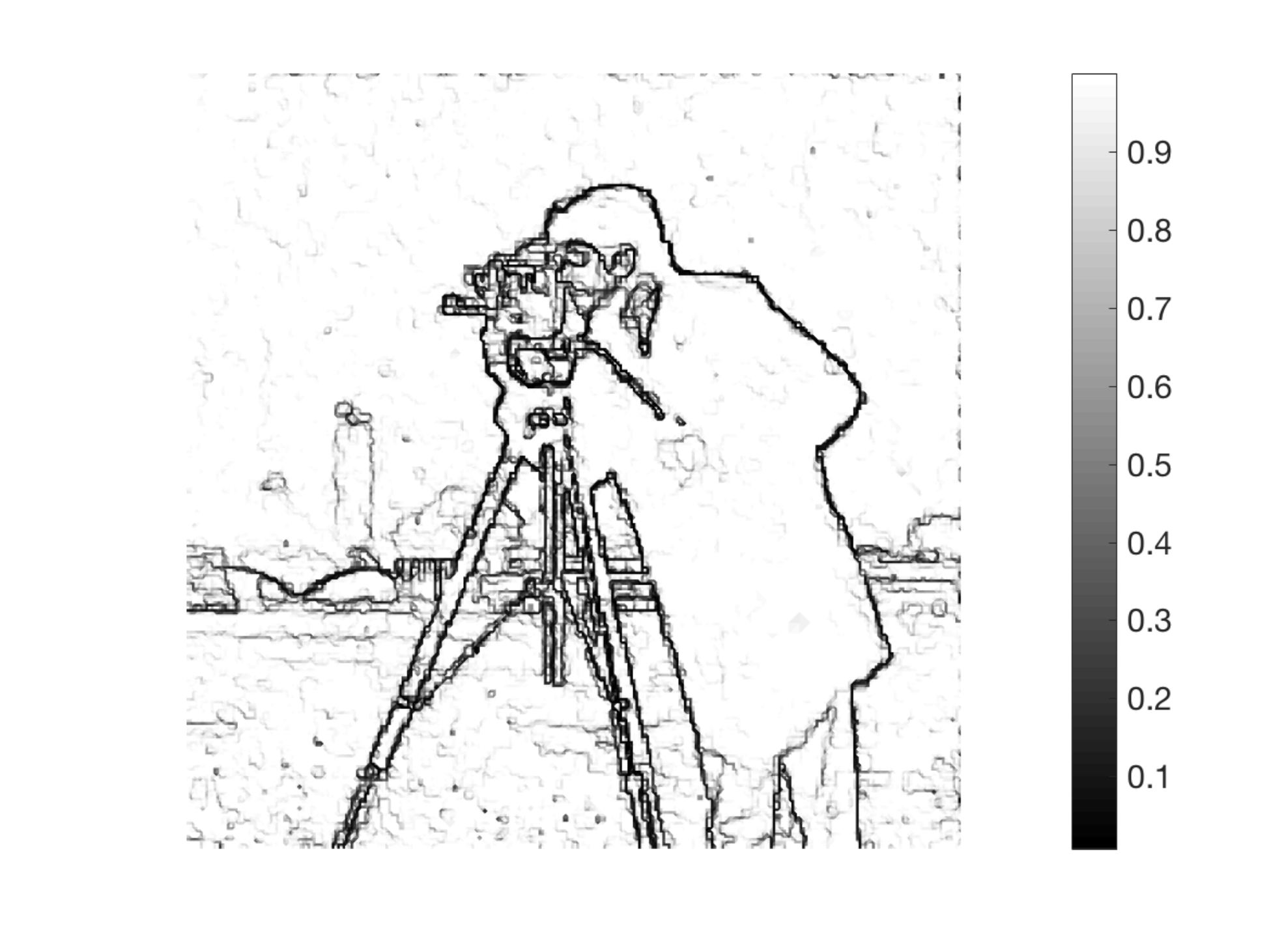}
\includegraphics[width=0.49\textwidth]{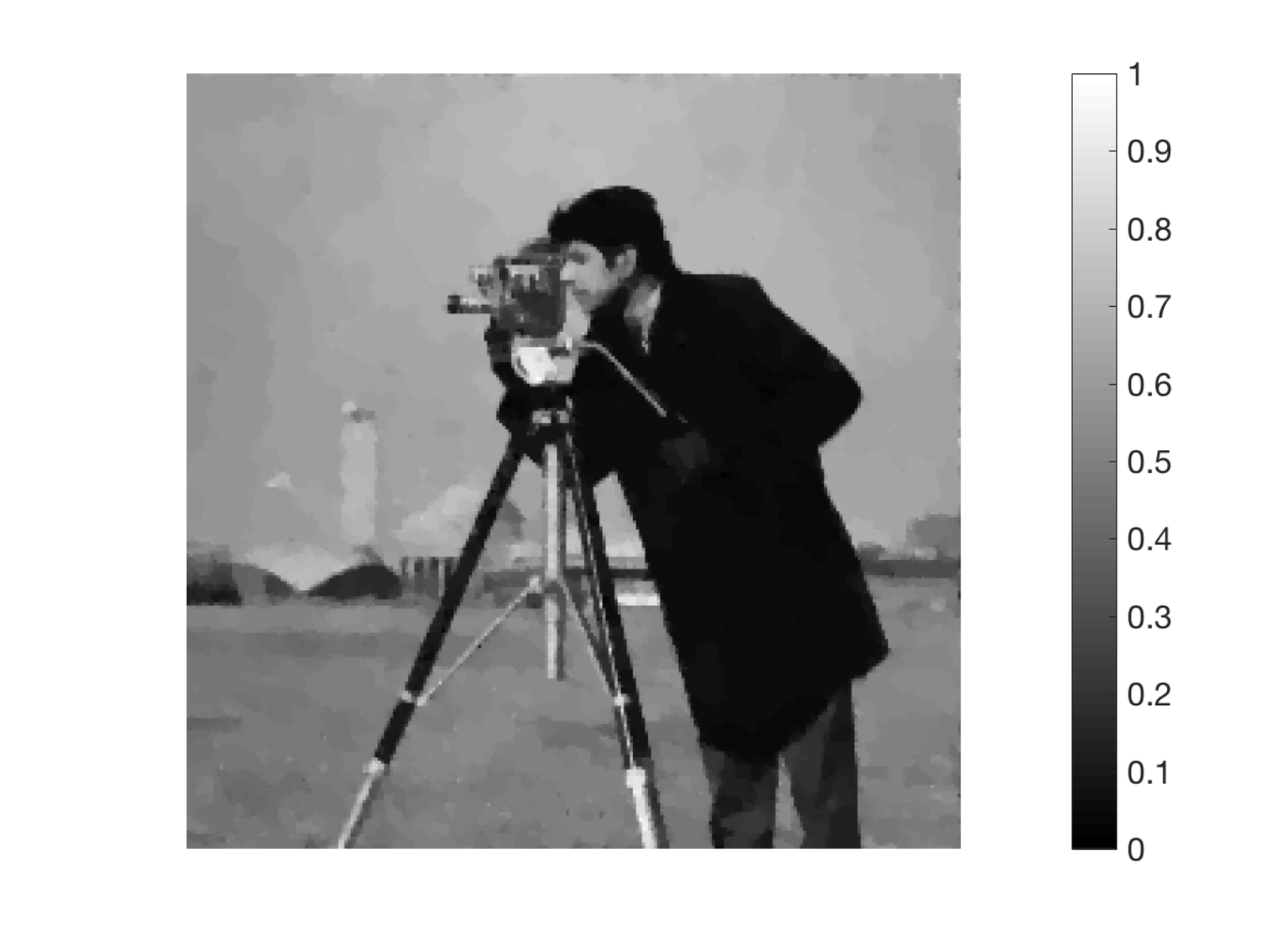}
\includegraphics[width=0.49\textwidth]{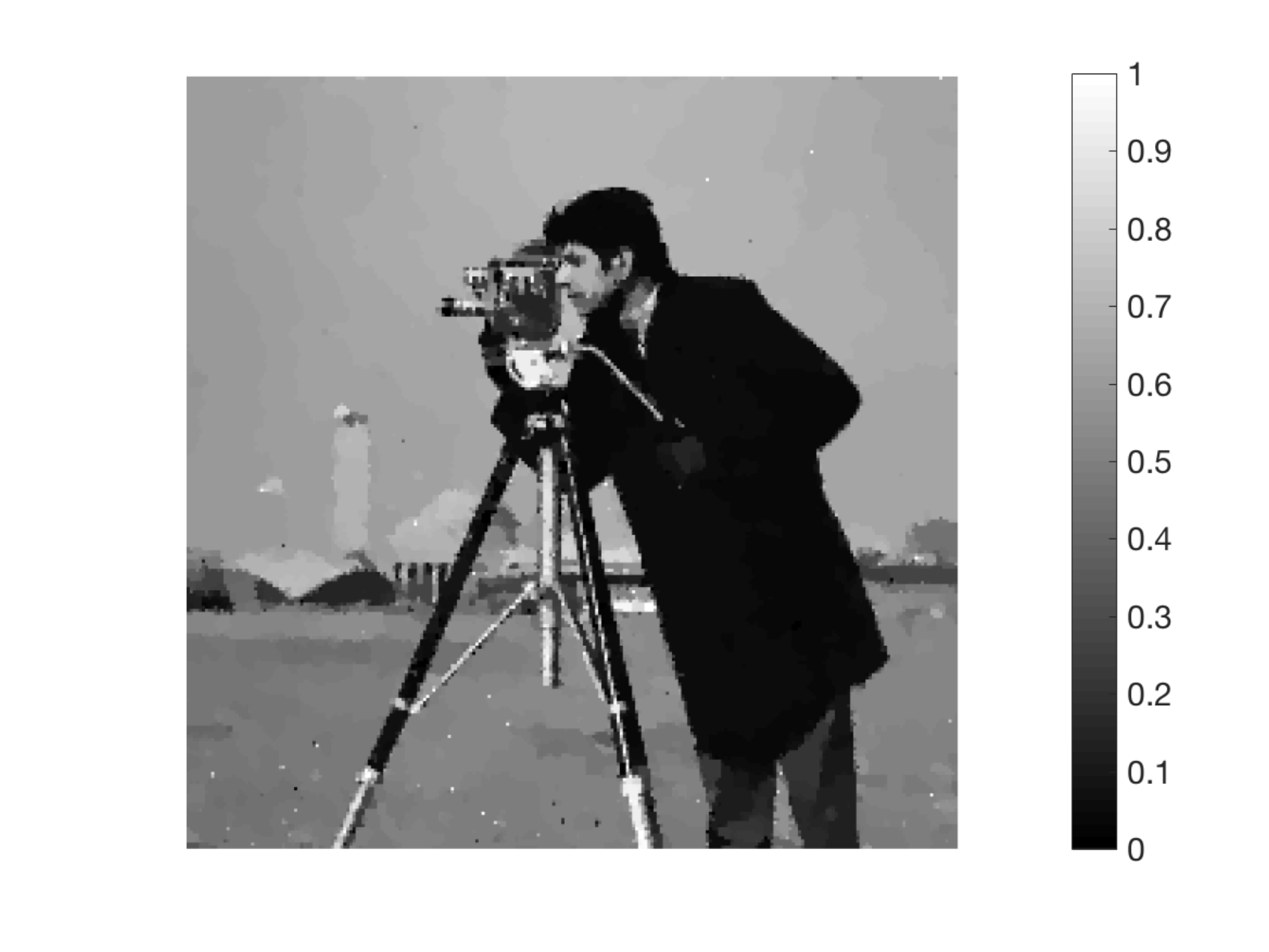}
\caption{\label{f:ex_3_sigma10}
Example 3 ($\sigma=0.1$). Top row (from left to right): original and noisy images, respectively. Middle row: $u_{\rm tv}$ (left) from Step~\ref{step1} in Algorithm~\ref{Algorithm} and the corresponding $s$  (right) from Step~\ref{step3}.
Bottom row: reconstruction using  total variation with optimized $\zeta>0$ (left) and our approach (right), respectively.}
\end{figure}

\begin{figure}[h!]
\centering
\includegraphics[width=0.49\textwidth]{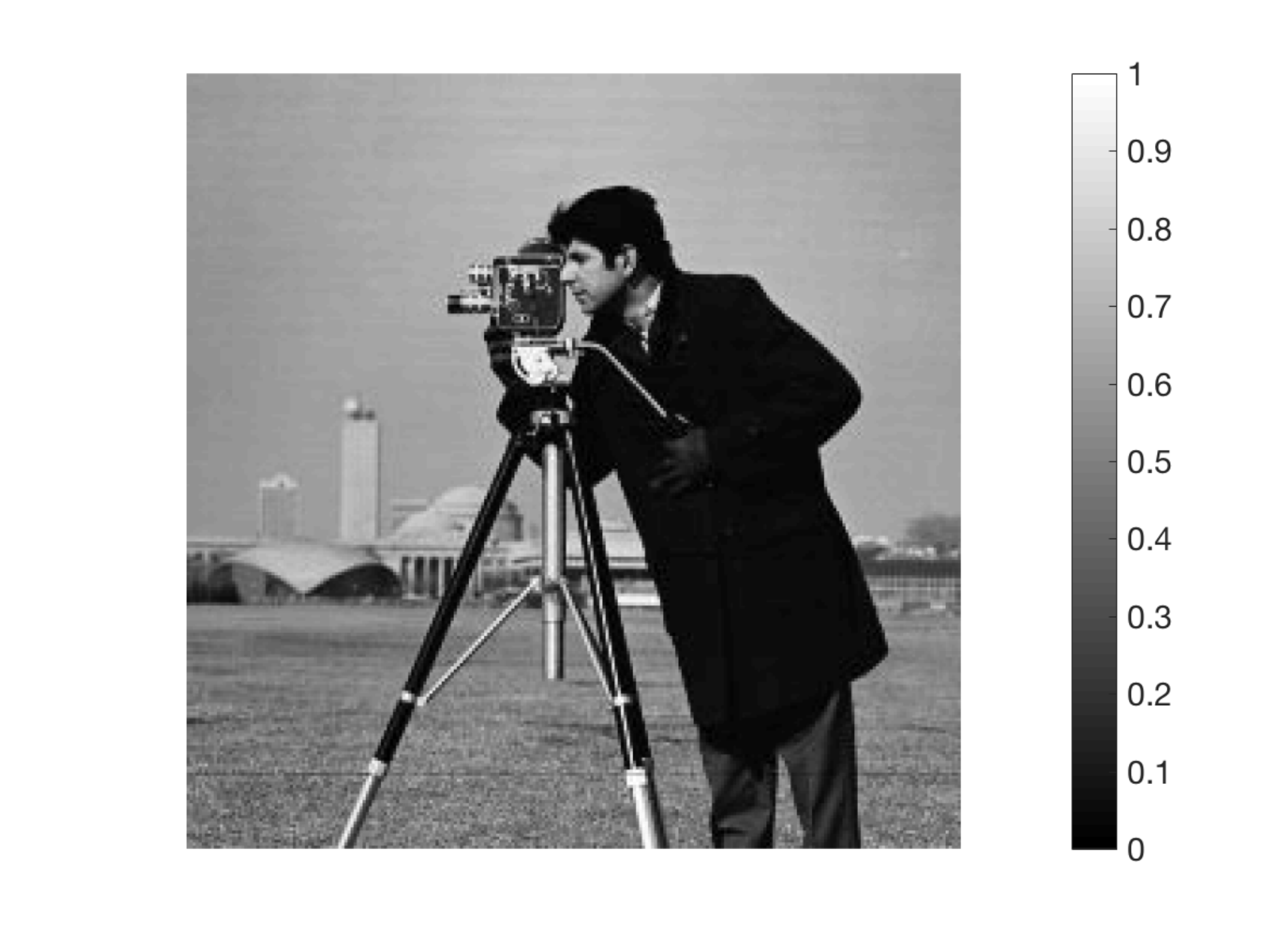}
\includegraphics[width=0.49\textwidth]{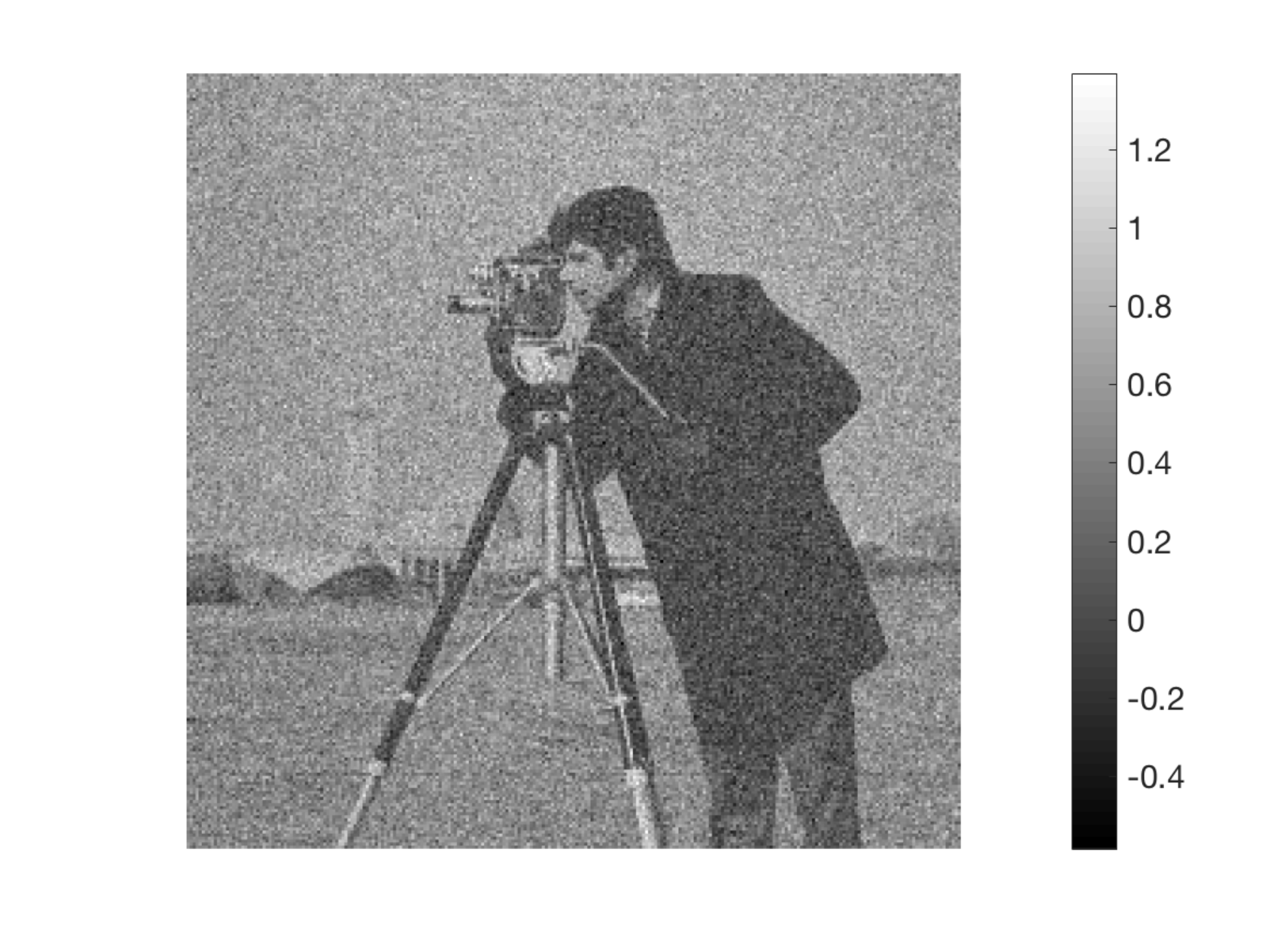}
\includegraphics[width=0.49\textwidth]{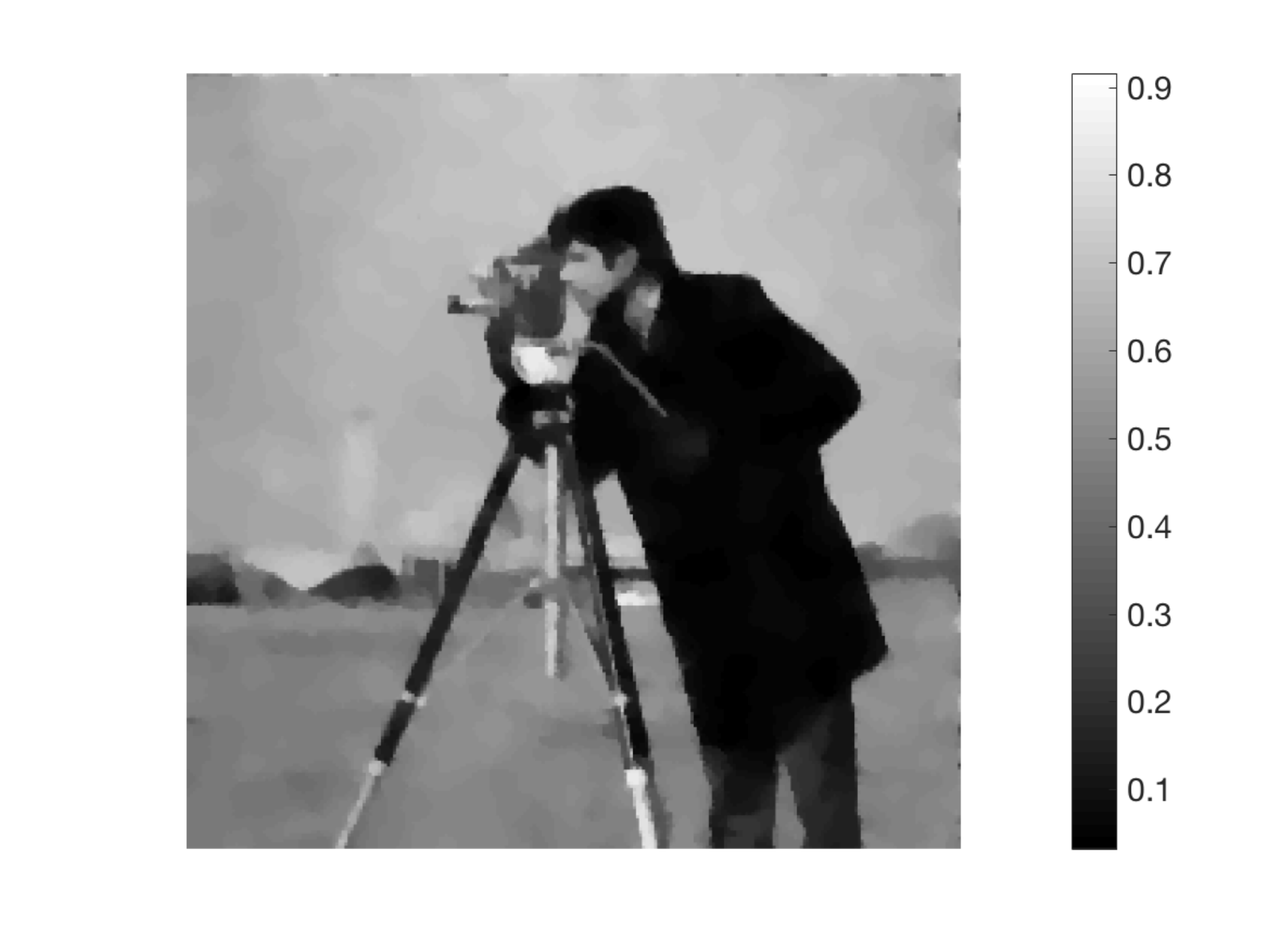}
\includegraphics[width=0.49\textwidth]{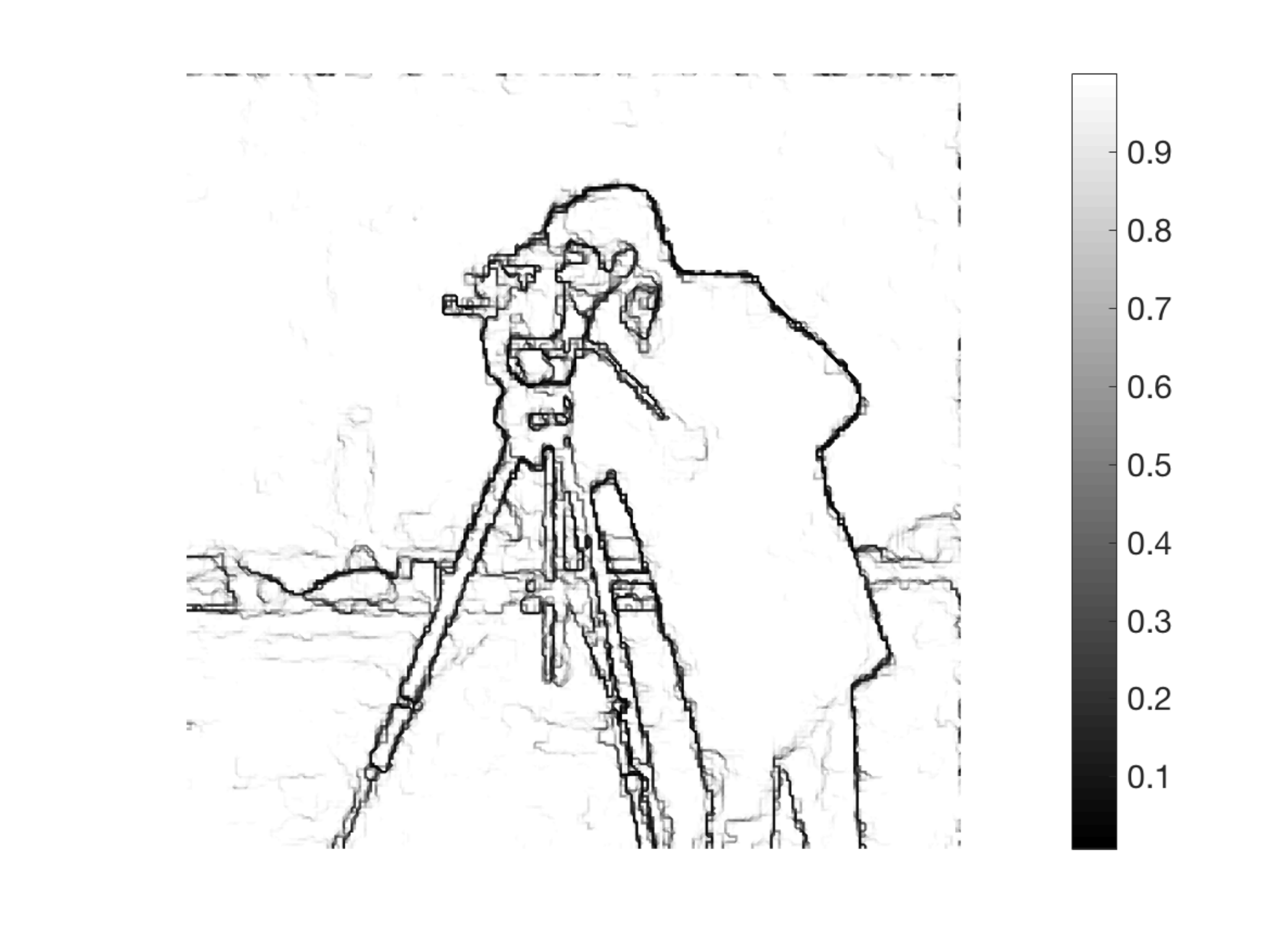}
\includegraphics[width=0.49\textwidth]{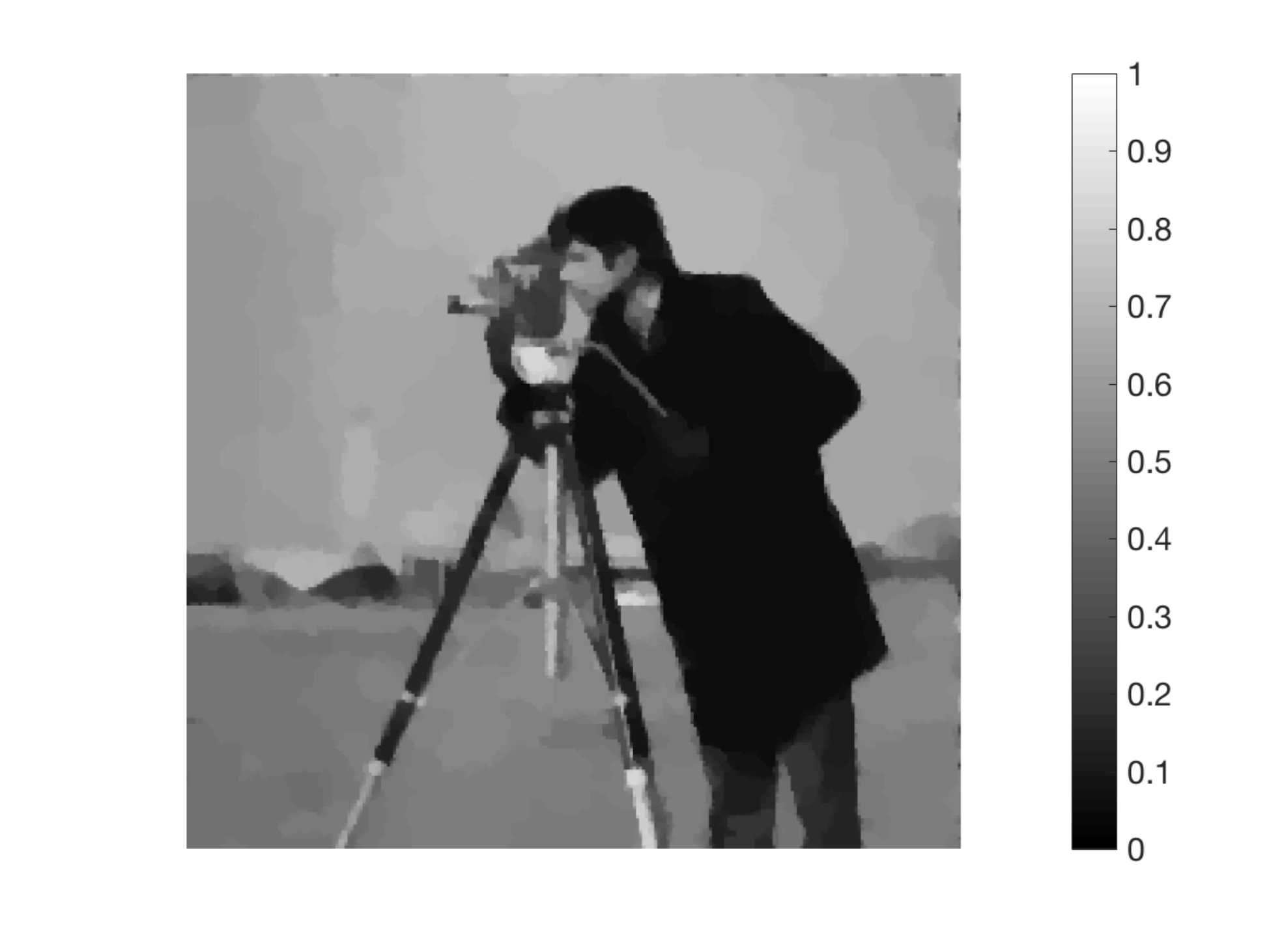}
\includegraphics[width=0.49\textwidth]{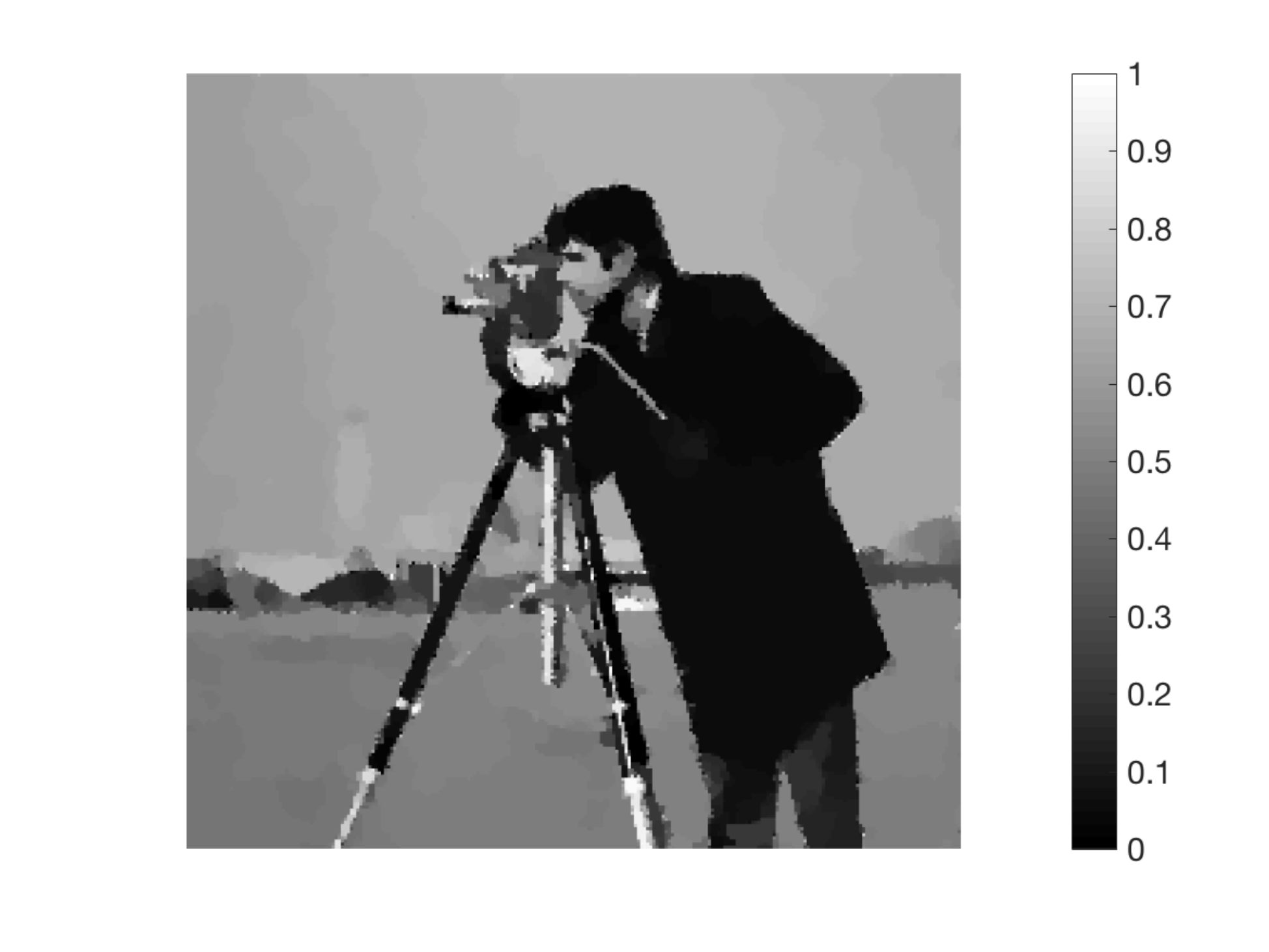}
\caption{\label{f:ex_3_sigma15}
Example 3 ($\sigma=0.15$). Top row (from left to right): original and noisy images, respectively. Middle row: $u_{\rm tv}$ (left) from Step~\ref{step1} in Algorithm~\ref{Algorithm} and the corresponding $s$  (right) from Step~\ref{step3}.
Bottom row: reconstruction using  total variation with optimized $\zeta>0$ (left) and our approach (right), respectively.}
\end{figure}

\section*{Acknowledgement}

We thank Gunay Dogan for point us to the edge error indicator and Pablo Raul Stinga
for sharing the reference \cite{ANekvinda_1993a}.

\bibliography{refs}
\bibliographystyle{plain}
\end{document}